\documentclass[reqno]{amsart}

\usepackage[utf8]{inputenc}
\usepackage[T1]{fontenc}
\usepackage[american]{babel,isodate}
\cleanlookdateon
\usepackage{lmodern}
\usepackage{xcolor}
\usepackage{pdfpages}
\usepackage{enumerate}
\usepackage{listings}
\usepackage{amsmath,amssymb,xcolor}
\definecolor{lightgrey}{rgb}{0.9,0.9,0.9}
\usepackage{physics}
\usepackage{ stmaryrd }
\usepackage{tikz}
\usepackage{tikz-cd}
\usetikzlibrary{arrows}
\usetikzlibrary{calc}
\usepackage{cancel}
\usetikzlibrary{babel}
\usepackage{lscape}
\usepackage{url}

\usepackage{hyperref}
\hypersetup{linktocpage,colorlinks=true,citecolor=purple,linkcolor=blue,urlcolor=blue,filecolor=black}

\usepackage{float}
\usepackage{graphicx, import}
\usepackage{mathtools}
\usepackage{dsfont}
\usepackage{pinlabel}
\usepackage{yhmath}

\usepackage{xargs}                      
\usepackage{xcolor} 
\definecolor{OliveGreen}{rgb}{0,0.6,0}
 
\usepackage[colorinlistoftodos, prependcaption,textsize=tiny]{todonotes}
\newcommandx{\unsure}[2][1=]{\todo[linecolor=red,backgroundcolor=red!25,bordercolor=red,#1]{#2}}
\newcommandx{\change}[2][1=]{\todo[linecolor=blue,backgroundcolor=blue!25,bordercolor=blue,#1]{#2}}
\newcommandx{\info}[2][1=]{\todo[linecolor=OliveGreen,backgroundcolor=OliveGreen!25,bordercolor=OliveGreen,#1]{#2}}

\newcommand{\centre}[1]{\begin{array}{c} #1 \end{array}}

\usepackage{stackengine}

\usepackage{amsmath,amsthm,amssymb, amsfonts, amscd}
\usepackage{bbm}
\usepackage{rotating,subcaption}
\usepackage{url}
\usepackage{graphicx,bm}
\usepackage{mathtools}
\usepackage[mathscr]{euscript}
\allowdisplaybreaks
\usepackage{lipsum}
\usepackage{bm}
\usepackage[capitalize]{cleveref} 
\usepackage{mathrsfs}

\usepackage{xpatch}
\makeatletter
\AtBeginDocument{\xpatchcmd{\@thm}{\thm@headpunct{.}}{\thm@headpunct{}}{}{}}
\makeatother

\newtheorem{theorem}{Theorem}[section]
\newtheorem{proposition}[theorem]{Proposition}

\newtheorem{lemma}[theorem]{Lemma}
\newtheorem{corollary}[theorem]{Corollary}
\newtheorem{conjecture}[theorem]{Conjecture}

\theoremstyle{definition}
\newtheorem{definition}[theorem]{Definition}
\newtheorem{example}[theorem]{Example}

\theoremstyle{remark}

\newtheorem{remark}[theorem]{Remark}
\newtheorem{notation}[theorem]{Notation}

\numberwithin{equation}{section}

\newcommand{\bk}{\Bbbk}

\usepackage{accents}
\newcommand{\uhat}{\underaccent{\check}}

\makeatletter
\newcommand{\cupr@tip}{\text{\raisebox{-0.1ex}{$\m@th\hat{}$}}}
\newcommand{\cupr}{\mathbin{\cup\cupr@}}

\newcommand{\cupr@}{%
  \mathchoice
  {\mkern-1.35mu\cupr@tip}
  {\mkern-1.35mu\cupr@tip}
  {\mkern-1.55mu\cupr@tip}
  {\mkern-1.875mu\cupr@tip}
}
\makeatother

\makeatletter
\newcommand{\capr@tip}{\text{\raisebox{0.47ex}{$\m@th\uhat{}$}}}
\newcommand{\capr}{\mathbin{\capr@\cap}}

\newcommand{\capr@}{%
  \mathchoice
  {\mkern11.6mu\capr@tip\mkern-11.6mu}
  {\mkern11.4mu\capr@tip\mkern-11.4mu}
  {\mkern11.1mu\capr@tip\mkern-11.1mu}
  {\mkern10.2mu\capr@tip\mkern-10.2mu}
}
\makeatother

\makeatletter
\newcommand{\capl@tip}{\text{\raisebox{0.47ex}{$\m@th\uhat{}$}}}
\newcommand{\capl}{\mathbin{\capl@\cap}}

\newcommand{\capl@}{%
  \mathchoice
  {\mkern2.1mu\capl@tip\mkern-2.1mu}
  {\mkern2.1mu\capl@tip\mkern-2.1mu}
  {\mkern2.3mu\capl@tip\mkern-2.3mu}
  {\mkern2.1mu\capl@tip\mkern-2.1mu}
}
\makeatother

\makeatletter
\newcommand{\cupl@tip}{\text{\raisebox{-0.1ex}{$\m@th\hat{}$}}}
\newcommand{\cupl}{\mathbin{\cupl@\cup}}

\newcommand{\cupl@}{%
  \mathchoice
  {\mkern1.35mu\cupl@tip\mkern-1.35mu}
  {\mkern1.35mu\cupl@tip\mkern-1.35mu}
  {\mkern1.55mu\cupl@tip\mkern-1.55mu}
  {\mkern1.875mu\cupl@tip\mkern-1.875mu}
}
\makeatother

\DeclareFontFamily{U}{mathx}{}
\DeclareFontShape{U}{mathx}{m}{n}{ <-> mathx10 }{}
\DeclareSymbolFont{mathx}{U}{mathx}{m}{n}
\DeclareFontSubstitution{U}{mathx}{m}{n}
\DeclareMathAccent{\widecheck}{0}{mathx}{"71}

\begin{document}


\title{Perturbed Alexander Invariants via Quantum Clusters Algebras}

\date{\today}

\author{Boudewijn Bosch}
\address{Bernouilli Institute, University of Groningen, Nijenborgh 9, 9747 AG, Groningen, The Netherlands}
\email{\href{mailto:b.j.bosch@rug.nl}{b.j.bosch@rug.nl}}




\begin{abstract}

We derive a perturbative expansion of knot invariants using quantum cluster algebras. In a recent series of papers, it has been shown that the universal $R$-matrix of $U_q(\mathfrak{g})$, with $\mathfrak{g}$ a Lie algebra of type ADE, can be interpreted as a cluster transformation. For $\mathfrak{g} = \mathfrak{sl}_2$, we show that this interpretation yields a perturbed $R$-matrix. We achieve this by expressing certain representations of quantum cluster algebras in terms of generators of the Weyl--Heisenberg algebra. By inserting an auxiliary parameter $\epsilon$, we derive a perturbed $R$-matrix. With it, we can compute a knot invariant whose zeroth-order term in $\epsilon$ equals $ \Delta_K (T)^{-1}$, the reciprocal of the Alexander polynomial, while higher-order terms in $\epsilon$ give rise to so-called perturbed Alexander invariants. These invariants also appear in perturbative series known by several names in the literature: the Melvin--Morton--Rozansky (MMR) expansion, the large-color expansion, the rational expansion, or the loop expansion. We implement our approach formally in \textit{Mathematica} and compute explicit examples. Moreover, we argue that the methods developed in this paper should extend beyond the case $\mathfrak g=\mathfrak{sl}_2$.

\end{abstract}



\maketitle

\setcounter{tocdepth}{2}
\tableofcontents


\section{Introduction}

\noindent
Cluster algebras are commutative rings endowed with an additional combinatorial structure: a set of generators (cluster variables) grouped into overlapping subsets (called \textit{clusters}). For any cluster $\mathbf{x}$ and any element $x \in \mathbf{x}$, there is another cluster obtained from $\mathbf{x}$ by transforming $x$ into an element $x'$. This transformation is called a \emph{mutation}. The data specifying a cluster algebra is called a \emph{cluster seed}, consisting of a set of cluster variables together with an associated exchange matrix (or, equivalently, a quiver). The exchange matrix defines mutation rules. Cluster algebras arise in areas such as Lie theory, Teichmüller theory, Poisson geometry, and quantum groups; see, e.g., \cite{fomin2002cluster,berenstein2003cluster}. They have also been used in knot theory in  \cite{NagaiTerashima2020}, where cluster algebras have been connected to the Alexander polynomial.

In~\cite{BerensteinZelevinsky2005}, the notion of a \textit{quantum cluster algebra} was introduced as a noncommutative deformation of a cluster algebra, in which the cluster variables generate a quantum torus. A \textit{quantum cluster seed} consists of quantum cluster variables together with an exchange matrix which defines their commutation relations. Mutations are defined as automorphisms of the quantum torus.

Quantum groups, by contrast, were constructed much earlier. In \cite{Drinfeld1985, Jimbo1985}, Drinfeld and Jimbo defined quantum groups as deformations of universal enveloping algebras of semisimple Lie algebras. They are endowed with the structure of a Hopf algebra and carry a universal $R$-matrix satisfying the Yang--Baxter equation. Quantum groups have proven an enormously valuable tool in knot theory after the development of quantum invariants of knots and links, originally introduced by Reshetikhin and Turaev \cite{ReshetikhinTuraev1990}. However, because they arise within an algebraic framework, the resulting invariants often lack a geometric interpretation.

Recently, Schrader and Shapiro \cite{Schrader2019} realized the quantum group $U_q(\mathfrak{sl}_n)$  within the framework of quantum cluster algebras by considering the quantum cluster structure associated to moduli space of $\mathrm{PGL}_n$-local systems on marked surfaces introduced in \cite{FockGoncharov2006}. This recovers, in particular, Faddeev’s picture of $U_q(\mathfrak{sl}_2)$ described in \cite{faddeev1999modulardoublequantumgroup}. The braid group action and, hence, the $R$-matrix admit a description via cluster transformations (mutations) associated to (half) Dehn twists on the two-punctured disk.

Goncharov and Shen extended the construction of Schrader and Shapiro in \cite{goncharov2019quantum}. In this work, they showed that the quantum group $U_q(\mathfrak g)$ for $\mathfrak{g}$ of type ADE likewise admits an interpretation within the theory of cluster algebras. By considering a surface $S$ and a semisimple adjoint algebraic group split $G$, a moduli space $\mathscr{P}_{G,S}$ is defined, from which quantum groups can be recovered. Especially noteworthy is that characteristic properties of quantum groups, such as the Hopf-algebra structure, Chevalley generators, Weyl group, and the $R$-matrix, arise from the geometry of $\mathscr{P}_{G,S}$.

Quantum groups typically lead to knot invariants by considering $n$-dimensional representations, leading to the \textit{colored Jones polynomials} $J_{\mathcal{K}}^{n}$. This is, however, not the only way quantum groups give rise to knot invariants. In \cite{bar2021perturbed}, an auxiliary parameter $\epsilon$ is inserted into the algebra relations of a Drinfeld double, leading to a perturbative expansion of the associated universal invariant of knots. This perturbative series was subsequently identified in \cite{bosch2025largecolorexpansionderiveduniversal} with a well-known expansion that appears under several names in the literature: the \emph{Melvin--Morton--Rozansky (MMR) expansion} \cite{MelvinMorton1995,BarNatanGaroufalidis1996,rozansky1998universalr}, the \emph{large-color expansion} \cite{park2020large}, the \emph{rational expansion} \cite{rozansky1998universalr}, or the \emph{loop expansion} \cite{Kricker2000}. More concretely, there exist symmetric Laurent polynomials
  \[
    P^{\mathcal{K}}_i \in \mathbb{Z}[t,t^{-1}], \qquad i \geq 0
    \]
    such that the colored Jones polynomial admits an expansion
    \[
      J_{\mathcal{K}}^{n}(q)=\sum_{i=0}^{\infty} \frac{P^{\mathcal{K}}_{i}\left(q^{n}\right)}{\Delta_{\mathcal{K}}^{2 i+1}\left(q^{n}\right)}(q-1)^{i} \in \mathbb{Q} \llbracket q-1 \rrbracket,
      \]
      where $P_{\mathcal{K}}^0(t) = 1$.

In \cite{bar2021perturbed}, the main results are proved using a Hopf-algebraic interpretation of Seifert surfaces. In subsequent work \cite{barnatan2024perturbedalexanderinvariant}, the first-order perturbed Alexander invariant is defined using a quite different method, where a Gaussian interpretation of the Alexander polynomial had been exploited. This Gaussian picture (arising from the universal invariant) has been made explicit in \cite{bosch2025tensorsgaussiansalexanderpolynomial} in an algebraic context. Concretely, for a given knot, one considers the matrix $A$, corresponding to a specifically chosen presentation matrix of the Alexander module, whose determinant is the Alexander polynomial. The perturbed Alexander invariant is obtained from a quadratic expression in the entries of $A^{-1}$. The similarity between the formulas for the perturbed Alexander invariant and the classical Alexander polynomial suggests an underlying topological explanation, which is still missing.

In this paper, we argue that the cluster-algebraic interpretation of quantum groups provides a natural starting point for considering perturbative expansions of knot invariants.  Because a class of cluster algebras has geometric origins in gluing equations for a hyperbolic ideal tetrahedron to a triangulation of a punctured disk \cite{Fomin2008,hikami2014braiding}, this approach suggests a route toward a topological interpretation of perturbed Alexander invariants. We show that the Alexander polynomial can be recovered at the classical level (Theorem \ref{thm:clusterburau}), and perturbed Alexander invariants arise after quantization of cluster algebras (Theorem \ref{thm:pertrmatrix}). 

This connection follows from the representation theory of quantum cluster algebras. Quantum cluster algebras admit representations that can be regarded as a sophisticated analog of the Schr\"odinger representation. To allow for perturbative expansions, we insert an auxiliary parameter $\epsilon$. This allows for an $\epsilon$-expansion of the $R$-matrix, giving rise to a “perturbed $R$-matrix”
\[
  R = R_0(1 + \epsilon f(T,x,y)) + O(\epsilon^2).
\]
Notably, the insertion of $\epsilon$ is chosen in such a way that the zeroth-order term of the $R$-matrix gives rise to the Alexander polynomial when applied to knots. Informally, we state the following.

\begin{theorem}
For a distinguished choice of cluster variables of an amalgamated cluster algebra of type $A_1$, the mutation sequence associated to a triangulation of the complement of a knot $\mathcal{K}$ recovers the Alexander polynomial of $\mathcal{K}$. After quantization and $\epsilon$-expansion, it yields perturbed Alexander invariants.
  \end{theorem}
Moreover, we aim to provide a first step toward a universal strategy for deriving perturbed knot invariants. Since quantum groups $U_q(\mathfrak g)$ with $\mathfrak g$ of type ADE have an interpretation in the theory of cluster algebras, the methods developed in this paper extend beyond the case $\mathfrak g=\mathfrak{sl}_2$. The main challenge then resides in choosing an appropriate cluster seed and a suitable basis for the associated representation. We give an example of this approach for the quantum cluster algebra of type $A_2$ that is related to $U_q(\mathfrak{sl}_3)$.

An advantage of approaching perturbative series of knot invariants via cluster algebras is its geometric interpretation. We discuss this in more detail in Section~\ref{sec:discussionthird}.

\section{Organization of the paper}
\noindent
In Section~\ref{sec:schrrep} we recall the Heisenberg group, its algebra, and the Schrödinger representation. In Section~\ref{sec:heisenbergalg} we review some identities and notational conventions related to the Weyl--Heisenberg algebra. We also discuss a normal-ordering formalism. Section~\ref{sec:clusteralgebras} provides background on cluster algebras and explains how braid group actions, and in particular the $R$-matrix, arise from cluster mutations. Section~\ref{sec:quantization} discusses the quantization of cluster algebras and its representation theory. In Section~\ref{sec:perturbed} we derive the $\epsilon$-expanded $R$-matrix up to the first order in $\epsilon$. Section~\ref{sec:towsl3} discusses the first steps toward deriving a perturbed $R$-matrix in the $\mathfrak{sl}_3$ case. In Section~\ref{sec:discussionthird}, we interpret the geometric meaning of our approach. Lastly, Appendix~\ref{sec:mathematica} provides a \textit{Mathematica} implementation that derives an expansion up to second order in $\epsilon$. This also includes a sanity check confirming that the $\epsilon$-expanded $R$-matrix satisfies its required conditions.

\subsection*{Acknowledgments} I would like to thank Roland van der Veen for his many helpful discussions and suggestions regarding the content of this paper, as well as Ivan Ip for directing me to the notion of the Schrödinger representation.

\section{The Schr\"odinger Representation}
\label{sec:schrrep}
\noindent
In this section, we review the Heisenberg group and its associated algebra, followed by the Schrödinger representation, based on \cite{lion1980weil}.

\subsection{Symplectic Vector Spaces}
Let $(V,\omega)$ be a $2n$-dimensional symplectic vector space over $\mathbb{R}$. We choose a basis
\[
(p_1,p_2,\dots,p_n,x_1,x_2,\dots,x_n)
\]
 with relations
\begin{align*}
  &\omega(p_i,p_j) = 0, \quad \omega(x_i,x_j) = 0, \\
  &\omega(p_i,x_j) = \delta_{ij}, \quad \omega(x_i,p_j) = - \delta_{ij}.
\end{align*}
This basis is called a \textit{symplectic basis} of $(V,\omega)$.

To $(V,\omega)$ we associate the Heisenberg Lie algebra denoted by $\mathfrak{h}_n$. It is a Lie algebra over $\mathbb{R}$, regarded as $V \oplus \left\langle z \right\rangle$ with relations
\begin{align*}
  &[x,y] = \omega(x,y) z, \quad \text{if } x,y \in V \\
  &[V,z] = 0, \text{ i.e. $\mathbb{R}z$ is the center of } \mathfrak{h}_{n}.
\end{align*}
Let $ L $ be a subspace of $V$. We define its orthogonal complement with respect to $ \omega $ as $ L^\perp $, given by
\[
L^\perp := \{ x \in V \mid \omega(x, y) = 0 \text{ for all } y \in L \}.
\]

A subspace $ \ell $ of $ V $ is called \textit{Lagrangian} if it satisfies $ \ell = \ell^\perp $. This means that $ \omega(x, y) = 0 $ for all $ x, y \in \ell $, implying that $ \ell $ is totally isotropic with respect to the bilinear form $ \omega $. Furthermore, if $ x $ satisfies $ \omega(x, \ell) = 0 $, then $ x \in \ell $. This implies that $ \ell $ is a maximal totally isotropic subspace of $ (V, \omega) $. Given a Lagrangian subspace of $(V,\omega)$, there exists a Lagrangian subspace $\ell'$ such that $\ell \oplus \ell' = V$. This choice of Lagrangian subspaces will be referred to as a \textit{polarization}. In addition, the \textit{symplectic group} is defined as

\[
\mathrm{Sp}(V, \omega) = \{ g \in \text{GL}(V) \mid \omega(g x, g y) = \omega(x, y) \text{ for all } x, y \in V \}.
\]

\subsection{The Heisenberg Group and the Schr\"odinger Representation}
\label{sec:heisenberggroup}

The Heisenberg group $ H_n $ is defined as the Lie group that arises from exponentiating the Heisenberg algebra $ \mathfrak{h}_n $, i.e.,
\[
H_n = \exp(\mathfrak{h}_n).
\]
The multiplication law can be deduced via the BCH formula:
\begin{equation}
\label{eq:multiplication}
  \exp(v + t z)  \cdot \exp(v' + t' z) = \exp(v + v' + (t + t' + \omega(v,v')/2) z)
\end{equation}
  with $v,v' \in V$, $t,t' \in \mathbb{R}$. The Euclidean measure $ \mathrm{d}v $ on the vector space $ V $ induces a Haar measure on the Heisenberg group $ H_n $, which is given by $ \mathrm{d}h = \mathrm{d}v  \mathrm{d}t $.

The group $ G = \mathrm{Sp}(V, \omega) $ acts on the Heisenberg group $ H_n $ via the action
\[
g \cdot (\exp(v + t z)) = \exp(gv + t z).
\]
Note that $ G $ preserves the center of $ H_n $.

Additionally, the group $L = \exp( \ell + \mathbb{R}z)$ is an abelian subgroup of $H_n$. Consider the function $f \colon H_n \to U(1)$ defined by $\exp(v + t z) \mapsto e^{2 \pi i t}$. It is clear that $f(h_1 h_2) = f(h_1) f(h_2)$ for $h_1,h_2 \in L$. Thus, $f|_L$ is a linear character of $L$. The Euclidean measure $\mathrm{d} v'$ on $\ell'$ gives rise to a positive measure $\mathrm{d}\overline{h}$ on the space $H_n/L$.

For a choice of Lagrangian subspace $ \ell $, we define the \textit{Schrödinger representation} $ W(\ell)$ of the Heisenberg group $ H_n $ as
\[
W(\ell) = \mathrm{Ind}_L^{H_n} (f|_{L}).
\]
More concretely, the representation acts on a Hilbert space $ \mathcal{H}(\ell) $, which is the completion of the space of continuous functions $ \vartheta $ on $ H_n $ satisfying the following conditions:
\begin{enumerate}[(i)]
\item\label{item:4}
 Covariance under $ L $:
   \[
   \vartheta(ab) = f(b)^{-1} \vartheta(a), \quad \text{for all } a \in H_n,  b \in L.
   \]
\item Square integration:
   The function $ n \mapsto |\vartheta(n)| $ is square-integrable on the quotient $ H_n/L $ with respect to the invariant measure $ \mathrm{d} \overline{h} $, that is
   \[
   \int_{H_n/L} | \vartheta(\overline{h})|^2 \mathrm{d} \overline{h} < \infty.
   \]
\end{enumerate}
The group $ H_n $ acts on this Hilbert space via left translations, given by
\[
(a \cdot \vartheta)(h) = \vartheta(a^{-1} h).
\]
By definition, any $\vartheta \in \mathcal{H}(\ell)$ is completely determined by its restriction to $\exp(\ell')$. Hence, there exists an isomorphism $\mathcal{H}(\ell) \cong L^2(\ell')$.

Let us now consider the action of $H_n$ on $\mathcal{H}(\ell)$ more explicitly. Using the multiplication law as in \eqref{eq:multiplication}, we note that
\begin{align}
  \label{eq:schrodinger}
 & \exp(x)\cdot \vartheta(\exp(y)) = e^{2 \pi i \omega(x,y)}\vartheta(\exp(y)) &&\qquad x \in \ell, y \in \ell', \\
  \nonumber &\exp(y_0) \cdot \vartheta(\exp(y)) = \vartheta(\exp(y-y_0)) &&\qquad y,y_0 \in \ell', \\
  \nonumber &\exp(t z) = e^{2 \pi i t} \mathrm{Id}.
\end{align}
The action described above is a realization of the so-called Schrödinger representation of the Heisenberg group. A fundamental result, due to Stone and Von Neumann, declares that this representation is the only irreducible unitary representation of the Heisenberg group satisfying the central character condition on $\exp(tz)$. We now state this result formally.
\begin{theorem}[Stone--Von Neumann theorem]
  The following two statements hold:
  \begin{enumerate}[\normalfont (i)]
    \item\label{item:3} $W(\ell)$ is an irreducible representation of $H_{n}$.
          \item Every unitary representation $T$ of $H_n$ on a Hilbert space $\mathcal{H}$ such that $$T(\exp(t z)) = e^{2 \pi i t} \mathrm{Id}_{\mathcal{H}}$$ is unitarily equivalent to a direct sum of copies of $W(\ell)$.
  \end{enumerate}
\end{theorem}
Let us consider the irreducible unitary representation $(W, \mathcal{H}(\ell))$ of $H_n$ corresponding to a fixed polarization $\ell$, as constructed above. Given any $g \in \mathrm{Sp}(2n, \mathbb{R})$, we can define a new representation $(W^g, \mathcal{H})$ by
\[
W^g(h) := W(g \cdot h), \quad \text{for all } h \in H_n.
\]
By the Stone--von Neumann Theorem, $W^g$ is unitarily equivalent to $W$. Hence, there exists a unitary operator $R(g)$ on $\mathcal{H}$ such that
\[
R(g) W(h) R(g)^{-1} = W(g \cdot h), \quad \text{for all } h \in H_n.
\]

The map $g \mapsto R(g)$ thus yields a projective representation of the symplectic group $\mathrm{Sp}(2n, \mathbb{R})$ on the Hilbert space $\mathcal{H}$. Note that $R(g)$ is defined up to a scalar. This projective representation is called the \textit{Weil representation} (also known as the metaplectic representation). To obtain an actual representation, rather than a merely projective one, one must pass to a double cover of $\mathrm{Sp}(2n, \mathbb{R})$, called the metaplectic group $\mathrm{Mp}(2n, \mathbb{R})$. For details of this construction, see \cite{lion1980weil}.

\section{The Weyl--Heisenberg algebra}
\label{sec:heisenbergalg}
\noindent
Because the Heisenberg algebra plays a central role in the Schrödinger representation, we introduce several identities and notational conventions to make later computations more manageable. We omit almost all proofs in this section. For details, see \cite{bar2021perturbed,bosch2025tensorsgaussiansalexanderpolynomial}.

We start by introducing notation we use throughout this paper. We assume that the reader is familiar with tensor products, but may not have seen them in this particular notation.

Let $\bk$ be a commutative ring and let $A$ be a unital associative $\bk$-algebra. For any finite set $I$, we write
\[
A^{\otimes I} := \bigotimes_{i\in I} A,
\]
Note that any bijection $\sigma\colon I\xrightarrow{\sim} J$ induces an algebra isomorphism
\[
P_\sigma\colon A^{\otimes I}\xrightarrow{\ \cong\ }A^{\otimes J},
\qquad
P_\sigma\left(\bigotimes_{i\in I} a_i\right)
=
\bigotimes_{j\in J} a_{\sigma^{-1}(j)}.
\]

\begin{definition}\label{def:leg}
For $i\in I$, the \emph{leg embedding}
\[
\iota_i^I\colon A\longrightarrow A^{\otimes I}
\]
is defined as the unital algebra map which sends $a\mapsto \bigotimes_{j\in I} x_j$ with $x_i=a$ and $x_{j}=1$ for $j \neq i$. \end{definition}
\begin{notation}
We denote
\[
a_i := \iota_i^I(a) \in A^{\otimes I}.
\]
Moreover, for a finite subset $S\subset I$ and $\{a_s\in A\}_{s\in S}$ we write $\prod_{s\in S} a_s$ for the product in $A^{\otimes I}$.
\end{notation}

Let $n \in \mathbb{N}$. Moreover, let $I$ be an ordered set with $n$ elements with ordering $i_{1} < i_{2}< \dots < i_{n}$.
\begin{definition}
  \label{def:heisenberg}
The \emph{$n$-th Weyl--Heisenberg algebra} $A_n(\bk)$ over $\bk$ is the unital associative $\bk$-algebra generated by $2n$ elements $\{\mathbf{x}_i,\mathbf{p}_i\}_{i \in I}$ subject to the relations
\begin{align*}
\mathbf{p}_i \mathbf{p}_j - \mathbf{p}_j \mathbf{p}_i = 0, \qquad \mathbf{x}_i \mathbf{x}_j - \mathbf{x}_j \mathbf{x}_i = 0, \qquad \mathbf{p}_i \mathbf{x}_j - \mathbf{x}_j \mathbf{p}_i = \delta_{ij}
\end{align*}
for all $i,j \in I$, where $\delta_{ij}$ is the Kronecker delta.
\end{definition}
\begin{remark}
The relations in the definition of the Weyl--Heisenberg algebra can also be expressed using the commutation bracket. For any two elements $\mathbf{x}$ and $\mathbf{y}$, their commutation bracket is defined as $[\mathbf{x}, \mathbf{y}] := \mathbf{x}\mathbf{y} - \mathbf{y}\mathbf{x}$. With this notation, the defining relations of the $n$-th Weyl--Heisenberg algebra can be written as
\begin{align*}
[\mathbf{p}_i, \mathbf{p}_j] = 0, \qquad [\mathbf{x}_i, \mathbf{x}_j] = 0, \qquad [\mathbf{p}_i, \mathbf{x}_j] = \delta_{ij}.
\end{align*}
\end{remark}

The Weyl--Heisenberg algebra $A_n(\bk)$ is closely related to the Heisenberg Lie algebra. In terms of the symplectic basis,  the Heisenberg Lie algebra in $n$ dimensions is defined as the Lie algebra $\mathfrak{h}_n$ with generators $\{\mathbf{p}_{k}\}_{k=1}^n \cup \{\mathbf{x}_{k}\}_{k=1}^n \cup \{\mathbf{z}\}$ and relations
\[
[\mathbf{p}_i, \mathbf{p}_j] = 0, \qquad [\mathbf{x}_i, \mathbf{x}_j] = 0, \qquad [\mathbf{p}_i, \mathbf{x}_j] = \delta_{ij}\mathbf{z}, \qquad \text{for all $1 \leq i, j \leq n$,}. 
\]
Here, $\mathbf{z}$ is a central element. We write $U(\mathfrak{h}_n)$ for the universal enveloping algebra of $\mathfrak{h}_n$, which is a noncommutative $\bk$-algebra generated by the elements of $\mathfrak{h}_n$. The $n$-th Weyl--Heisenberg algebra $A_n(\bk)$ is isomorphic to the quotient of $U(\mathfrak{h}_n)$ by the two-sided ideal generated by $\mathbf{z} - 1$, i.e., $A_n(\bk) \cong U(\mathfrak{h}_n) / (\mathbf{z} - 1)$.

The relation to the universal enveloping algebra allows us to invoke the Poincar\'e--Birkhoff--Witt theorem. We hereby endow $A_n(\bk)$ with an ordering
\begin{equation}
\label{eq:ordering}
  \mathbf{x}_{i_{1}} < \dots < \mathbf{x}_{i_{n}} < \mathbf{p}_{i_{1}} < \dots < \mathbf{p}_{i_{n}}.
\end{equation}
We call any ordering of this form a \textit{normal ordering}. Applying the Poincaré--Birkhoff--Witt theorem to $A_n(\bk)$, we deduce that the normally ordered monomials
\[
  x^{\alpha}p^{\beta}
  := x_{i_1}^{\alpha_{1}}\cdots x_{i_n}^{\alpha_{n}}
     p_{i_1}^{\beta_{1}}\cdots p_{i_n}^{\beta_{n}},
  \qquad (\alpha,\beta) \in \mathbb{Z}_{\geq 0}^{n} \times \mathbb{Z}_{\geq 0}^{n},
\]
form a vector space basis of $A_n(\bk)$. From this, we conclude that there exists a vector space isomorphism
\[
  \mathcal{N}_{I}\colon  \bigotimes_{i \in I}\bk[x_i,p_i] \xrightarrow{\sim} A_n(\bk),
  \text{ linearly defined by } \prod_{i \in I} x_i^{\alpha_{i}}p_i^{\beta_{i}} \mapsto (\prod_{i \in I} \mathbf{x}_i^{\alpha_{i}})(\prod_{j \in I}\mathbf{p}_j^{\beta_{j}}). \footnote{We implicitly used the isomorphism $\bk[x_1] \otimes \bk[x_2]\otimes \dots \otimes \bk[x_n] \cong \bk[x_1,x_2,\dots,x_n]$, which we shall do throughout this work.}
\]
   We omit the subscript $I$ in $\mathcal{N}_{I}$ when it is clear from context. To ease some of the notation, set $\bk[z_I]:=\bk[x_I,p_I]:=\bigotimes_{i \in I}\bk[x_i,p_i]$ with $z_i=(x_i,p_i)$.
\begin{example}
Let $I = \{1,2\}$ and $f = \mathbf{p}_1^{2} \mathbf{x}_1\mathbf{x}_{2} \in A_{2}(k)$. Then we find
\[
f = \mathbf{p}_1^{2}  \mathbf{x}_1 \mathbf{x}_2 = 2\mathbf{x}_2\mathbf{p}_1+\mathbf{x}_1\mathbf{x}_2\mathbf{p}_1^2 = \mathcal{N}( 2 x_{2}p_{1} + x_{1} x_{2} p_{1}^{2}  ).
\]
\end{example}
The space $\bk[z_I]$ can be endowed with the structure of a commutative algebra by viewing it as a polynomial ring. In contrast, $A_n(\bk)$ is a noncommutative algebra. Therefore, the map $\mathcal{N} \colon \bk[z_I] \to A_n(\bk)$ is not an isomorphism of algebras. We will use this distinction throughout this paper.

As we will see in the upcoming discussion, we apply exponential and logarithmic maps repeatedly to elements of the Weyl--Heisenberg algebra. Given the significance of these operations in our approach, we formally define the deformation of the underlying algebraic structure that makes them well-defined.

\begin{definition}
Let $A$ be an algebra over a field $\bk$. The \emph{trivial deformation} of $A$ into an algebra over $\bk\llbracket h \rrbracket$, is the $\bk \llbracket h \rrbracket$-algebra that arises by extending $\bk$ to $\bk \llbracket h \rrbracket$. In other words: it is the tensor product $\bk \llbracket h \rrbracket \otimes_{\bk} A$ of $\bk \llbracket h \rrbracket$ and $A$ as algebras, where the multiplication is given by $(a \otimes x)(b \otimes y) = ab \otimes xy$ for all $a, b \in \bk \llbracket h \rrbracket$ and $x, y \in A$.
\end{definition}
Every reference to the Weyl--Heisenberg algebra in our discussion means that we consider $A_n(\bk)$ as a $\bk \llbracket h \rrbracket$-algebra. In this way, the exponential map and logarithm are well-defined, so we can invoke the BCH formula. The map $\mathcal{N}_I$ extends $\bk \llbracket h \rrbracket$-linearly and $h$-adically continuously, making it an isomorphism of topological $\bk \llbracket h \rrbracket$-modules.

\begin{notation}
  Assume the generators $\{\mathbf{x}_i,\mathbf{p}_i\}_{i \in I}$ of $A_n(\bk)$ to be ordered as in \eqref{eq:ordering}. Moreover, set $x = (x_{i_{1}}, \dots, x_{i_{n}})$ and $p = (p_{i_{1}}, \dots, p_{i_{n}})$. Similarly, we denote $\mathbf{x}:=(\mathbf{x}_{i_{1}},\dots,\mathbf{x}_{i_{n}})$ and $\mathbf{p} := (\mathbf{p}_{i_{1}},\dots,\mathbf{p}_{i_{n}})$. Extending the map $\mathcal{N}$ entry-wise yields $\mathbf{x}= \mathcal{N}(x)$ and $\mathbf{p} = \mathcal{N}(p)$. Throughout this paper we shall use the Einstein summation convention.
\end{notation}

\begin{lemma}
  \label{lem:expconj}
Let $A\in \mathrm{Mat}_{n}(\bk\llbracket h \rrbracket)$. The following identities hold:
\begin{enumerate}[\normalfont i)]
  \item\label{item:1}   $\exp(h \mathbf{x}^{\top} A \mathbf{p})\mathbf{x}_{j} \exp(-h \mathbf{x}^{\top} A \mathbf{p})= \mathbf{x}_{i} (e^{hA})_{ij}$;
        \item  $\exp(h \mathbf{x}^{\top} A \mathbf{p})\mathbf{p}_{j} \exp(-h \mathbf{x}^{\top} A \mathbf{p})= \mathbf{p}_{i} (e^{-hA})^{\top}_{ij}$.
\end{enumerate}

\end{lemma}

Let $P \colon A_{n}(\bk) \otimes A_{n}(\bk) \to A_{n}(\bk) \otimes A_{n}(\bk)$ be the interchanging map defined by $a \otimes b \mapsto b \otimes a$. Also, let $\widetilde{P}$ be the map that interchanges the columns of a $2\times 2$ matrix.
\begin{proposition}
  \label{pro:adexp}
  Let $n=2$ and $A \in \mathrm{Mat}_{2}(\bk \llbracket h \rrbracket)$. The following identity holds:
  \[
    P \circ \mathrm{Ad}_{\exp(h \mathbf{x}^{\top} A \mathbf{p})} \mathbf{x} =  \widetilde{P}((e^{hA})^{\top}) \mathbf{x}.
    \]
\end{proposition}
\begin{proof}
This is an immediate consequence of Lemma \ref{lem:expconj}.
\end{proof}
The following proposition in inspired by \cite[Lemma 10]{bar2021perturbed}
\begin{proposition}
  \label{pro:normalordering}
Let $A\in \mathrm{Mat}_{n}(\bk \llbracket h \rrbracket)$. The following identity holds:
\[
\exp(h  \mathbf{x}^{\top} A \mathbf{p}) = \mathcal{N} \left(  \exp(x^{\top} (e^{h  A}- \mathds{1}) p)\right).
\]
\end{proposition}
We introduce a new notation for elements of the type for which Proposition \ref{pro:normalordering} can be applied.
\begin{notation}
  \label{def:monhom}
  Denote $\mathcal{A}_{n}(\bk) := \exp(h \mathrm{Mat}_n(\bk \llbracket h \rrbracket))$ with monoid structure given by multiplication.
  Define the map $\varphi \colon \mathcal{A}_n(\bk) \to A_n(\bk)$ by
  \[
    A \mapsto \mathcal{N} \left( \exp(  x^{\top}   (A - \mathds{1})  p ) \right).
    \]
\end{notation}

The map $\varphi$ provides a dictionary between taking tensor products and adjoining block-matrices. As shall be shown below, adjoining identity factors in a tensor product corresponds, under $\varphi$, to adjoining identity summands in a direct sum.

\begin{notation}
  \label{not:tensorthings}

  Fix integers $m \ge 2$ and $1 \le i < j \le m$ and let $A\in\mathcal A_2(\Bbbk)$. Define
  \[
    \varphi_{i,j}^{m}(A)
    :=
    \bigl(I^{\otimes(i-1)} \otimes \mathrm{Id} \otimes I^{\otimes(j-i-1)} \otimes \mathrm{Id} \otimes I^{\otimes(m-j)}\bigr)\bigl(\varphi(A)\bigr).
  \]
  Here $I$ denotes the identity on the underlying tensor factor.

Let $R_k = \bk \llbracket h \rrbracket$ and $R=\bigoplus_{k=1}^m R_k$. Consider the canonical injections $\iota_k \colon R_k \to R$ and projections $\pi_k\colon R \to R_k$. Moreover, let $T\in \mathrm{End}(R_i\oplus R_j)$ having entries $T_{ab}\in \mathrm{Hom}(R_b,R_a)$ for $a,b\in\{i,j\}$.  We define the block-insertion map
  \[
    \Psi^{m}_{i,j}(T)
    :=
    \sum_{k\neq i,j}\iota_k\pi_k
    +
    \sum_{a,b\in\{i,j\}}\iota_aT_{ab}\pi_b
    \in \mathrm{End}(R).
  \]
  In other words,  $\Psi^m_{i,j}(T)$ acts as the identity on $R_k$ for $k\notin\{i,j\}$, and reduces to $T$ on $R_i\oplus R_j$. Hence, $\Psi^m_{i,j}$ places a $2\times2$ block $T$ in positions $i$ and $j$.
\end{notation}  
\begin{proposition}
  \label{pro:phipsi}
Let $A \in \mathcal{A}_{2}(\bk)$. The following equality holds:
\[
\varphi^{m}_{i,j}(A) = \varphi(\Psi^{m}_{i,j}(A)).
\]
\end{proposition}
\begin{example}
Consider
\[
  A =
  \begin{pmatrix}
    T & 0 \\
    1 - T^{2} & T
  \end{pmatrix},
  \qquad
  T = e^{ht}.
\]

When we insert this $2\times 2$ block into positions $1,2$ of a bigger $3\times 3$ direct sum via $\Psi^{3}_{1,2}$, we obtain
\[
  \Psi^{3}_{1,2}(A)
  =
  \begin{pmatrix}
    T & 0 & 0 \\[2pt]
    1 - T^{2} & T & 0 \\[2pt]
    0 & 0 & 1
  \end{pmatrix}.
\]
Note that  $\Psi^{3}_{1,2}(A)$ acts as $A$ on the first two summands and as the identity on the third.
Consistent with the previous discussion, we find
\[
  \varphi^{3}_{1,2}(A) = \varphi(A) \otimes 1 = \varphi\bigl(\Psi^{3}_{1,2}(A)\bigr),
\]
which agrees with Proposition~\ref{pro:phipsi}.
\end{example}

Combining Proposition \ref{pro:normalordering} with Lemma \ref{lem:expconj} gives rise to the following.
\begin{proposition}
  \label{pro:phixandp}
  Let $A \in \mathcal{A}_n(\bk)$. The following identities hold:
\begin{enumerate}[\normalfont i)]
    \item $\varphi(A) \mathbf{x}_j \varphi(A)^{-1} = \mathbf{x}_i A_{ij}$;
    \item $\varphi(A) \mathbf{p}_j \varphi(A)^{-1} = \mathbf{p}_i (A^{-1})^{\top}_{ij}$.
  \end{enumerate}
\end{proposition}
With the required structure in place, we can now derive the following useful identity.
\begin{theorem}
  \label{cor:identity}
  Let $A,B \in \mathcal{A}_n(\bk)$. The following equality holds:
  \[
    \varphi(A) \varphi(B)= \varphi(A B),
    \]
    i.e., $\varphi$ is a monoid homomorphism.
\end{theorem}

Observe that $\varphi(A)$, for $A \in \mathcal{A}_n(\bk)$, takes the form of a Gaussian. This suggests considering perturbed Gaussians, as discussed in the next proposition. Later, we show that the $R$-matrix in quantum group theory assumes this form, which justifies the notation used.
\begin{proposition}
  \label{pro:perturbedexptill2}
  Let $A \in \mathcal{A}_{2}(\bk )$.
Define
\begin{align*}
 A_{2}(\bk(\epsilon)) \ni R :=   \varphi(A)  \mathrm{exp}(\epsilon\mathcal{N}( f_{1}(x,p) + \epsilon f_{2}(x,p))).
\end{align*}
Then
\begin{align*}
  R \mathbf{x}_{j} R^{-1} &= \mathbf{x}_{i}A_{ij} + \epsilon\mathrm{Ad}_{\varphi(A)}   \mathcal{N}(\partial_{p_{j}} f_{1}(x,p)  + \epsilon\partial_{p_{j}}f_{2}(x,p)) \\
  &\quad + \frac1 2 \epsilon^{2}\mathrm{Ad}_{\varphi(A)}[\mathcal{N}(f_{1}(x,p)),\mathcal{N}(\partial_{p_{j}}f_{1}(x,p))] \mod \epsilon^{3}.
\end{align*}
\end{proposition}
\begin{proof}
  Let $f(x,p,\epsilon) = f_{1}(x,p) + \epsilon f_{2}(x,p)$. We have
\begin{align*}
  R \mathbf{x}_{j} R^{-1} &= \mathrm{Ad}_{\varphi(A)}   \mathrm{exp}(\mathcal{N}(\epsilon f(x,p, \epsilon))) \mathbf{x}_{j}  \mathrm{exp}(\mathcal{N}(-\epsilon f(x,p, \epsilon)))  \\
                          &= \mathrm{Ad}_{\varphi(A)} e^{[\mathcal{N}(\epsilon f(x,p, \epsilon)), \cdot]}  \mathbf{x}_{j}.
\end{align*}
Note that for any polynomial $g(x,p)$, we have $[\mathcal{N}(g(x,p)),\mathbf{x}_{j}] = \mathcal{N}(\partial_{p_{j}} g(x,p))$. Expanding the exponential up to second order in $\epsilon$ yields the desired expression.
\end{proof}

\section{Cluster Algebras}
\label{sec:clusteralgebras}
\noindent
In this section, we review cluster algebras in the classical sense. We show that for a specific choice of cluster variables, we can construct a braiding operator that yields the Alexander polynomial.
  \subsection{Definitions}
  We recall the definition of cluster algebras based on \cite{Schrader2019}.
  A \textit{seed} is a triple $(I,I_0,\varepsilon)$ consisting of: a finite set denoted $I$, a subset $I_0 \subset I$, and a skew-symmetric matrix $\varepsilon = (\varepsilon_{ij})_{i,j \in I}$ with values in $\frac{1}{2}\mathbb{Z}$, where $\varepsilon_{ij}$ is an integer unless $i,j \in I_0$. The matrix $\varepsilon$ is called the \textit{exchange matrix}. To any seed $\Sigma$ two algebraic tori are associated: one called the \textit{cluster $\mathcal{A}$-torus} $\mathcal{A} = (\mathbb{C}^{\times})^{|I|}$ with coordinates $\{A_1, A_2, \ldots, A_{|I|}\}$, and the other known as the \textit{cluster $\mathcal{X}$-torus} $\mathcal{X} = (\mathbb{C}^{\times})^{|I|}$ with coordinates $\{X_1, X_2, \ldots, X_{|I|}\}$. The coordinates $A_i$ and $X_i$ are termed \textit{cluster $\mathcal{A}$-variables} and \textit{cluster $\mathcal{X}$-variables}, respectively. Variables are said to be \textit{frozen} when $i \in I_0$. The collection $\{A_i\}_{i\in I}$ (respectively $\{X_i\}_{i\in I}$) associated to a given seed is called a \textit{cluster}. A chosen seed from which the mutation process begins is called the \textit{initial seed}.

Let $M$ be the $|I| \times |I|$ matrix with entries $M_{ij} = 0$ unless both $i$ and $j$ are frozen, and such that the matrix $\widetilde{\varepsilon} = \varepsilon + M$ has integer coefficients. The regular map $p_{\Sigma}^M \colon \mathcal{A}_{\Sigma} \to \mathcal{X}_{\Sigma}$, known as the \textit{cluster ensemble map}, is defined by the formula:
\[ \left(p_{\Sigma}^M\right)^* X_k = \prod_{i \in I} A_i^{\widetilde{\varepsilon}_{ki}}. \]

Given a pair of seeds $\Sigma = (I, I_0, \varepsilon)$, $\Sigma' = (I',I_0',\varepsilon')$, together with an element $k \in I \setminus I_0$, a bijection $\mu_k \colon I \to I'$ is said to be a \textit{mutation in direction $k$}, if $\mu_k(I_0) = I_0'$ and
\[
  \varepsilon_{\mu_{k}(i), \mu_{k}(j)}^{\prime}=\left\{\begin{array}{ll}-\varepsilon_{i j} & \text { if } i=k \text { or } j=k, \\ \varepsilon_{i j} & \text { if } \varepsilon_{i k} \varepsilon_{k j} \leqslant 0, \\ \varepsilon_{i j}+\left|\varepsilon_{i k}\right| \varepsilon_{k j} & \text { if } \varepsilon_{i k} \varepsilon_{k j}>0.\end{array}\right.
  \]

Rather than solely relying on combinatorial information, the seed is represented graphically using a quiver, whose vertices are labeled with elements of the set $I$, and the number of arrows from $i$ to $j$ is $[\varepsilon_{ij}]_+= \mathrm{max} \{0, \varepsilon_{ij} \}$. If an entry is half-integer valued, then a dashed arrow is drawn. Vertices related to $I_0$ are displayed by squares instead of circles. 

\begin{example}[\cite{shen2022clusternaturequantumgroups}]
  Let $I= \{a,b,c,d\}$ and $I_0 = \{c,d\}$ and
  \[
\varepsilon=\left(\begin{array}{cccc}0 & -1 & 0 & 0 \\ 1 & 0 & -1 & 1 \\ 0 & 1 & 0 & -\frac{1}{2} \\ 0 & -1 & \frac{1}{2} & 0\end{array}\right).
    \]
    The corresponding quiver is shown in the left graph of Figure \ref{fig:mutationexample}.
The dashed arrow between the fixed vertices $c$ and $d$ signifies that $\varepsilon_{dc} = \frac1 2$. The image below shows the result of a quiver mutation in the direction $b$.
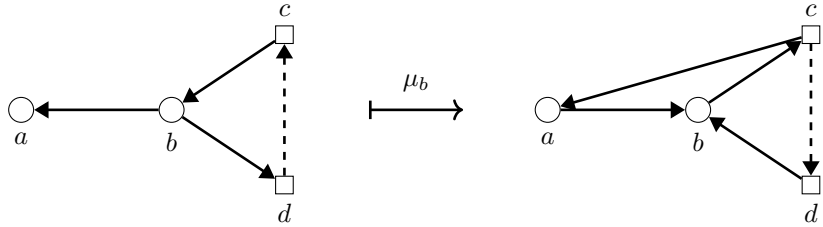
\begin{figure}[H]
  \centering
\begin{tikzpicture}
  \node[circle, draw, fill=white] (a) at (0,0) {};
  \node[below] at (a.south) {$a$};
  \node[circle, draw, fill=white] (b) at (2,0) {};
  \node[below] at (b.south) {$b$};
  \node[rectangle, draw, fill=white] (c) at (3.5,1) {};
  \node[above] at (c.north) {$c$};
  \node[rectangle, draw, fill=white] (d) at (3.5,-1) {};
  \node[below] at (d.south) {$d$};
  \node[] (e) at (4.5,0) {};
  \node[] (f) at (6,0) {};
  \node[] (g) at (5.25,0) {};
  \node[above] at (g.north) {$\mu_b$};
  \node[circle, draw, fill=white] (a1) at (7,0) {};
  \node[below] at (a1.south) {$a$};
  \node[circle, draw, fill=white] (b1) at (9,0) {};
  \node[below] at (b1.south) {$b$};
  \node[rectangle, draw, fill=white] (c1) at (10.5,1) {};
  \node[above] at (c1.north) {$c$};
  \node[rectangle, draw, fill=white] (d1) at (10.5,-1) {};
  \node[below] at (d1.south) {$d$};
  \draw[-{Triangle}, line width=1pt] (b) -- (a);
  \draw[-{Triangle}, line width=1pt] (b) -- (d);
  \draw[-{Triangle}, line width=1pt, dashed] (d) -- (c);
  \draw[-{Triangle}, line width=1pt] (c) -- (b);
  \draw[|->, line width=1pt] (e) -- (f);
  \draw[{Triangle}-, line width=1pt] (b1) -- (a1);
  \draw[{Triangle}-, line width=1pt] (b1) -- (d1);
  \draw[{Triangle}-, line width=1pt, dashed] (d1) -- (c1);
  \draw[{Triangle}-, line width=1pt] (c1) -- (b1);
  \draw[-{Triangle}, line width=1pt] (c1) -- (a1);
\end{tikzpicture}
\caption{A quiver mutation in the direction $b$.}
\label{fig:mutationexample}
\end{figure}
\end{example}

To a seed mutation $\mu_k$, we associate a pair of birational isomorphisms of cluster tori $\mu_k^{\mathcal{A}} \colon \mathcal{A}_{\Sigma} \to \mathcal{A}_{\Sigma'}$ and $\mu_k^{\mathcal{X}} \colon \mathcal{X}_{\Sigma} \to \mathcal{X}_{\Sigma'}$ defined as follows
\begin{equation}
  \label{eq:clasmut}
  \begin{array}{l}\left(\mu_{k}^{\mathcal{A}}\right)^{*} A_{\mu_{k}(i)}=\left\{\begin{array}{ll}A_{k}^{-1}\left(\prod_{i \mid \varepsilon_{k i}>0} A_{i}^{\varepsilon_{k i}}+\prod_{i \mid \varepsilon_{k i}<0} A_{i}^{-\varepsilon_{k i}}\right) & \text { if } i=k, \\ A_{i} & \text { if } i \neq k,\end{array}\right. \\
\text{and} \\
    \left(\mu_{k}^{\mathcal{X}}\right)^{*} X_{\mu_{k}(i)}=\left\{\begin{array}{ll}X_{k}^{-1} & \text { if } i=k, \\ X_{i}\left(1+X_{k}^{-\operatorname{sgn}\left(\varepsilon_{k i}\right)}\right)^{-\varepsilon_{k i}} & \text { if } i \neq k.\end{array}\right.\end{array}
  \end{equation}
  These maps are called \textit{cluster $\mathcal{A}$-} and \textit{cluster $\mathcal{X}$-mutations}, respectively. The cluster ensemble morphism $p_{\Sigma}^M$ intertwines these two types of mutations with the equation
\[ \mu_{k}^{\mathcal{X}} \circ p_{\Sigma}^{M} = p_{\mu_k(\Sigma)}^{M} \circ \mu_{k}^{\mathcal{A}}. \]
   \begin{definition}
     Let $\Sigma$ be a seed. The entries of the clusters of all seeds reachable from $\Sigma$ by a sequence of mutations are called the \textit{cluster variables}.

     The \textit{cluster algebra} associated to the seed $\Sigma$ is the subring of the fraction field of $\mathcal{O}(\mathcal{A}_{\Sigma})$ generated by cluster variables from all seed mutations equivalent to $\Sigma$.
\end{definition}
A \textit{permutation} of a seed $\sigma \colon I \to I'$ is a bijection $\sigma \colon I \to I'$ such that
\begin{align*}
  \sigma(I_0) &= I'_0, \\
  \varepsilon_{ij}' &= \varepsilon_{\sigma(i) \sigma(j)}.
\end{align*}
As with seed mutations, this map induces an isomorphism of cluster tori
\[
  (\sigma^{\mathcal{A}})^{*} \colon \mathcal{A}_{\Sigma} \to \mathcal{A}_{\Sigma'} \qquad \text{and} \qquad (\sigma^{\mathcal{X}})^{*} \colon \mathcal{X}_{\Sigma} \to \mathcal{X}_{\Sigma'}
\]
by relabeling the coordinates. We write $\sigma_{i,j}$ for the permutation that swaps the subscript $i$ and $j$, e.g.
\[
  \sigma_{i,j}(\dots,x_i,\dots,x_j,\dots) = (\dots,x_j,\dots,x_i,\dots).
  \]
A \textit{Mathematica} implementation of these formulas can be found in Appendix~\ref{sec:mathematica}.

\subsection{Braiding Operator}
\label{sec:braidingoperator}
The $R$-matrix is fundamental in knot theory because it allows the construction of knot invariants. It can be derived from the Drinfeld double of the Borel subgroup of a semisimple Lie algebra. It also arises within the theory of cluster algebras, as discussed in \cite{Fomin2008,hikami2014braiding}. Geometrically, braid group generators correspond to half-Dehn twists on a punctured disk. 

A half-Dehn twist on this surface can be understood via its action on an ideal triangulation of a punctured disk, through a finite sequence of “flips”. A flip removes a common edge of two adjacent triangles and adds a new diagonal edge of the quadrilateral. After a finite number of such moves, the resulting triangulation agrees with the one obtained by applying the half-Dehn twist to the surface. 

Consider, for example, the triangulation of the 2-punctured disk as illustrated in Figure \ref{fig:triangulation}.
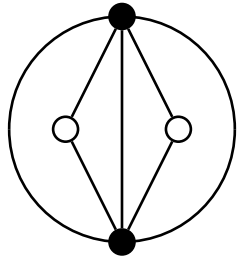
\begin{figure}[H]
  \centering
\begin{tikzpicture}[scale=0.75, every path/.style={line width=1pt}]
  \node[circle, draw, fill=black] (0) at (0, 2) {};
  \node[circle, draw, fill=black] (1) at (0, -2) {};
  \node (4) at (-2, 0) {};
  \node (5) at (2, 0) {};
  \node[circle, draw, fill=white] (2) at (-1, 0) {};
  \node[circle, draw, fill=white] (3) at (1, 0) {};
  \draw (0) to (1);
  \draw (1) to (2);
  \draw (2) to (0);
  \draw (0) to (3);
  \draw (3) to (1);
  \draw [in=180, out=-90] (4.center) to (1.center);
  \draw [in=-180, out=90] (4.center) to (0.center);
  \draw [in=90, out=0] (0.center) to (5.center);
  \draw [in=0, out=-90] (5.center) to (1.center);
\end{tikzpicture}
\caption{A triangulation of a 2-punctured disk}
  \label{fig:triangulation}
\end{figure}
The sequence of flips that give rise to the half Dehn twist on the punctured surface has been illustrated in Figure \ref{fig:dehn}
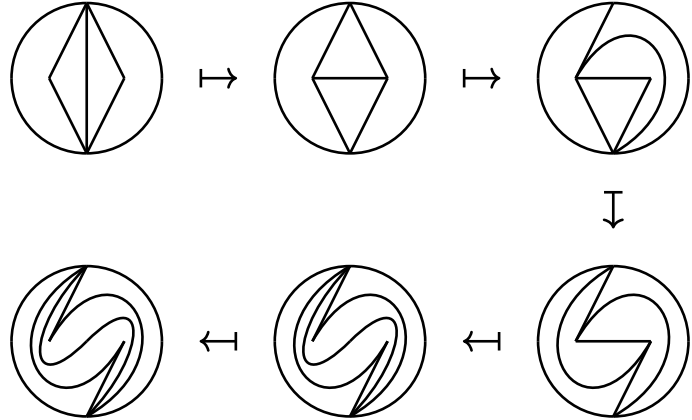
\begin{figure}[H]
  \centering
\begin{tikzpicture}[scale=0.5, every path/.style={line width=1pt}]
  \node (0) at (0, 2) {};
  \node (1) at (0, -2) {};
  \node (2) at (-1, 0) {};
  \node (3) at (1, 0) {};
  \node (4) at (-2, 0) {};
  \node (5) at (2, 0) {};
  \node (6) at (7, 2) {};
  \node (7) at (7, -2) {};
  \node (8) at (6, 0) {};
  \node (9) at (8, 0) {};
  \node (10) at (5, 0) {};
  \node (11) at (9, 0) {};
  \node (12) at (8, 0) {};
  \node (13) at (14, 2) {};
  \node (14) at (14, -2) {};
  \node (15) at (13, 0) {};
  \node (16) at (15, 0) {};
  \node (17) at (12, 0) {};
  \node (18) at (16, 0) {};
  \node (19) at (15, 0) {};
  \node (20) at (6.25, 1.5) {};
  \node (21) at (3, 0) {};
  \node (22) at (4, 0) {};
  \node (23) at (10, 0) {};
  \node (24) at (11, 0) {};
  \node (25) at (15, 0.75) {};
  \node (39) at (14, -5) {};
  \node (40) at (14, -9) {};
  \node (41) at (13, -7) {};
  \node (42) at (15, -7) {};
  \node (43) at (12, -7) {};
  \node (44) at (16, -7) {};
  \node (45) at (15, -7) {};
  \node (47) at (3, -7) {};
  \node (48) at (4, -7) {};
  \node (49) at (10, -7) {};
  \node (50) at (11, -7) {};
  \node (51) at (15, -6.25) {};
  \node (52) at (14, -3) {};
  \node (53) at (14, -4) {};
  \node (54) at (7, -5) {};
  \node (55) at (7, -9) {};
  \node (56) at (6, -7) {};
  \node (57) at (8, -7) {};
  \node (58) at (5, -7) {};
  \node (59) at (9, -7) {};
  \node (60) at (8, -7) {};
  \node (61) at (8, -6.25) {};
  \node (62) at (7, -7) {};
  \node (63) at (0, -5) {};
  \node (64) at (0, -9) {};
  \node (65) at (-1, -7) {};
  \node (66) at (1, -7) {};
  \node (67) at (-2, -7) {};
  \node (68) at (2, -7) {};
  \node (69) at (1, -7) {};
  \node (70) at (1, -6.25) {};
  \node (71) at (0, -7) {};
  \draw (0.center) to (1.center);
  \draw (1.center) to (2.center);
  \draw (2.center) to (0.center);
  \draw (0.center) to (3.center);
  \draw (3.center) to (1.center);
  \draw [in=180, out=-90] (4.center) to (1.center);
  \draw [in=-180, out=90] (4.center) to (0.center);
  \draw [in=90, out=0] (0.center) to (5.center);
  \draw [in=0, out=-90] (5.center) to (1.center);
  \draw (7.center) to (8.center);
  \draw (8.center) to (6.center);
  \draw (6.center) to (9.center);
  \draw (9.center) to (7.center);
  \draw [in=180, out=-90] (10.center) to (7.center);
  \draw [in=-180, out=90] (10.center) to (6.center);
  \draw [in=90, out=0] (6.center) to (11.center);
  \draw [in=0, out=-90] (11.center) to (7.center);
  \draw (8.center) to (12.center);
  \draw (14.center) to (15.center);
  \draw (15.center) to (13.center);
  \draw (16.center) to (14.center);
  \draw [in=180, out=-90] (17.center) to (14.center);
  \draw [in=-180, out=90] (17.center) to (13.center);
  \draw [in=90, out=0] (13.center) to (18.center);
  \draw [in=0, out=-90] (18.center) to (14.center);
  \draw (15.center) to (19.center);
    \draw[|->, line width=1pt] (21.center) -- (22.center);
    \draw[|->, line width=1pt] (23.center) -- (24.center);
  \draw [in=30, out=60, looseness=3.75] (15.center) to (14.center);
  \draw (41.center) to (39.center);
  \draw (42.center) to (40.center);
  \draw [in=180, out=-90] (43.center) to (40.center);
  \draw [in=-180, out=90] (43.center) to (39.center);
  \draw [in=90, out=0] (39.center) to (44.center);
  \draw [in=0, out=-90] (44.center) to (40.center);
  \draw (41.center) to (45.center);
    \draw[<-|, line width=1pt] (47.center) -- (48.center);
    \draw[<-|, line width=1pt] (49.center) -- (50.center);
  \draw [in=30, out=60, looseness=4.00] (41.center) to (40.center);
    \draw[|->, line width=1pt] (52.center) -- (53.center);
  \draw [in=-120, out=-150, looseness=4.00] (39.center) to (45.center);
  \draw (56.center) to (54.center);
  \draw (57.center) to (55.center);
  \draw [in=180, out=-90] (58.center) to (55.center);
  \draw [in=-180, out=90] (58.center) to (54.center);
  \draw [in=90, out=0] (54.center) to (59.center);
  \draw [in=0, out=-90] (59.center) to (55.center);
  \draw [in=30, out=60, looseness=4.00] (56.center) to (55.center);
  \draw [in=-120, out=-150, looseness=4.00] (54.center) to (60.center);
  \draw [in=45, out=45, looseness=3.00] (62.center) to (55.center);
  \draw [in=-135, out=-135, looseness=3.00] (54.center) to (62.center);
  \draw (65.center) to (63.center);
  \draw (66.center) to (64.center);
  \draw [in=180, out=-90] (67.center) to (64.center);
  \draw [in=-180, out=90] (67.center) to (63.center);
  \draw [in=90, out=0] (63.center) to (68.center);
  \draw [in=0, out=-90] (68.center) to (64.center);
  \draw [in=30, out=60, looseness=4.00] (65.center) to (64.center);
  \draw [in=-120, out=-150, looseness=4.00] (63.center) to (69.center);
  \draw [in=45, out=45, looseness=3.00] (71.center) to (64.center);
  \draw [in=-135, out=-135, looseness=3.00] (63.center) to (71.center);
\end{tikzpicture}
\caption{Half Dehn-twist for the 2-punctured disk.}
\label{fig:dehn}
\end{figure}
The sequence of flips that give rise to the half-Dehn twist has an interpretation in the theory of cluster algebras as explained in \cite{Fomin2008}. We associate to each ideal triangulation $T$ of a punctured surface a quiver $Q_T$ whose vertices correspond to the arcs of $T$. Moreover, we assign a \textit{cluster variable} to each vertex. For each triangle of $T$, we draw an arrow $i \to j$ in $Q_T$ whenever $i$ and $j$ are sides of that triangle and $j$ follows $i$ in clockwise order. The surface orientation makes this unambiguous. Finally, any $2$-cycles in the quiver are canceled. For example, consider the two-punctured disk in Figure \ref{fig:triangulation}. Following the prescription above, we build the quiver:

\[
\begin{tikzpicture}
  \node[rectangle, draw, fill=white] (a) at (0,0) {};
  \node[below] at (a.south) {$1$};
  \node[circle, draw, fill=white] (b) at (1,-1) {};
  \node[below] at (b.south) {$2$};
  \node[circle, draw, fill=white] (c) at (2,0) {};
  \node[below] at (c.south) {$4$};
  \node[circle, draw, fill=white] (d) at (1,1) {};
  \node[below] at (d.south) {$3$};

  \node[circle, draw, fill=white] (5) at (3,-1) {};
  \node[below] at (5.south) {$5$};
  \node[circle, draw, fill=white] (6) at (3,1) {};
  \node[below] at (6.south) {$6$};
  \node[rectangle, draw, fill=white] (7) at (4,0) {};
  \node[below] at (7.south) {$7$};
  \draw[-{Triangle}, line width=1pt] (a) -- (b);
  \draw[-{Triangle}, line width=1pt] (b) -- (c);
  \draw[-{Triangle}, line width=1pt] (c) -- (d);
  \draw[-{Triangle}, line width=1pt] (d) -- (a);
  \draw[-{Triangle}, line width=1pt] (c) -- (5);
  \draw[-{Triangle}, line width=1pt] (5) -- (7);
  \draw[-{Triangle}, line width=1pt] (7) -- (6);
  \draw[-{Triangle}, line width=1pt] (6) -- (c);
\end{tikzpicture}
\]
The vertices are labeled by $I=\{1,2,3,4,5,6,7\}$, and we set $I_0 = \{1,7\}$ in accordance with the definition of cluster algebras. 

We define the $R$-operator acting on any of the cluster $\mathcal{A}$- or $\mathcal{X}$-variables as
\begin{equation}
\label{eq:roperator}
R = \sigma_{3,5} \circ \sigma_{2,5}  \circ \sigma_{3,6} \circ  \mu_4 \circ  \mu_2 \circ  \mu_6 \circ  \mu_{4}.
\end{equation}
We omitted the superscripts on the (per)mutations to improve readability. This series of (per)mutations results from performing a half-Dehn twist on the punctured disk, where the two punctures are interchanged in a counterclockwise fashion, consistent with Figure \ref{fig:dehn}

In terms of the cluster $\mathcal{A}$-variables, we have
\begin{align*}
  R^{\mathcal{A}} \bigg(A_{1}, A_{2}, \ldots, A_{7}\bigg)=\bigg(A_{1}, A_{5}, \frac{A_{1} A_{3} A_{5}+A_{3} A_{4} A_{5}+A_{1} A_{2} A_{6}}{A_{2} A_{4}}, \\ \quad \frac{A_{1} A_{3} A_{4} A_{5}+A_{3} A_{4}^{2} A_{5}+A_{1} A_{3} A_{5} A_{7}+A_{3} A_{4} A_{5} A_{7}+A_{1} A_{2} A_{6} A_{7}}{A_{2} A_{4} A_{6}} \\ \quad \frac{A_{3} A_{4} A_{5}+A_{3} A_{5} A_{7}+A_{2} A_{6} A_{7}}{A_{4} A_{6}}, A_{3}, A_{7}\bigg),
\end{align*}
while in terms of the cluster $\mathcal{X}$-variables;
\begin{align*}
  R^{\mathcal{X}}&\left(X_{1}, X_{2}, \ldots, X_{7}\right)=\bigg(X_{1}(1+X_{2}+X_{2} X_{4}), \frac{X_{2} X_{4} X_{5} X_{6}}{1+X_{2}+X_{6}+X_{2} X_{6}+X_{2} X_{4} X_{6}}, \\  &\frac{1+X_{2}+X_{6}+X_{2} X_{6}+X_{2} X_{4} X_{6}}{X_{2} X_{4}}, \frac{X_{4}}{(1+X_{2}+X_{2} X_{4})(1+X_{6}+X_{4} X_{6})}, \\ & \frac{1+X_{2}+X_{6}+X_{2} X_{6}+X_{2} X_{4} X_{6}}{X_{4} X_{6}}, \frac{X_{2} X_{3} X_{4} X_{6}}{1+X_{2}+X_{6}+X_{2} X_{6}+X_{2} X_{4} X_{6}}, \\
  &(1+X_{6}+X_{4} X_{6}) X_{7}\bigg).
\end{align*}
The last expression has been verified in \textit{Mathematica}, as described in Appendix \ref{sec:mathematica}.

The discussion above focuses on the two-punctured disk, but to compute knot invariants, we must also consider the multi-punctured disk. This lets us generate the braid-group operators. To achieve this, we consider the following quiver:
\[
\begin{tikzpicture}[scale=0.9]
  \node[rectangle, draw, fill=white] (0) at (-5, 0) {};
  \node[below] at (0.south) {$1$};
  \node[circle, draw, fill=white] (1) at (-4, -1) {};
  \node[below] at (1.south) {$2$};
  \node[circle, draw, fill=white] (2) at (-3, 0) {};
  \node[below] at (2.south) {$4$};
  \node[circle, draw, fill=white] (3) at (-4, 1) {};
  \node[below] at (3.south) {$3$};
  \node[circle, draw, fill=white] (4) at (-2, -1) {};
  \node[below] at (4.south) {$5$};
  \node[circle, draw, fill=white] (5) at (-1, 0) {};
  \node[below] at (5.south) {$7$};
  \node[circle, draw, fill=white] (6) at (-2, 1) {};
  \node[below] at (6.south) {$6$};
  \node (7) at (-0.5, 0.5) {};
  \node (8) at (-0.5, -0.5) {};
  \node (9) at (0.5, 0.5) {};
  \node (10) at (0.5, -0.5) {};
  \node[circle, draw, fill=white] (11) at (1, 0) {};
  \node[below,yshift=-0.2cm] at (11.south) {$3i-2$};
  \node[circle, draw, fill=white] (12) at (2, 1) {};
  \node[below] at (12.south) {$3i$};
  \node[circle, draw, fill=white] (13) at (3, 0) {};
  \node[below,yshift=-0.2cm] at (13.south) {$3i+1$};
  \node[circle, draw, fill=white] (14) at (2, -1) {};
  \node[below] at (14.south) {$3i-1$};
  \node (15) at (3.5, 0.5) {};
  \node (16) at (3.5, -0.5) {};
  \node (17) at (4.5, 0.5) {};
  \node (18) at (4.5, -0.5) {};
  \node[circle, draw, fill=white] (19) at (5, 0) {};
  \node[below,yshift=-0.2cm] at (19.south) {$3n-2$};
  \node[circle, draw, fill=white] (20) at (6, 1) {};
  \node[below] at (20.south) {$3n$};
  \node[rectangle, draw, fill=white] (21) at (7, 0) {};
  \node[below,yshift=-0.25cm] at (21.south) {$3n+1$};
  \node[circle, draw, fill=white] (22) at (6, -1) {};
  \node[below] at (22.south) {$3n-1$};
  \node (23) at (-0.5, 0) {};
  \node (24) at (0.5, 0) {};
  \node (25) at (3.5, 0) {};
  \node (26) at (4.5, 0) {};
  \draw[-{Triangle}, line width=1pt] (0) -- (1);
  \draw[-{Triangle}, line width=1pt] (1) -- (2);
  \draw[-{Triangle}, line width=1pt] (2) -- (3);
  \draw[-{Triangle}, line width=1pt] (3) -- (0);
  \draw[-{Triangle}, line width=1pt] (2) -- (4);
  \draw[-{Triangle}, line width=1pt] (4) -- (5);
  \draw[-{Triangle}, line width=1pt] (6) -- (2);
  \draw[-{Triangle}, line width=1pt] (6) -- (5);
  \draw[line width=1pt] (5) -- (7);
  \draw[line width=1pt] (5) -- (8);
  \draw[line width=1pt] (9) -- (11);
  \draw[line width=1pt] (10) -- (11);
  \draw[-{Triangle}, line width=1pt] (11) -- (14);
  \draw[-{Triangle}, line width=1pt] (14) -- (13);
  \draw[-{Triangle}, line width=1pt] (13) -- (12);
  \draw[-{Triangle}, line width=1pt] (12) -- (11);
  \draw[line width=1pt] (13) -- (16);
  \draw[line width=1pt]  (15) -- (13);
  \draw[line width=1pt] (17) -- (19);
  \draw[line width=1pt] (19) -- (18);
  \draw[-{Triangle}, line width=1pt] (19) -- (22);
  \draw[-{Triangle}, line width=1pt] (22) -- (21);
  \draw[-{Triangle}, line width=1pt] (21) -- (20);
  \draw[-{Triangle}, line width=1pt] (20) -- (19);
  \draw[dotted] (23) -- (24);
  \draw[dotted] (25) -- (26);
\end{tikzpicture}
  \]
  This quiver results from the triangulation of the $n$-punctured disk. For any $i=1,\dots, n-1$, there exists an $R$-operator acting on the associated cluster seed, given by
  \begin{equation}
  \label{eq:Rgen}
    \stackrel{i}{\mathrm{R}}=\sigma_{3 i, 3 i+2} \sigma_{3 i-1,3 i+2} \sigma_{3 i, 3 i+3} \mu_{3 i+1} \mu_{3 i-1} \mu_{3 i+3} \mu_{3 i+1}.
  \end{equation}
    This generalizes \eqref{eq:roperator}.
\subsection{The Alexander Polynomial arising from Cluster Algebras}
\label{sec:alexcluster}
In \cite{hikami2015braids}, Hikami provided a method for computing knot invariants from cluster algebras using the braid-group presentation of a knot. Let $\mathcal K$ be a knot. By Alexander's Theorem, $\mathcal K$ can be realized as the closure of a braid $\beta \in \mathcal B_n$, where
\[
\mathcal{B}_{n}
=
\left\langle
\sigma_{1}, \sigma_{2}, \ldots, \sigma_{n-1}
\middle|
\begin{array}{l}
\sigma_{i} \sigma_{j}=\sigma_{j} \sigma_{i}
\quad \text{for } |i-j|>1, \\[4pt]
\sigma_{i} \sigma_{i+1} \sigma_{i}
=
\sigma_{i+1} \sigma_{i} \sigma_{i+1}
\quad \text{for } i=1,2,\ldots,n-2
\end{array}
\right\rangle .
\]
Fix a braid-word presentation
\[
\beta
=
\sigma_{k_1}^{b_1}
\sigma_{k_2}^{b_2}
\cdots
\sigma_{k_m}^{b_m},
\qquad
b_j \in \{\pm 1\}.
\]
Let $\mathcal F = \mathbb C(x_1,\dots,x_{3n+1})$. To each generator $\sigma_i$ of the braid group $\mathcal{B}_n$ there is an invertible birational transformation
\[
 \stackrel{i}{\mathrm{R}} \ \in \mathrm{BirAut}(\mathcal F),
\]
given by \eqref{eq:Rgen}, satisfying the braid relations
\[
 \stackrel{i}{\mathrm{R}} \stackrel{j}{\mathrm{R}} \ = \ \stackrel{j}{\mathrm{R}}  \stackrel{i}{\mathrm{R}}
\quad \text{for } |i-j|>1,
\qquad
 \stackrel{i}{\mathrm{R}}  \stackrel{i+1}{\mathrm{R}} \stackrel{i}{\mathrm{R}}
=
 \stackrel{i+1}{\mathrm{R}}  \stackrel{i}{\mathrm{R}} \stackrel{i+1}{\mathrm{R}}.
\]
Define the group homomorphism
\[
\mathcal C \colon \mathcal B_n \longrightarrow \mathrm{BirAut}(\mathcal F),
\text{ by }
\sigma_i \mapsto   \stackrel{i}{\mathrm{R}}.
\]
These identifications suggest the following definition.
\begin{definition}
  \label{def:clusterpattern}
For any initial seed
\[
\boldsymbol{x}[1]=(x[1]_1,\dots,x[1]_{3n+1}),
\]
the braid word $\beta$ induces a sequence
\[
\boldsymbol{x}[1]
\xrightarrow{(\stackrel{k_1}{\mathrm{R}})^{\varepsilon_1}}
\boldsymbol{x}[2]
\xrightarrow{(\stackrel{k_2}{\mathrm{R}})^{\varepsilon_2}}
\cdots
\xrightarrow{(\stackrel{k_m}{\mathrm{R}})^{\varepsilon_m}}
\boldsymbol{x}[m+1],
\]
where $\boldsymbol{x}[j+1]=\mathcal C(\sigma_{k_j}^{\varepsilon_j})(\boldsymbol{x}[j])$. We refer to this sequence as the \emph{cluster pattern associated to the braid presentation of $\mathcal K$}.
\end{definition}
Imposing the periodicity condition 
    \begin{equation}
    \label{eq:condition}
      \boldsymbol{x}[1]=\boldsymbol{x}[m+1],
    \end{equation}
 yields an algebraic condition associated with a braid presentation. It reflects the fact that the braid closes to a knot. Using a particular choice of cluster variables satisfying \eqref{eq:condition}, Hikami computed the complex volume of certain knots. We diverge from this approach, using a different choice of cluster variables, which leads to the Alexander polynomial. We illustrate this in the following example.

\begin{example}
  \label{ex:clustertrefoil}
Let $\mathcal{K}=3_1$ be the trefoil knot. Its braid-group presentation is $\sigma_1^{3}$, with corresponding cluster pattern
  \[
   \boldsymbol{x}[1] \xrightarrow{\substack{1 \\ R}}   \boldsymbol{x}[2] \xrightarrow{\substack{1 \\ R}}   \boldsymbol{x}[3] \xrightarrow{\substack{1 \\ R}}   \boldsymbol{x}[4].
    \]
    As the initial cluster variable, we choose
    \[
      \boldsymbol{x}[1]=(x_1 T^{-1}, -T^{2}, -1, -T^{-1} x_1^{-1}x_2, -T^{2}, -1, x_2^{-1}).
      \]
      Applying the $R$-operator three times yields
      \begin{align*}
     \boldsymbol{x}[4]=&\Bigg(
                         -\frac{(-1+T^{2})(1+T^{4})x_{1}}{T} + (1+T^{2}(-1+T^{2}))x_{2},-T^{2}, -1, \\
      &-\frac{(1-T^{2}+T^{4})x_{1} - T(-1+T^{2})x_{2}}{(-1 + T)(1+T)(1+T^{4})x_{1} + T(-1 + T^{2} - T^{4})x_{2}},-T^{-2},-1,\\
        &\frac{1}{T(1-T^{2} + T^{4})x_{1} - (-1 + T)T^{2}(1+T)x_{2}} \Bigg).
      \end{align*}
      Setting $\boldsymbol{x}[1] = \boldsymbol{x}[4]$ gives rise to the following equation
      \begin{align*}
&\begin{pmatrix}
  x_1 \\
  x_2
\end{pmatrix} =
\begin{pmatrix}
  1-T^{6} + T^{4} -T^{2} & T^{5}-T^{3} +T^{1} \\
  T^{5}-T^{3} + T^{1} & -T^{4} + T^{2}
\end{pmatrix}
\begin{pmatrix}
  x_1 \\
  x_2
\end{pmatrix} \\
&        \Rightarrow
\begin{pmatrix}
  0 \\
  0
\end{pmatrix} =
\begin{pmatrix}
  -T^{6} + T^{4} -T^{2} & T^{5}-T^{3} +T \\
  T^{5}-T^{3} + T & -T^{4} + T^{2}-1
\end{pmatrix}
\begin{pmatrix}
  x_1 \\
  x_2
\end{pmatrix}.
      \end{align*}
The Alexander polynomial corresponding to the trefoil up to a factor $\pm T^n$ for some $n \in \mathbb{Z}$ can now be read off from any entry of the matrix above. Indeed, one finds $\Delta_{\mathcal{K}}(T^2) \doteq T^{2} -1 + T^{-2}$. 
\end{example}
The fact that the method described above yields the Alexander polynomial is not coincidental, but a direct consequence of the choice of the initial cluster variables. In the forthcoming Section~\ref{sec:clusterrealization}, we provide the origin of this choice.

\section{Quantization}
\label{sec:quantization}

\noindent
In this section, we build on the structures introduced in the previous section. We review \textit{quantum cluster algebras}, originally introduced in \cite{BerensteinZelevinsky2005}. By considering representations of the quantum cluster algebra, we derive an $\epsilon$-expanded $R$-matrix. We can use it to construct perturbed knot invariants.

\subsection{Quantum Cluster Algebras}
\label{sec:qca}
Let $\Sigma=(I,I_0,\varepsilon)$ be a seed. The algebra of functions $\mathcal{O}(\mathcal{X}_{\Sigma})$ admits a quantization $\mathcal{X}_{\Sigma}^q:= \mathcal{O}_q(\mathcal{X}_{\Sigma})$, which is a $\mathbb{Z}[q^{\pm \frac1 2}]$-algebra defined by generators $X_i^{\pm 1}$ $(i \in I)$ satisfying
\[
  X_i X_j = q^{-2 \varepsilon_{ij}}X_j X_i,
  \]
  known as the \textit{quantum torus algebra}.
For any mutation $\mu_k$, we associate a \textit{quantum cluster mutation}, which is a homomorphism $\mu_k^q \colon \mathcal{X}_{\Sigma}^q \to \mathcal{X}_{\Sigma'}^q$, defined by
\[
  \mu_{k}^{q}\left(X_{i}\right)=\left\{\begin{array}{ll}X_{k}^{-1}, & \text { if } i=k, \\ X_{i} \prod_{r=1}^{\varepsilon_{k i}}\left(1+q^{2 r-1} X_{k}^{-1}\right)^{-1}, & \text { if } i \neq k \text { and } \varepsilon_{k i} \geqslant 0, \\ X_{i} \prod_{r=1}^{-\varepsilon_{k i}}\left(1+q^{2 r-1} X_{k}\right), & \text { if } i \neq k \text { and } \varepsilon_{k i} \leqslant 0 .\end{array}\right.
  \]
  It is readily shown that the new variables $X_i'$ satisfy the relation
  \[
    X_i' X_j' = q^{-2 \varepsilon_{ij}'}X_j'X_i'.
    \]

    Two seeds give rise to a new seed via a process called amalgamation, which we describe as follows.

\begin{definition}
  Let $Q_1$ and $Q_2$ be a pair of quivers associated to the seeds $\Sigma_1=(I^1,I_0^1,\varepsilon^1)$ and $\Sigma_2 = (I^2,I_0^2,\varepsilon^2)$. Moreover, let $J_1 \subset I_0^1$, $J_2 \subset I_0^2$ be subsets of frozen nodes of $Q_1$ and $Q_2$, respectively. Suppose there exists a bijection $\phi \colon J_1 \to J_2$. Then the \textit{amalgamation} of the quivers $Q_1$ and $Q_2$ along $\phi$ is a new quiver $Q$ that is constructed in the following two steps
\begin{enumerate}
\item\label{item:2} for any $i \in I^1$, the vertices $v_i \in Q_1$ and $v_{\phi(i)} \in Q_2$ are identified in the union $Q_1 \sqcup Q_2$;
        \item for any a pair $ i, j \in I^1 $ for which $ Q_1 $ contains an arrow $ v_i \to v_j $ with label $ \varepsilon_{ij} $, and $ Q_2 $ contains an arrow $ v_{\phi(i)} \to v_{\phi(j)} $ with label $ \varepsilon_{\phi(i),\phi(j)} $, draw a new arrow in $ Q $ between the corresponding vertices and assign the label $ \varepsilon_{ij} + \varepsilon_{\phi(i),\phi(j)} $ to it.
\end{enumerate}
\end{definition}

The amalgamation of two quivers, $ Q_1 $ and $ Q_2 $, into a single quiver $ Q $ gives rise to an embedding of the associated quantum cluster $\mathcal{X}$-tori, $\mathcal{X}_{\Sigma}^{q} \to \mathcal{X}_{\Sigma_{1}}^{q} \otimes \mathcal{X}_{\Sigma_{2}}^{q}$, defined as follows:
\begin{equation}
\label{eq:embedding}
X_i \mapsto
\begin{cases}
X_i \otimes 1, & \text{if } i \in I^1 \setminus I_0^1, \\
1 \otimes X_i, & \text{if } i \in I^2 \setminus I_0^2, \\
X_i \otimes X_{\phi(i)}, & \text{otherwise.}
\end{cases}
\end{equation}
\begin{example}
  \label{ex:amalg}
Consider the quivers
\[
  \mathcal{D}=
\vcenter{\hbox{\begin{tikzpicture}
  \node[rectangle, draw, fill=white] (a) at (0,0) {};
  \node[below] at (a.south) {$1'$};
  \node[circle, draw, fill=white] (b) at (1,-1) {};
  \node[below] at (b.south) {$2'$};
  \node[rectangle, draw, fill=white] (c) at (2,0) {};
  \node[below] at (c.south) {$4'$};
  \node[circle, draw, fill=white] (d) at (1,1) {};
  \node[below] at (d.south) {$3'$};

  \draw[-{Triangle}, line width=1pt] (a) -- (b);
  \draw[-{Triangle}, line width=1pt] (b) -- (c);
  \draw[-{Triangle}, line width=1pt] (c) -- (d);
  \draw[-{Triangle}, line width=1pt] (d) -- (a);
\end{tikzpicture}}}
\qquad
  \mathcal{E}=
\vcenter{\hbox{\begin{tikzpicture}
  \node[rectangle, draw, fill=white] (a) at (0,0) {};
  \node[below] at (a.south) {$5'$};
  \node[circle, draw, fill=white] (b) at (1,-1) {};
  \node[below] at (b.south) {$6'$};
  \node[rectangle, draw, fill=white] (c) at (2,0) {};
  \node[below] at (c.south) {$8'$};
  \node[circle, draw, fill=white] (d) at (1,1) {};
  \node[below] at (d.south) {$7'$};

  \draw[-{Triangle}, line width=1pt] (a) -- (b);
  \draw[-{Triangle}, line width=1pt] (b) -- (c);
  \draw[-{Triangle}, line width=1pt] (c) -- (d);
  \draw[-{Triangle}, line width=1pt] (d) -- (a);
\end{tikzpicture}}}
\]
The corresponding quantum torus algebras are also denoted $\mathcal{D}$ and $\mathcal{E}$, respectively. Define the map $\phi  \colon \{4\} \to \{5\}$ by $4 \mapsto 5$. Amalgamation along $\phi$ gives rise to the quiver
\[
\mathcal{Z}=\vcenter{\hbox{
\begin{tikzpicture}
  \node[rectangle, draw, fill=white] (a) at (0,0) {};
  \node[below] at (a.south) {$1$};
  \node[circle, draw, fill=white] (b) at (1,-1) {};
  \node[below] at (b.south) {$2$};
  \node[circle, draw, fill=white] (c) at (2,0) {};
  \node[below] at (c.south) {$4$};
  \node[circle, draw, fill=white] (d) at (1,1) {};
  \node[below] at (d.south) {$3$};

  \node[circle, draw, fill=white] (5) at (3,-1) {};
  \node[below] at (5.south) {$5$};
  \node[circle, draw, fill=white] (6) at (3,1) {};
  \node[below] at (6.south) {$6$};
  \node[rectangle, draw, fill=white] (7) at (4,0) {};
  \node[below] at (7.south) {$7$};
  \draw[-{Triangle}, line width=1pt] (a) -- (b);
  \draw[-{Triangle}, line width=1pt] (b) -- (c);
  \draw[-{Triangle}, line width=1pt] (c) -- (d);
  \draw[-{Triangle}, line width=1pt] (d) -- (a);
  \draw[-{Triangle}, line width=1pt] (c) -- (5);
  \draw[-{Triangle}, line width=1pt] (5) -- (7);
  \draw[-{Triangle}, line width=1pt] (7) -- (6);
  \draw[-{Triangle}, line width=1pt] (6) -- (c);
\end{tikzpicture}}}
\]
As described above, there exists an embedding
\[
\begin{aligned}
    X_1 &\mapsto X_{1'} \otimes 1, & X_2&\mapsto X_{2'} \otimes 1, \\
  X_3 &\mapsto  X_{3'} \otimes 1,    & X_4&\mapsto X_{4'} \otimes X_{5'}, \\
  X_5 &\mapsto 1 \otimes X_{6'}, & X_{6}&\mapsto 1 \otimes X_{7'}, \\
  X_7 &\mapsto 1 \otimes X_{8'}.
\end{aligned}
\]
\end{example}
In Appendix \ref{sec:mathematica}, a \textit{Mathematica} implementation has been provided.
\subsection{Representations of Quantum Cluster Algebras}

To construct the representation theory of quantum cluster algebras in line with \cite{goncharov2019quantum}, let us recall the Schr\"odinger representation. Let $b \in \mathbb{R}$, and $q = e^{ \pi i b^2}$. The representation described in \eqref{eq:schrodinger} can be used to build a representation of quantum cluster algebras.

Consider a quantum cluster algebra $\Sigma = (I,I_0,\varepsilon)$. The skew-symmetric matrix $\varepsilon$ induces a skew-symmetric form on a lattice $\Lambda$. This in turn induces a skew-symmetric form $\left\langle \cdot, \cdot \right\rangle$  on the vector space $\Lambda \otimes \mathbb{R}$. For now, assume the form $\left\langle \cdot, \cdot \right\rangle$ is nondegenerate. This turns it into a proper symplectic form. Choose a polarization $\Lambda \otimes \mathbb{R} = \ell \oplus \ell'$. Note that \eqref{eq:schrodinger} gives rise to a representation of the Heisenberg group on the Hilbert space $\mathcal{H}(\ell)$ with action
\begin{equation}
  \label{eq:schrodinger2}
  \begin{aligned}
 & e^{i b x}\cdot \vartheta(e^{ y}) = e^{- 2 \pi  b \omega(x,y)}\vartheta(e^{y}) &&\qquad x \in \ell, y \in \ell' \\
   &e^{i b y_0} \cdot \vartheta(e^{y}) = \vartheta(e^{y - i b y_0}) = e^{-i b y_0 \partial_y} \vartheta(e^y) &&\qquad y,y_0 \in \ell' \\
   &\exp(t z) = e^{2 \pi i t} \mathrm{Id}.
\end{aligned}
\end{equation}
This gives rise to a representation of the corresponding quantum torus algebra.

Now suppose that $\Lambda$ is a lattice with an arbitrary skew-symmetric form. As above, this gives rise to a skew-symmetric form $\left\langle \cdot, \cdot \right\rangle$ on $\Lambda \otimes \mathbb{R}$. Let $\Lambda_0 \otimes \mathbb{R} \subset \Lambda \otimes \mathbb{R}$ be the kernel of this form. To $\Lambda \otimes \mathbb{R} $ we associate the representation defined on $\Lambda/\Lambda_0 \otimes \mathbb{R}$ as introduced above, where $e^{b v}$ for $v \in \Lambda_0$ acts as a scalar. Returning then to the basis in which $\varepsilon$ is expressed, we obtain a representation of the quantum torus algebra related to $\Sigma$. We denote this representation by $\rho_{\Sigma,\ell}$. We denote the corresponding Hilbert space by $\mathcal{H}_{\Sigma, \ell}$.
\begin{example}[Cluster algebra of type $A_1$]
  \label{ex:sl2schrodinger}
Let $\mathcal{D}$ be the quantum cluster algebra related to the quiver
\[
  \mathcal{D}=
\vcenter{\hbox{\begin{tikzpicture}
  \node[rectangle, draw, fill=white] (a) at (0,0) {};
  \node[below] at (a.south) {$1$};
  \node[circle, draw, fill=white] (b) at (1,-1) {};
  \node[below] at (b.south) {$2$};
  \node[rectangle, draw, fill=white] (c) at (2,0) {};
  \node[below] at (c.south) {$4$};
  \node[circle, draw, fill=white] (d) at (1,1) {};
  \node[below] at (d.south) {$3$};

  \draw[-{Triangle}, line width=1pt] (a) -- (b);
  \draw[-{Triangle}, line width=1pt] (b) -- (c);
  \draw[-{Triangle}, line width=1pt] (c) -- (d);
  \draw[-{Triangle}, line width=1pt] (d) -- (a);
\end{tikzpicture}}}
\]
This is the cluster algebra of type $A_1$. To it, we associate the lattice $\Lambda = \mathbb{Z}^4$ with a skew-symmetric form
\[
\varepsilon =
\begin{pmatrix}
  0 & 1 & -1 & 0 \\
  -1 & 0 & 0 & 1 \\
  1 & 0 & 0 & -1 \\
  0 & -1 & 1 & 0
\end{pmatrix}.
\]
Note that $\mathbb{R}^4 \cong \Lambda \otimes \mathbb{R}$ as a real vector space. The matrix $\varepsilon$ has nullspace $K = \mathbb{R}(1,0,0,1) \oplus \mathbb{R} (0,1,1,0)$. We choose a new basis
\begin{align*}
u_1 &= \frac{1}{\sqrt{2}}(1,0,0,-1) = \frac{1}{\sqrt{2}}(e_1-e_4), \\
u_2 &= \frac{1}{\sqrt{2}}(0,1,-1,0) = \frac{1}{\sqrt{2}}(e_2 - e_3 ),\\
u_3 &= \frac{1}{\sqrt{2}}(1,0,0,1) = \frac{1}{\sqrt{2}}(e_1+e_4 ),\\
u_4 &= \frac{1}{\sqrt{2}}(0,1,1,0) = \frac{1}{\sqrt{2}}(e_2 + e_{3}).
\end{align*}
We see that $\varepsilon(u_1) = -2 u_2$ and $\varepsilon(u_2) = 2 u_1$. Thus, in the space $(\Lambda \otimes \mathbb{R})/K$ with basis $\{u_1,u_2\}$, the induced map $\overline{\varepsilon}$ is represented as the matrix
\[
  \overline{\varepsilon} =
  \begin{pmatrix}
    0 & 2 \\
    -2 & 0
\end{pmatrix}.
  \]
  This matrix is clearly nondegenerate and skew-symmetric.
  Denote $x_{0} = (1,0)$ and $y_0 = (0,1)$. Following the procedure of Section~\ref{sec:heisenberggroup}, we choose a polarization $(\Lambda \otimes \mathbb{R})/K =\ell \oplus \ell'$ with $\ell = \mathbb{R} x_0$ and $\ell' = \mathbb{R} y_0$.

 We obtain a representation of the Heisenberg group $H_1$, realized on the Hilbert space~$L^2(\mathbb{R})$, with the following actions
\begin{align*}
 & e^{ib x_{0}}\cdot \vartheta(e^{ y}) = e^{-4 \pi  b y }\vartheta(e^{y}),\\
  &e^{i b y_0} \cdot \vartheta(e^{y}) = \vartheta(e^{y - i}) = e^{-ib  \partial_y} \vartheta(e^y), \\
  &\exp(t z) = e^{2 \pi i t} \mathrm{Id},
\end{align*}
with $y \in \ell'$. Note that these are the actions of the Heisenberg group with $\overline{\varepsilon}$ expressed in the basis $\{u_1,u_2\}$. As stated above, $\{u_3,u_4\}$ act as scalars
\begin{align*}
  e^{i b u_3}\cdot \vartheta(e^{ y}) = e^{-  \sqrt{2}\pi i b \lambda_{1} }\vartheta(e^{y}) \quad \text{ and } \quad
  e^{i b u_4} \cdot \vartheta(e^{y}) = e^{ - \sqrt{2} \pi i b \lambda_2 } \vartheta(e^{y}).
\end{align*}
Returning to the basis $\{e_1,e_2,e_3,e_4\}$, the following actions are obtained:
\begin{align*}
 & e^{i b e_1}\cdot \vartheta(e^{ y}) = e^{\frac{i b\sqrt{2}}{2}(u_1 + u_3)} \cdot \vartheta(e^{ y}) = e^{ - 2 \sqrt{2}\pi  b y - \pi i b\lambda_1 }\vartheta(e^{y}),\\
 & e^{i b e_2} \cdot \vartheta(e^{ y}) = e^{\frac{i b\sqrt{2}}{2}(u_2 + u_4)} \cdot \vartheta(e^{ y}) = e^{ -\frac{i b \sqrt{2}}{2}  \partial_y - \pi i b  \lambda_2 }\vartheta(e^{y}),\\
 & e^{i b e_3} \cdot \vartheta(e^{ y}) = e^{\frac{i b\sqrt{2}}{2}(- u_2 + u_4)} \cdot \vartheta(e^{ y}) = e^{\frac{i b \sqrt{2}}{2}  \partial_y - \pi i b  \lambda_2 }\vartheta(e^{y}),\\
 & e^{i b e_4}\cdot \vartheta(e^{ y}) = e^{\frac{i b\sqrt{2}}{2}(-u_1 + u_3)} \cdot \vartheta(e^{ y}) = e^{  2 \sqrt{2}\pi  by -\pi i b \lambda_1  }\vartheta(e^{y}),
\end{align*}
with $y \in \ell'$. Thus, we found a representation $\rho_{\mathcal{D},\ell}$ of the quantum torus algebra generated by $\{X_1,X_2,X_3,X_4\}$, where each generator acts on $\mathcal{H}(\ell)$ in the following manner:
\begin{equation}
  \label{eq:torusrep1}
\begin{aligned}
 & \rho_{\mathcal{D},\ell}(X_1) = e^{ - 2 \sqrt{2}\pi  b y - \pi i b\lambda_1 } &&  \rho_{\mathcal{D},\ell}(X_2) = e^{ -\frac{i b \sqrt{2}}{2}  \partial_y - \pi i b  \lambda_2 }, \\
 & \rho_{\mathcal{D},\ell}(X_3)= e^{\frac{i b \sqrt{2}}{2}  \partial_y - \pi i b  \lambda_2 } && \rho_{\mathcal{D},\ell}(X_4) =  e^{  2 \sqrt{2}\pi  by -\pi i b \lambda_1  }.
\end{aligned}
\end{equation}
Note that interchanging the roles of the Lagrangian subspaces $\ell$ and $\ell'$ is equivalent to performing a Fourier transform.

Let us change coordinates
\begin{equation}
\label{eq:changecoordinates}
 x = e^{-2 \sqrt{2} \pi b y}\quad \text{and} \quad \partial_{y} = -2 \sqrt{2} \pi b x \partial_{x},
\end{equation}
  and set
  \[
    \hbar = i \pi b^{2}, \quad T_{1} = e^{-\pi i b \lambda_{1}}, \quad T_{2} = e^{-\pi i b \lambda_{2}}.
    \]
  The representation then reduces to
  \begin{equation}
\begin{aligned}
  \label{eq:torusrep2}
 & \rho_{\mathcal{D},\ell}(X_1) = T_{1}x,  &&  \rho_{\mathcal{D},\ell}(X_2) =T_{2} e^{ 2 \hbar x \partial_{x} } \\
 & \rho_{\mathcal{D},\ell}(X_3)= T_{2}e^{- 2\hbar x \partial_{x} }, && \rho_{\mathcal{D},\ell}(X_4) =  T_{1} x^{-1} .
\end{aligned}
\end{equation}
\end{example}

Having established a representation of the quantum torus algebra associated with a given quantum cluster algebra, the next step is to extend this to a representation of the entire quantum cluster algebra. By this, we mean constructing a family of Hilbert-space representations on which the groupoid of cluster transformations acts via unitary operators. Such representations are built in \cite{goncharov2019quantum}, where a Schwartz space is introduced as a dense subspace:
\[
  \mathcal{S}_{\Sigma, \ell} \subset \mathcal{H}_{\Sigma, \ell}.
\]
Each quantum cluster transformation $\mu_k$ induces a unitary operator
\[
    \mathbf{K}_k \colon \mathcal{H}_{\Sigma',\ell'} \to \mathcal{H}_{\Sigma,\ell}, \qquad \mathbf{K}_k(\mathcal{S}_{\Sigma',\ell'}) = \mathcal{S}_{\Sigma,\ell},
\]
such that for all $ F \in \mathcal{O}_q(\mathcal{X}_{\Sigma}) $ and $ s \in \mathcal{S}_{\Sigma',\ell'} $, one has
\[
      \mathbf{K}_k \circ \rho_{\Sigma',\ell'}(F)(s) = \rho_{\Sigma,\ell}(F) \circ \mathbf{K}_k(s).
\]
Although operators $ \mathbf{K}_k $ play an elaborate role in the representation theory of cluster algebras, in this paper, we focus primarily on mapping quantum cluster variables to the Heisenberg algebra, as in \eqref{eq:torusrep2}.

\section{The $\epsilon$-expanded R-matrix}
\label{sec:perturbed}
\noindent
In this section, we derive the $\epsilon$-expansion of the $R$-matrix. We begin by recalling the cluster realization of $U_q(\mathfrak{sl}_2)$ constructed in \cite{Schrader2019}. We then compute the zeroth-order term from the quantum-group perspective. Next, we repeat the computation from the viewpoint of quantum cluster algebras, which also yields the first-order contribution in~$\epsilon$.

\subsection{Cluster Realization of $U_q(\mathfrak{sl}_{2})$}
\label{sec:clusterrealization}
In \cite{Schrader2019}, Schrader and Shapiro explicitly construct an embedding of $U_{q}(\mathfrak{sl}_{n})$ into a particular quantum torus algebra, generalizing the well-known result of Faddeev \cite{faddeev1999modulardoublequantumgroup} for $U_{q}(\mathfrak{sl}_{2})$. Significantly, the algebra automorphism of $U_{q}(\mathfrak{sl}_{n+1})^{\otimes 2}$ given by conjugation with the $R$-matrix was realized as a sequence of cluster mutations, much like \eqref{eq:roperator}.

Let $q$ be a formal parameter and consider the $\mathbb{C}(q)$-algebra $\mathfrak{D}$ generated by elements
\[
  \{E,F,K,K'\}
  \]
  subject to the relations

\begin{equation}
  \label{eq:sl2relations}
  \begin{aligned}
    KE &= q^2 EK, & KF &= q^{-2} FK, \\
    K' E &= q^{-2} EK', & K' F &= q^2 FK', \\
    K'K &= KK', & [E,F] &= (q-q^{-1})(K - K').
\end{aligned}
\end{equation}
In addition,  let $\mathcal{D}$ be the quantum torus algebra associated to the quiver $\mathcal{D}$ as in Example~\ref{ex:sl2schrodinger}.
\begin{theorem}
  \label{thm:embedding}
  There is an embedding of algebras $\iota \colon \mathfrak{D} \to \mathcal{D}$ defined by the following assignments
\begin{align*}
  E \mapsto i X_{4} (1+q X_{3}), \quad K \mapsto q^{2} X_{4} X_{3} X_{1}, \\
  F \mapsto i X_{1}(1+q X_{2}), \quad K' \mapsto q^{2}X_{1} X_{2}X_{4}.
  \end{align*}
\end{theorem}
\begin{proof}
  See \cite[Thm. 4.4]{Schrader2019}.
\end{proof}

The algebra $\mathfrak{D}$ admits an $h$-formal version, denoted by $\overline{\mathfrak{D}}$. This is an algebra over the ring $\mathbb{C}\llbracket h \rrbracket$ generated by
\[
\{E,F,H,H'\}
\]
subject to the relations \eqref{eq:sl2relations}, where
\[
q = e^{h}, \qquad K = q^{H} = e^{hH}.
\]
The $h$-formal version of the quantum group $U_h(\mathfrak{sl}_2)$ is defined as the quotient of the algebra $\overline{\mathfrak{D}}$ by the relation
\[
  H + H' = 0.
\]
Note that $\mathfrak{D}$ embeds as a topological subalgebra of $\overline{\mathfrak{D}}$. 

Recall that the universal $R$-matrix of $\overline{\mathfrak{D}}$ is an element
\[
\mathcal{R} \in \overline{\mathfrak{D}} \widetilde{\otimes} \overline{\mathfrak{D}},
\]
where $\widetilde{\otimes}$ denotes the completion of the tensor product with respect to the $h$-adic topology.

We also consider the $h$-formal version of the algebra $\mathcal{D}$. Let $\{X_i \mid i=1,\dots,4\}$ be the generators of the quantum torus algebra associated with the $\mathcal{D}$-quiver, and let $(\varepsilon_{ij})$ denote its exchange matrix. Define $\overline{\mathcal{D}}$ to be the algebra over $\mathbb{C}\llbracket h \rrbracket$ generated by $\{x_i \mid i=1,\dots,7\}$ subject to the relations
\[
[x_i,x_j] = -2h \varepsilon_{ij}.
\]

Setting $q=e^{h}$ and mapping
\[
X_i \mapsto e^{x_i}
\]
turns $\mathcal{D}$ into a topological subalgebra of $\overline{\mathcal{D}}$. Under this identification, the embedding of Theorem~\ref{thm:embedding} extends to a homomorphism
\[
\overline{\iota} \colon \overline{\mathfrak{D}} \to \overline{\mathcal{D}}.
\]
Via $\overline{\iota}$, the operator $\mathrm{Ad}_{\mathcal{R}}$ defines an automorphism of $\overline{\mathcal{D}} \otimes \overline{\mathcal{D}}$. To simplify notation, we often omit the symbol $\overline{\iota}$.

Finally, we denote by $P$ the automorphism of $\overline{\mathcal{D}} \otimes \overline{\mathcal{D}}$ that permutes the tensor factors:
\[
P(X \otimes Y) = Y \otimes X.
\]

Recall the quiver $\mathcal{Z}$ and its associated quantum torus algebra, also denoted by $\mathcal{Z}$. It is obtained by amalgamating the quivers $\mathcal{D}$ and $\mathcal{E}$ as described in Example~\ref{ex:amalg}.

  \begin{theorem}
    \label{thm:Rmatrixcluster}
    The composition
    \[
      P \circ \mathrm{Ad}_{\mathcal{R}} \colon \overline{\mathcal{D}} \otimes \overline{\mathcal{D}}\to \overline{\mathcal{D}} \otimes \overline{\mathcal{D}}
    \]
      restricts to an algebra homomorphism on the subalgebra $\mathcal{Z}$ (defined in Example \ref{ex:amalg}). Moreover, the following automorphisms of $\mathcal{Z}$ coincide:
      \[
        P \circ \mathrm{Ad}_{\mathcal{R}} = (\sigma^{*})^{-1} \circ \mu_{4}^{q} \circ \mu_{6}^{q} \circ \mu_{2}^{q} \circ \mu_{4}^{q},
      \]
      where $\sigma$ is a permutation of the quiver returning to the original seed.
  \end{theorem}
  \begin{proof}
    See \cite[Thm. 7.1]{Schrader2019}.
  \end{proof}
  \begin{remark}
In this paper, we adopt a slightly different convention than used in \cite{Schrader2019,goncharov2019quantum}. In these works, the Dehn twist is performed in the opposite direction, leading to the mutation sequence $\mu_{4}^{q} \circ \mu_{3}^{q} \circ \mu_{5}^{q} \circ \mu_{4}^{q}$. To translate this convention to the convention taken in our approach, one must also interchange $E$ with $F$ and $K$ with $K'$. Our convention matches the approach in \cite{hikami2014braiding,hikami2015braids}.
    \end{remark}
  So far, we only considered the interpretation of $U_{q}(\mathfrak{sl}_{2})$ in terms of quantum cluster algebras in a universal manner (i.e., independent of a choice of representation). By mapping the quantum cluster variables to operators arising from the Schrödinger representation (as in Example \ref{ex:sl2schrodinger}), we can expand the $R$-matrix without reference to quantum groups. This is the topic of the next sections.
\subsection{Expanded $R$-matrix: a Quantum Group Perspective}
Let us again consider the representation $\rho_{\mathcal{D},\ell}$ derived in Example \ref{ex:sl2schrodinger}. By Theorem \ref{thm:embedding}, there is an embedding $\iota \colon \mathfrak{D} \to \mathcal{D}$ satisfying
\begin{align*}
  E \mapsto i X_{4} (1+q X_{3}), \quad K \mapsto q^{2} X_{4} X_{3} X_{1}, \\
  F \mapsto i X_{1}(1+q X_{2}), \quad K' \mapsto q^{2}X_{1} X_{2}X_{4}.
  \end{align*}
  With the representation $\rho_{\mathcal{D},\ell}$ at hand, it is straightforward to derive a $U_{q}(\mathfrak{sl}_{2})$-representation, with $q = e^{i \pi b^{2}}$.
\begin{proposition}
  \label{pro:perturbativerep}
Suppose that
\[
T_1^2 = T_2^{-1},
\qquad
iT^{-1/2} = T_1.
\]
Then the representation
\[
\rho_{\mathcal{D},\ell} \colon \mathcal{D} \to \mathrm{End}(\mathcal{S}_{\mathcal{D},\ell})
\]
defined in Example~\ref{ex:sl2schrodinger} is equivalent to the representation $\rho \colon \mathcal{D} \to \mathrm{End}(\mathcal{S}_{\mathcal{D},\ell})$ given by
\begin{align*}
  \rho(X_1) &= i T^{-1} x, &
  \rho(X_2) &= -q T^{2} e^{2\hbar x \partial_{x}}, \\
  \rho(X_3) &= -q^{-1} e^{-2\hbar x \partial_{x}}, &
  \rho(X_4) &= i x^{-1}.
\end{align*}
This induces a $U_q(\mathfrak{sl}_2)$-representation defined by
  \begin{align*}
    (\rho  \circ \iota)(E)  &= -x^{-1} (1 - e^{-2\hbar x \partial_{x}}), \\
    (\rho \circ \iota)(F)  &= -T^{-1} x (1 - q^2 T^{2} e^{2\hbar x \partial_{x}}), \\
    (\rho  \circ \iota)(K)  &= q^{-1} T^{-1} e^{-2\hbar x \partial_{x}}.
  \end{align*}
  with  $b \in \mathbb{R}$, $q = e^{i \pi b^2}$ and $\hbar := i \pi b^2$.
\end{proposition}
\begin{proof}
The representation $\rho_{\mathcal{D},\ell} \colon \mathfrak{D} \to \mathrm{End}(\mathcal{S}_{\mathcal{D}, \ell})$ together with the map $\iota \colon \mathfrak{D} \to \mathcal{D}$ give rise to a $\mathfrak{D}$-representation $\sigma$ defined by
  \begin{align*}
    \sigma(E) &= i T_1 x^{-1} (1 + q T_2 e^{-2\hbar x \partial_{x}}), \quad
    &&\sigma(F) = i T_1 x (1 + q T_2 e^{2\hbar x \partial_{x}}), \\
    \sigma(K) &= T_1^2 T_2 e^{-2\hbar x \partial_{x}}, \quad
    &&\sigma(K') = T_1^2 T_2 e^{2\hbar x \partial_{x}}.
  \end{align*}
  The condition $\sigma(K K') =\pm \mathrm{Id}$ implies $T_1^2 = T_2^{-1}$. By setting $i T^{-1/2} := T_1$, we obtain
  \begin{align*}
    \sigma(E) &= -T^{-\frac1 2} x^{-1} (1 - q T e^{-2\hbar x \partial_{x}}), \\
    \sigma(F) &= -T^{-\frac1 2} x (1 - q T e^{2\hbar x \partial_{x}}), \\
    \sigma(K) &= e^{-2\hbar x \partial_{x}}.
  \end{align*}

  Now assume $T = e^{\hbar t}$ and define the operator
  \[
    S_1 := x^{(-t - 1)/2} = e^{(-t - 1)\log x / 2}.
  \]
  This operator is well-defined under the coordinate transformation in \eqref{eq:changecoordinates}. Then
  \[
    S_1 e^{-2\hbar x \partial_{x}} S_1^{-1} = e^{-2\hbar x \partial_{x} - (t + 1)\hbar} = q^{-1} T^{-1} e^{-2\hbar x \partial_{x}}.
  \]
  Conjugating the representation $\sigma$ with $S_1$ yields
  \begin{align*}
    S_1 \sigma(E) S_1^{-1} &= -T^{-\frac1 2} x^{-1} (1 - e^{-2\hbar x \partial_{x}}), \\
    S_1 \sigma(F) S_1^{-1} &= -T^{-\frac1 2} x (1 - q^2 T^{2} e^{2\hbar x \partial_{x}}), \\
    S_1 \sigma(K) S_1^{-1} &= q^{-1} T^{-1} e^{-2\hbar x \partial_{x}}.
  \end{align*}

  Next, conjugate with
  \[
    S_2 := e^{(-\frac{\hbar t}{2}) x \partial_{x}}.
  \]
  Observe that
  \[
    S_2 x S_2^{-1} = e^{(-\frac{\hbar t}{2}) x \partial_{x}} x e^{-(-\frac{\hbar t}{2}) x \partial_{x}} = T^{-\frac1 2} x.
  \]
  Therefore, further conjugation by $S_2$ implies
  \begin{align*}
    S_2 S_1 \sigma(E) S_1^{-1} S_2^{-1} &= -x^{-1} (1 - e^{-2\hbar x \partial_{x}}), \\
    S_2 S_1 \sigma(F) S_1^{-1} S_2^{-1} &= -T^{-1} x (1 - q^2 T^{2} e^{2\hbar x \partial_{x}}), \\
    S_2 S_1 \sigma(K) S_1^{-1} S_2^{-1} &= q^{-1} T^{-1} e^{-2\hbar x \partial_{x}}.
  \end{align*}

  Set $S := S_2 S_1$. Recall the representation $\rho_{\mathcal{D},\ell}$ given by:
  \begin{align*}
    \rho_{\mathcal{D},\ell}(X_1) &= T_1 x, & \rho_{\mathcal{D},\ell}(X_2) &= T_2 e^{2\hbar x \partial_{x}}, \\
    \rho_{\mathcal{D},\ell}(X_3) &= T_2 e^{-2\hbar x \partial_{x}}, & \rho_{\mathcal{D},\ell}(X_4) &= T_1 x^{-1}.
  \end{align*}
  Then the conjugated representation satisfies
  \begin{align*}
    S \rho_{\mathcal{D},\ell}(X_1) S^{-1} &= i T^{-1} x, & S \rho_{\mathcal{D},\ell}(X_2) S^{-1} &= -q T^{2} e^{2\hbar x \partial_{x}}, \\
    S \rho_{\mathcal{D},\ell}(X_3) S^{-1} &= -q^{-1} e^{-2\hbar x \partial_{x}}, & S \rho_{\mathcal{D},\ell}(X_4) S^{-1} &= i x^{-1},
  \end{align*}
  as claimed.
\end{proof}
\begin{definition}
  \label{def:pertrep}
The representation $\rho \colon \mathcal{D} \to \mathrm{End}(\mathcal{S})$ constructed in Proposition~\ref{pro:perturbativerep} is called the \textit{perturbative representation}.
\end{definition}

\begin{remark}
  \label{rem:positiverep}
It is equally possible to consider the equivalent \textit{positive representations} constructed in \cite{ip2018cluster,ip_2020}. In this line of research, the representation $\mathcal{P}_\lambda$ of split real quantum groups $U_q(\mathfrak{g}_{\mathbb{R}})$ are constructed. By forgetting the real structure of these representations, their generators can be expressed as Laurent polynomials in certain $q$-commuting variables, leading to an embedding of quantum groups into a certain quantum torus.

A change of words of the reduced expression of the longest element $w_0 \in W$ of the Weyl group induces unitary equivalences of the positive representations. These unitary equivalences correspond to quantum cluster mutations of the quantum torus algebra.

As remarked in \cite[Remark 4.6]{Schrader2019}, the positive representations obtained in this way are equivalent to the representations defined in Definition \ref{ex:sl2schrodinger}.
\end{remark}

We introduce a rescaled version of the $ h$-adic algebras considered above. Let $\epsilon$ be an additional parameter and consider the ring $\mathbb{C}_{\epsilon} := \mathbb{C}(\epsilon)$. For any $h$-formal algebra introduced previously, we define its \emph{$\epsilon$-rescaled version} by replacing every occurrence of $h$ by $\epsilon h$. We denote these algebras with a tilde instead of an overline. For instance, the algebra $\widetilde{\mathfrak{D}}$ is the $\mathbb{C}_{\epsilon}\llbracket h \rrbracket$-algebra generated by $\{E,F,H,H'\}$ with the same defining relations as $\overline{\mathfrak{D}}$, except that
\[
q = e^{\epsilon h}, \qquad K = e^{\epsilon h H}, \quad \text{ and } \quad K' = e^{\epsilon h H'}.
\]
Similarly, the algebra $\widetilde{\mathcal{D}}$ is generated by $\{x_i \mid i=1,\dots,4\}$ with relations
\[
[x_i,x_j] = -2\epsilon h \varepsilon_{ij}.
\]
The $R$-matrix $\mathcal{R}$ is also defined with $h$ replaced by $\epsilon h$. The map $\iota \colon \mathfrak{D} \to \mathcal{D}$ extends to a map $\widetilde{\iota} \colon \widetilde{\mathfrak{D}} \to \widetilde{\mathcal{D}}$.

\begin{proposition}
  \label{pro:perthom}

  The perturbative representation gives rise to a homomorphism of algebras
  $\psi \colon \widetilde{\mathcal{D} }\to A_{1}(\mathbb{C}_{\epsilon}(t))[\mathbf{x}^{-1}]$ defined by
  \begin{align*}
    & X_1 \mapsto i T^{-1} \mathbf{x}, && X_2 \mapsto -q T^{2} e^{2\epsilon h \mathbf{x} \mathbf{p}}, \\
    & X_3 \mapsto -q^{-1} e^{-2 \epsilon h \mathbf{x} \mathbf{p}}, && X_4 \mapsto i \mathbf{x}^{-1}
  \end{align*}
  with $T = e^{h t}$ and $q = e^{\epsilon h}$. In particular, this induces a homomorphism $\eta \colon \widetilde{\mathfrak{D}} \to A_1(\mathbb{C}_{\epsilon}(t))$ defined by
  \begin{align*}
   &E \mapsto  -\mathbf{x}^{-1} (1 - e^{-2 \epsilon h \mathbf{x} \mathbf{p}}),\\
   &F \mapsto  -T^{-1} \mathbf{x} (1 - q^2 T^{2} e^{2\epsilon h \mathbf{x} \mathbf{p}}), \\
    &H \mapsto -2 h \epsilon \mathbf{x}\mathbf{p} - th - h \epsilon, \\
 &H’ \mapsto 2 h \epsilon \mathbf{x}\mathbf{p} + th + h \epsilon,
  \end{align*}
  where $\eta = \psi \circ \widetilde{\iota}$.
\end{proposition}
\begin{proof}
  This follows immediately from Proposition \ref{pro:perturbativerep}.
\end{proof}
One benefit of working with the perturbative representation is the fact that the $R$-matrix can be described as a perturbative series in $\epsilon$, as we shall see below.
\begin{theorem}
  \label{thm:rmatrixburau}
  Let $\mathcal{R}$ be the universal $R$-matrix of $\widetilde{\mathfrak{D}}$ and
  \[
    \varphi \colon \mathcal{A}_{n}(\mathbb{C}(t)) \to A_{n}(\mathbb{C}(t))
  \]
  be the monoid homomorphism as in Notation \ref{def:monhom}. The following equality holds
  \[
    (\eta \otimes \eta)(\mathcal{R}) = \alpha \cdot \left( \varphi \left(
\begin{pmatrix}
                                         T & 0 \\
                                         1-T^{2} & T
\end{pmatrix}\right) + O(\epsilon)\right), \qquad T = e^{ht}
\]
for some $\alpha\in \mathbb{C}_\epsilon(t) \llbracket h \rrbracket$.
\end{theorem}
\begin{proof}
  Recall that
  \begin{align*}
    \mathcal{R} &= \sum_{n = 0}^{\infty} \frac{1}{[n]_{q}!} q^{\frac{H \otimes H}{2} + \frac{n(n-1)}{2}} E^{n} \otimes (q-q^{-1})^{-n}F^{n}.
    \end{align*}
 First, we note that
    \begin{align*}
      (\eta \otimes \eta) \left( \frac{E^{n} \otimes F^{n}}{(q-q^{-1})^{n}} \right) &=\left(  \mathbf{x}^{-1} \frac{1- e^{-2 h \epsilon \mathbf{x} \mathbf{p}}}{q-q^{-1}}  \otimes T^{-1} \mathbf{x} (1-q^{2} T^{2} e^{2 h \epsilon \mathbf{x}\mathbf{p} })\right)^{n}\\
     &= (T^{-1} -T)^{n} \mathbf{p}^{n} \otimes \mathbf{x}^{n} + O(\epsilon).
    \end{align*}
    Using the fact that $(\psi \circ \iota)([n]!) = n! + O(\epsilon)$, we find
    \begin{align*}
      (\eta \otimes \eta) \left( \sum_{n = 0}^{\infty} \frac{1}{n!} \frac{E^{n} \otimes F^{n}}{(q-q^{-1})^{n}} \right) &= e^{(T^{-1}-T) \mathbf{x} \otimes \mathbf{p}} + O(\epsilon) \\
                                                                                                                       &= \varphi \left(\begin{pmatrix}
                                                                                                                           1 & 0 \\
                                                                                                                           T^{-1}-T & 1
                                                                                                                           \end{pmatrix}\right) + O(\epsilon).
    \end{align*}
    Moreover,
    \begin{align*}
      (\eta \otimes \eta) \left(q^{\frac{H \otimes H}{2}} \right) &= \exp \left( h \epsilon (2\mathbf{x} \mathbf{p} + t \epsilon^{-1} +1) \otimes (2 \mathbf{x} \mathbf{p} + t \epsilon^{-1} +1 \right)/2) \\
                                                                  &= e^{ht} e^{\frac{h t^{2}\epsilon^{-1}}{2}}\left( \exp \left(h(t \mathbf{x} \mathbf{p} \otimes 1 + 1 \otimes t \mathbf{x} \mathbf{p} ) \right) + O(\epsilon)\right)\\
                                                                  &= e^{ht} e^{\frac{h t^{2}\epsilon^{-1}}{2}}\left( \varphi \left(
                                                                    \begin{pmatrix}
                                                                      T & 0 \\
                                                                      0 & T
                                                                    \end{pmatrix} \right) +O(\epsilon)\right).
      \end{align*}
Using the fact that $\varphi$ is a monoid homomorphism, as was shown in Theorem \ref{cor:identity}, we conclude
  \begin{align*}
    (\eta \otimes \eta)(\mathcal{R}) &= e^{ht} e^{\frac{h t^{2}\epsilon^{-1}}{2}} \left( \varphi \left(
                                                                    \begin{pmatrix}
                                                                      T & 0 \\
                                                                      0 & T
                                                                    \end{pmatrix}\right) 
\varphi \left(\begin{pmatrix}
                                                                      1 &  0 \\
                                                                      T^{-1}-T & 1
                                                                    \end{pmatrix}
                                       \right) +O(\epsilon)\right)\\
                                     &=e^{ht} e^{\frac{h t^{2}\epsilon^{-1}}{2}} \left( \varphi \left(
                                       \begin{pmatrix}
                                         T & 0 \\
                                         1-T^{2} & T
                                        \end{pmatrix}\right) +O(\epsilon) \right) 
    \end{align*}
    as desired.
\end{proof}

One possible approach to building a perturbed $R$-matrix is to carry out a higher-order expansion in $\epsilon$, in the spirit of \cite{overbay2013perturbative}, which relies on repeated applications of the Baker--Campbell--Hausdorff formula. Our approach is different: we expand the $R$-matrix from the perspective of quantum cluster algebras, as discussed in the next section.

We note that $(\eta \otimes \eta)(\mathcal{R})$ is equal (up to a scalar) to the R-matrix of the XC-algebra $\mathcal{A}$ as in \cite[Prop. 4.3]{bosch2025tensorsgaussiansalexanderpolynomial}. It was shown in \cite[Thm. 4.6]{bosch2025tensorsgaussiansalexanderpolynomial} that the universal invariant of a knot related to this $R$-matrix (with auxiliary structure coming from the XC-structure) results in the Alexander polynomial.

\subsection{Expanded $R$-matrix: a Quantum Cluster Algebra Perspective}
\label{sec:quantumclusterper}
Although the quantum group perspective is useful, we aim to work entirely in the language of cluster algebras. Accordingly, we interpret Theorem \ref{thm:pertrmatrix} purely in cluster-algebraic terms. By considering higher orders of $\epsilon$, we derive a perturbed $R$-matrix.

Consider the cluster algebra associated with the following quiver.
\[
  \mathcal{D}=
\vcenter{\hbox{\begin{tikzpicture}
  \node[rectangle, draw, fill=white] (a) at (0,0) {};
  \node[below] at (a.south) {$1$};
  \node[circle, draw, fill=white] (b) at (1,-1) {};
  \node[below] at (b.south) {$2$};
  \node[rectangle, draw, fill=white] (c) at (2,0) {};
  \node[below] at (c.south) {$4$};
  \node[circle, draw, fill=white] (d) at (1,1) {};
  \node[below] at (d.south) {$3$};

  \draw[-{Triangle}, line width=1pt] (a) -- (b);
  \draw[-{Triangle}, line width=1pt] (b) -- (c);
  \draw[-{Triangle}, line width=1pt] (c) -- (d);
  \draw[-{Triangle}, line width=1pt] (d) -- (a);
\end{tikzpicture}}}
\]
Let $\{X_1,\dots,X_4\}$ denote the generators of the quantum torus algebra related to $\widetilde{\mathcal{D}}$ as defined in Section \ref{sec:clusterrealization}. Define
\[
  \mathbf{X}^{\widetilde{\mathcal{D}}} := \psi (X_{1},X_{2},X_{3},X_{4})
  \]
with $\psi$ the homomorphism given in Proposition \ref{pro:perthom}. This yields
\begin{equation}
  \label{eq:clustervar}
\begin{aligned}
 & \mathbf{X}^{\widetilde{\mathcal{D}}}_{1} = i T^{-1} \mathbf{x},  &&  \mathbf{X}^{\widetilde{\mathcal{D}}}_{2} =-q T^{2} e^{2h \epsilon\mathbf{x} \mathbf{p}}, \\
 & \mathbf{X}^{\widetilde{\mathcal{D}}}_{3}= -q^{-1} e^{-2h\epsilon \mathbf{x} \mathbf{p}}, && \mathbf{X}^{\widetilde{\mathcal{D}}}_{4} =  i \mathbf{x}^{-1},
\end{aligned}
\end{equation}
with $q = e^{h \epsilon}$. Now consider the quantum cluster algebra related to the quiver
\[
\mathcal{Z}=\vcenter{\hbox{
\begin{tikzpicture}
  \node[rectangle, draw, fill=white] (a) at (0,0) {};
  \node[below] at (a.south) {$1$};
  \node[circle, draw, fill=white] (b) at (1,-1) {};
  \node[below] at (b.south) {$2$};
  \node[circle, draw, fill=white] (c) at (2,0) {};
  \node[below] at (c.south) {$4$};
  \node[circle, draw, fill=white] (d) at (1,1) {};
  \node[below] at (d.south) {$3$};

  \node[circle, draw, fill=white] (5) at (3,-1) {};
  \node[below] at (5.south) {$5$};
  \node[circle, draw, fill=white] (6) at (3,1) {};
  \node[below] at (6.south) {$6$};
  \node[rectangle, draw, fill=white] (7) at (4,0) {};
  \node[below] at (7.south) {$7$};
  \draw[-{Triangle}, line width=1pt] (a) -- (b);
  \draw[-{Triangle}, line width=1pt] (b) -- (c);
  \draw[-{Triangle}, line width=1pt] (c) -- (d);
  \draw[-{Triangle}, line width=1pt] (d) -- (a);
  \draw[-{Triangle}, line width=1pt] (c) -- (5);
  \draw[-{Triangle}, line width=1pt] (5) -- (7);
  \draw[-{Triangle}, line width=1pt] (7) -- (6);
  \draw[-{Triangle}, line width=1pt] (6) -- (c);
\end{tikzpicture}}}
\]
Recall the embedding of Example \ref{ex:amalg}. By setting
\[
\mathbf{X}^{\mathcal{Z}} = (\psi \otimes \psi)(X_{1} \otimes 1, X_{2} \otimes 1, X_{3} \otimes 1, X_{4} \otimes X_{1}, 1 \otimes X_{2}, 1 \otimes X_{3}, 1 \otimes X_{4}),
\]
we obtain
\begin{equation}
  \label{eq:clustervariables}
\begin{aligned}
 & \mathbf{X}^{\mathcal{Z}}_{1} = i T^{-1} \mathbf{x}_{1},  &&  \mathbf{X}^{\mathcal{Z}}_{2} =-q T^{2} e^{2 h \epsilon \mathbf{x}_{1} \mathbf{p}_{1}}, \\
& \mathbf{X}^{\mathcal{Z}}_{3}= -q^{-1} e^{-2 h \epsilon \mathbf{x}_{1} \mathbf{p}_{1}}, && \mathbf{X}^{\mathcal{Z}}_{4} =  -T^{-1} \mathbf{x}_{1}^{-1} \mathbf{x}_{2}, \\
 & \mathbf{X}^{\mathcal{Z}}_{5} = -q T^{2} e^{2 h \epsilon \mathbf{x}_{2} \mathbf{p}_{2}},  &&  \mathbf{X}^{\mathcal{Z}}_{6} =-q^{-1} e^{-2 h \epsilon \mathbf{x}_{2} \mathbf{p}_{2}}, \\
& \mathbf{X}^{\mathcal{Z}}_{7}= i \mathbf{x}_{2}^{-1},
\end{aligned}
\end{equation}
with $q = e^{h \epsilon}$ and $T = e^{h t}$.

\begin{notation}
  \label{not:R}
  Let $\alpha \in \mathbb{C}_\epsilon(t) \llbracket h \rrbracket$ as in Theorem \ref{thm:rmatrixburau}. Denote
\[
  R := \alpha^{-1}(\eta \otimes \eta)(\mathcal{R}) \quad \text{ and } \quad R_{0} := R |_{\epsilon = 0}.
\]
\end{notation}

\begin{proposition}
  \label{pro:usefulidentity}
  The following two identities hold:
  \begin{enumerate}[\normalfont i)]
    \item\label{item:1} $P(R_{0} \mathbf{x}_1 R_{0}^{-1}) = \left( \mathbf{x}_{1} - \mathbf{x}_{1} T^{2} + \mathbf{x}_{2} T \right);$
          \item
    $ P(R_0 \mathbf{x}_2^{-1} R_0^{-1})= T^{-1}\mathbf{x}_{1}^{-1}$.
  \end{enumerate}
\end{proposition}
\begin{proof}
  From Theorem \ref{thm:Rmatrixcluster}, we know that
  \[
    P \circ \mathrm{Ad}_{R}  = \sigma_{3,5} \circ \sigma_{2,5} \circ \sigma_{3,6} \circ \mu_{4} \circ \mu_{6} \circ \mu_{2} \circ \mu_{4} + O(\epsilon).
    \]
    In particular, this implies
    \begin{align*}
      P(R_{0} \mathbf{X}_{1}^{\mathcal{Z}} R_{0}^{-1})  &= (\sigma_{3,5} \circ \sigma_{2,5} \circ \sigma_{3,6} \circ \mu_{4} \circ \mu_{6} \circ \mu_{2} \circ \mu_{4}) (\mathbf{X}_{1}^{\mathcal{Z}})\\
      &= \mathbf{X}^{\mathcal{Z}}_{1}(1+\mathbf{X}^{\mathcal{Z}}_{2} + \mathbf{X}^{\mathcal{Z}}_{2} \mathbf{X}^{\mathcal{Z}}_{4}),
    \end{align*}
which yields
\begin{align*}
 P(R_{0}  \mathbf{x}_{1} R_{0}^{-1}) = \left( \mathbf{x}_{1} - \mathbf{x}_{1} T^{2} + \mathbf{x}_{2} T \right).
\end{align*}
Moreover, we have
    \begin{align*}
      P(R \mathbf{X}_{7}^{\mathcal{Z}} R^{-1})  &= (\sigma_{3,5} \circ \sigma_{2,5} \circ \sigma_{3,6} \circ \mu_{4} \circ \mu_{6} \circ \mu_{2} \circ \mu_{4}) (\mathbf{X}_{7}^{\mathcal{Z}})\\
      &= (1 + \mathbf{X}_{6}^{\mathcal{Z}} + \mathbf{X}^{\mathcal{Z}}_{4}\mathbf{X}^{\mathcal{Z}}_{6})\mathbf{X}^{\mathcal{Z}}_{7},
    \end{align*}
which implies
\begin{align*}
  P(R_{0}   \mathbf{x}_{2}^{-1} R_{0}^{-1}) =  T^{-1}\mathbf{x}_{1}^{-1}. 
\end{align*}
This proves the proposition.
\end{proof}

The previous proposition is consistent with Theorem \ref{thm:rmatrixburau}, which we summarize in the following corollary.
\begin{corollary}
  \label{cor:clusterR0}
  Let $\mathcal{R}$ be the universal $R$-matrix of $\widetilde{\mathfrak{D}}$ and
  \[
    \varphi \colon \mathcal{A}_{2}(\mathbb{C}(t)) \to A_{2}(\mathbb{C}(t))
  \]
  be the monoid homomorphism as in Notation \ref{def:monhom}. The following equality holds

  \[
    R_{0} =\varphi \left(
\begin{pmatrix}
                                         T &  0 \\
                                         1-T^{2} & T
\end{pmatrix} \right).
\]
\end{corollary}
Proposition \ref{pro:usefulidentity} suggests that cluster algebras contain all the information needed to derive the Alexander polynomial of a given knot from its associated cluster pattern (described in Definition~\ref{def:clusterpattern}). To make this explicit, we introduce a map resulting from the $n$-fold amalgamation of the cluster variables in \eqref{eq:clustervar} at $\epsilon = 0$.
\begin{notation}
Let $n\ge 1$.
Consider the map
\[
\Gamma_n \colon (\mathbb{C}^\times)^n \longrightarrow \mathbb{C}(T)^{3n+1}
\]
defined as follows
\[
(\Gamma_n(\mathbf{x}))_1 = iT^{-1}\mathbf{x}_1,\qquad
(\Gamma_n(\mathbf{x}))_2 = -T^2,\qquad
(\Gamma_n(\mathbf{x}))_3 = -1,
\]
and
\[
(\Gamma_n(\mathbf{x}))_{3n+1} = i \mathbf{x}_n^{-1};
\]
and indices $i\in\{1,\dots,3n+1\}\setminus\{1,2,3,3n+1\}$, define
\[
(\Gamma_n(\mathbf{x}))_i =
\begin{cases}
 -T^{-1} \mathbf{x}_{(i-1)/3}^{-1}\mathbf{x}_{(i+2)/3},
   & \text{if } i\equiv 1 \pmod{3},\\[4pt]
 -T^2,
   & \text{if } i\equiv 2 \pmod{3},\\[4pt]
 -1,
   & \text{if } i\equiv 0 \pmod{3}.
\end{cases}
\]
Note that the map $\Gamma_n$ is injective.
\end{notation}
  Recall the definition of the Burau representation.
Denote by
\[
\rho_n \colon \mathcal{B}_n \longrightarrow \mathrm{GL}_n(\mathbb{Z}[T,T^{-1}])
\]
the \emph{unreduced Burau representation}, determined on generators by
\[
\rho_n(\sigma_i)
=
I_{i-1}\oplus
\begin{pmatrix}
1-T^2 & T^2\\
1 & 0
\end{pmatrix}
\oplus I_{n-i-1},
\qquad i=1,\dots,n-1.
\]
Through conjugation with the matrix $\mathrm{diag}(1,T,T^2,\dots,T^{n-1})$, the Burau representation is equivalent to the representation
\[
\tilde{\rho}_n \colon \mathcal{B}_n \longrightarrow \mathrm{GL}_n(\mathbb{Z}[T,T^{-1}])
\]
defined by
\[
\tilde{\rho}_n(\sigma_i)
=
I_{i-1}\oplus
\begin{pmatrix}
1-T^2 & T\\
T & 0
\end{pmatrix}
\oplus I_{n-i-1},
\qquad i=1,\dots,n-1.
\]

The following proposition connects cluster algebra theory to the Burau representation.
\begin{proposition}
  \label{pro:CGamma}
Let $\mathcal{K}$ be a knot, with braid representative $\beta \in \mathcal{B}_n$. Let $\mathcal{C}$ be as in Definition \ref{def:clusterpattern}. Then
  \[
\mathcal{C}(\beta) \circ \Gamma_n = \Gamma_n \circ \tilde{\rho}_n(\beta).
  \]

\end{proposition}
\begin{proof}
  First, we check the claim on generators. Also, let us start with $ n=2$.  Consider the cluster variables as in \eqref{eq:clustervar}. Through amalgamation and the embedding as in \eqref{eq:embedding}, these extend to $\mathbf{X}^{\mathcal{Z}}$ as in \eqref{eq:clustervariables}. Note that $\mathbf{X}^{\mathcal{Z}}|_{\epsilon = 0} = \Gamma_2(\mathbf{x}_1,\mathbf{x}_2)$.   With Proposition \ref{pro:usefulidentity}, we obtain
  \[
    P(\mathrm{Ad}_{R_0} (\Gamma_2(\mathbf{x}_1,\mathbf{x}_2))) = \Gamma_2(\tilde{\rho}_2(\sigma_1)(\mathbf{x}_1,\mathbf{x}_2)).
  \]
Hence, using Theorem \ref{thm:Rmatrixcluster} we conclude that $\mathcal{C}(\sigma_1)(\Gamma_2(\mathbf{x}_1, \mathbf{x}_2)) = \Gamma_2(\tilde{\rho}_2(\sigma_1)(\mathbf{x}_1,\mathbf{x}_2))$.

This result is not restricted to the Dehn twist in a two-punctured disk. 
Consider again the cluster variables defined in \eqref{eq:clustervar}. 
After $n$-fold amalgamation and applying the embedding in \eqref{eq:embedding}, 
these variables specialize at $\epsilon = 0$ to $\Gamma_n(\mathbf{x})$. 
By a computation analogous to the one above, we obtain
\[
\mathcal{C}(\sigma_i) \circ \Gamma_n 
= 
\Gamma_n \circ \tilde{\rho}_n(\sigma_i), 
\qquad i \in \{1,\dots,n-1\}.
\]

Lastly, note that
\[
\mathcal{C}(\sigma_i \cdot \sigma_j) \circ \Gamma_n 
= \mathcal{C}(\sigma_i) \circ \mathcal{C}(\sigma_j) \circ \Gamma_n =  \mathcal{C}(\sigma_i) \circ \Gamma_n \circ \tilde{\rho}_n(\sigma_j) = \Gamma_n \circ \tilde{\rho}_n( \sigma_i \cdot \sigma_j).
\]
This proves the proposition.
\end{proof}
Recall the result of Example \ref{ex:clustertrefoil}. There we suggested that the classical cluster structure can be used to derive the Alexander polynomial corresponding to a given knot by setting the cluster variables to specific values. This shall be rigorously proven in the following theorem. 
\begin{theorem}
  \label{thm:clusterburau}
Let $\mathcal K$ be an oriented knot represented as the closure of a braid
\[
\beta = \sigma_{k_1}^{\varepsilon_1}\sigma_{k_2}^{\varepsilon_2}\cdots
\sigma_{k_m}^{\varepsilon_m}\ \in\ \mathcal B_n,
\qquad \varepsilon_j\in\{\pm1\}.
\]
Define $\boldsymbol{y}[1] = \Gamma_n(\mathbf{x})$. The condition 
\[
\boldsymbol y[1] = \boldsymbol y[m+1] := \mathcal{C}(\beta) \boldsymbol y[1],
\]
defines a matrix $B(T)$ such that
\[
B(T)(\mathbf x_1,\ldots,\mathbf x_{n}) = 0.
\]

Let $\hat B(T)$ be the $(n-1)\times(n-1)$-matrix obtained from $B(T)$ by deleting
any row and the corresponding column. Then
\[
\Delta_{\mathcal K}(T^2)\ \doteq\ \det \hat B(T),
\]
where $\Delta_{\mathcal K}(T)$ is the Alexander polynomial of $\mathcal K$
and $\doteq$ denotes equality up to multiplication by a unit $\pm T^k$.
\end{theorem}
\begin{proof}
  The condition $\boldsymbol y[1] = \boldsymbol y[m+1]$ reduces to
  \[
    \Gamma_n(\mathbf{x})= \mathcal{C}(\beta) \circ \Gamma_n(\mathbf{x}) = \Gamma_n(\tilde{\rho}_n(\beta )(\mathbf{x}))
  \]
  with the use of Proposition \ref{pro:CGamma}. Using injectivity of $\Gamma_n$, we thus conclude that 
  \[
    (\tilde{\rho}_n(\beta) - \mathbb{I}_n  )\mathbf{x} = 0,
  \]
  where $\mathbb{I}_n$ denotes the $n\times n$ identity matrix. Since $\tilde{\rho}_n$ is equivalent to the Burau representation, the determinant of any minor obtained by deleting a row and a column from $\tilde{\rho}_n(\beta) - \mathbb{I}_n$ computes the Alexander polynomial of $\mathcal{K}$, up to multiplication by a unit $\pm T^k$.
\end{proof}
\begin{example}
  Let $\mathcal{K}$ be the figure-eight knot. It can be represented as the closure of a $3$-strand braid
  \[
  \beta_{\mathcal{K}} =   \sigma_1 \sigma_2^{-1} \sigma_1 \sigma_2^{-1}.
  \]
  Choose the initial variables
\[
  \boldsymbol{y}[1] = (i T^{-1}\mathbf{x}_1, -T^2,-1,- T^{-1}\mathbf{x}_2 \mathbf{x}_1^{-1},-T^2,-1,-T^{-1}\mathbf{x}_3 \mathbf{x}_2^{-1},-T^2,-1,i \mathbf{x}_3^{-1}).
\]
We get $\boldsymbol{y}[5] =  (\stackrel{2}{\mathrm{R}})^{-1}  \stackrel{1}{\mathrm{R}}  (\stackrel{2}{\mathrm{R}})^{-1}   \stackrel{1}{\mathrm{R}} \boldsymbol{y}[1]$,
which yields
\[
\begin{aligned}
\boldsymbol{y}[5] = \bigg(&\frac{i\left(( -1 + T^2 )^2 \mathbf{x}_1 + (T - T^3) \mathbf{x}_2 + \mathbf{x}_3\right)}{T},
-T^2,
-1, \\[4pt]
&\frac{\mathbf{x}_3 - T^2 (\mathbf{x}_1 + \mathbf{x}_3)}{T^4\left(( -1 + T^2 )^2 \mathbf{x}_1 + (T - T^3) \mathbf{x}_2 + \mathbf{x}_3\right)},
-T^2,
-1, \\[4pt]
&-1 + \frac{1}{T^2}
+ \frac{T^2\left(( -1 + T^2 ) \mathbf{x}_1 - T \mathbf{x}_2\right)}{T^2 \mathbf{x}_1 + ( -1 + T^2 ) \mathbf{x}_3},
-T^2,
-1, \\[4pt]
&-\frac{i T^4}{T^2( -1 + T^2 )^2 \mathbf{x}_1 - T^5 \mathbf{x}_2 - ( -1 + T^2 )^2 \mathbf{x}_3} \bigg).
\end{aligned}
\]
See also Appendix \ref{sec:mathematica}. The condition $\boldsymbol{y}[5] = \boldsymbol{y}[1]$ reduces to the equation
\[
\begin{pmatrix}
T^4 - 2T^2 & T - T^3 & 1 \\
T^{-1} & -1 & -T^{-3} + T^{-1} \\
2 - T^{-2} - T^2 & T & T^{-4} - 2T^{-2}
\end{pmatrix} \begin{pmatrix}\mathbf{x}_1 \\ \mathbf{x}_2 \\ \mathbf{x}_3\end{pmatrix} = \begin{pmatrix} 0 \\ 0 \\ 0 \end{pmatrix}.
\]
The determinant of the left upper $2 \times 2$ matrix is equal to $-T^4 + 3 T^2 -1  \doteq \Delta_{\mathcal{K}} (T^2)$.
\end{example}

Thus far, we have limited our discussion to the zeroth-order term and omitted higher-order terms. We now present a framework using quantum cluster algebras that allows us to construct an $\epsilon$-expanded R-matrix.
\begin{lemma}
  \label{lem:someidentities}
The following four identities hold:

\begin{enumerate}[\normalfont i)]
\item $R_{0}^{-1} \mathbf{x}_{1} R_{0} = T^{-1}\mathbf{x}_{1} + (1-T^{-2})\mathbf{x}_{2}$; 
\item $R_{0}^{-1} \mathbf{x}_{2} R_{0} = T^{-1}\mathbf{x}_{2}$; 
\item $R_{0}^{-1} \mathbf{p}_{1} R_{0} = T \mathbf{p}_{1}$; 
\item $R_{0}^{-1} \mathbf{p}_{2} R_{0} = (1-T^{2}) \mathbf{p}_{1} + T \mathbf{p}_{2}$.
\end{enumerate}
\end{lemma}
    \begin{proof}
      This is a direct consequence of Proposition \ref{pro:phixandp}.
    \end{proof}

The quantum cluster algebra perspective now allows for a rather straightforward extraction of perturbative terms. We thus obtain an $\epsilon$-expanded $R$-matrix, which we refer to as a \textit{perturbed $R$-matrix}.

\begin{theorem}

\label{thm:pertrmatrix}
Set $T = e^{ht} \in \mathbb{C}(t) \llbracket h \rrbracket$ and $R$ and $R_0$ be as in Notation \ref{not:R}. . The following equality holds
\begin{equation}
\label{eq:pertrmatrix}
   R=R_{0}(1 + \epsilon \mathcal{N}(f(T,x,p,\epsilon))) 
\end{equation}
with
\begin{align*}
f(T,x,p,\epsilon) &= x_{1}p_{1} + x_{2} p_{2} - (T^{-1}+T) x_{2} p_{1} + 2 x_{1}x_{2}p_{1}p_{2} + (T^{-1} - T) x_{1} x_{2} p_{1}^{2} \\
                        &\quad - (1-T^{-2})x_{2}^{2} p_{1}^{2} -2 T^{-1}x_{2}^{2} p_{1} p_{2} + O(\epsilon).
\end{align*}
\end{theorem}
\begin{proof}

  Let us proceed by finding the first higher-order term. We have
    \begin{align}
      P(R \mathbf{X}_{1}^{\mathcal{Z}} R^{-1})  &=      P(R i T^{-1} \mathbf{x}_{1} R^{-1}) \nonumber \\
                                                              &=(\sigma_{3,5} \circ \sigma_{2,5} \circ \sigma_{3,6} \circ \mu_{4}^{q} \circ \mu_{6}^{q} \circ \mu_{2}^{q} \circ \mu_{4}^{q}) (\mathbf{X}_{1}^{\mathcal{Z}}) \nonumber  \\
      &= \mathbf{X}^{\mathcal{Z}}_{1}(1+q \mathbf{X}^{\mathcal{Z}}_{2}(1 + q  \mathbf{X}^{\mathcal{Z}}_{4})) \nonumber \\
      &= i T^{-1} \mathbf{x}_{1}(1 - q^{2}T^{2} q^{2 \mathbf{x}_{1} \mathbf{p}_{1}}(1-T^{-1} q \mathbf{x}_{1}^{-1}\mathbf{x}_{2})) \nonumber \\
                                                              &= \mathcal{N}(i T^{-1} x_{1}(1 - q^{2}T^{2} (1-T^{-1} q^{-1} x_{1}^{-1}x_{2}))q^{2 x_{1} p_{1}}) \nonumber \\
                                                              &= \mathcal{N}\bigg(i T^{-1}(x_{1} - T^{2}x_{1} + T x_{2}) \nonumber \\
      &+ i T^{-1} (-2 T^{2} x_{1} - 2 T^{2}x_{1}^{2}p_{1} + T x_{2} + 2 T x_{1}x_{2}p_{1})\epsilon h\bigg) + O(\epsilon^2) \label{eq:rx1r}.
    \end{align}
    Moreover, we have
   \begin{align*}
      P(R \mathbf{X}_{7}^{\mathcal{Z}} R^{-1})  &=      P(R i  \mathbf{x}_{2}^{-1} R^{-1}) \\
                                                              &=(\sigma_{3,5} \circ \sigma_{2,5} \circ \sigma_{3,6} \circ \mu_{4}^{q} \circ \mu_{6}^{q} \circ \mu_{2}^{q} \circ \mu_{4}^{q}) (\mathbf{X}_{7}^{\mathcal{Z}})  \\
      &= (1+q(1+q \mathbf{X}^{\mathcal{Z}}_{4})\mathbf{X}_{6}^{\mathcal{Z}})\mathbf{X}^{\mathcal{Z}}_{7} \\
      &= i \mathbf{x}_{2}^{-1}(1- q^{-2 \mathbf{x}_{2} \mathbf{p}_{2}} (1-T^{-1}q \mathbf{x}_{1}^{-1} \mathbf{x}_{2})) \\
                                                              &= i \mathbf{x}_{2}^{-1}(1-  (1-T^{-1}q^{-1} \mathbf{x}_{1}^{-1} \mathbf{x}_{2})q^{-2 \mathbf{x}_{2} \mathbf{p}_{2}}) \\
     &= i T^{-1} \mathbf{x}_{1}^{-1} (1+(-1 + 2 T \mathbf{x}_{1} \mathbf{p}_{2} -2 \mathbf{x}_{2}\mathbf{p}_{2}) \epsilon h) + O(\epsilon^2).
\end{align*}
This implies
   \begin{align*}
     P(R   \mathbf{x}_{2} R^{-1}) &=       T \mathbf{x}_{1} (1-(-1 + 2 T \mathbf{x}_{1} \mathbf{p}_{2} -2 \mathbf{x}_{2}\mathbf{p}_{2})\epsilon h)  + O(\epsilon^{2}) \\
     &= T \mathbf{x}_{1} + (  T \mathbf{x}_{1} -2  T^{2} \mathbf{x}_{1}^{2} \mathbf{p}_{2} + 2  T \mathbf{x}_{1} \mathbf{x}_{2} \mathbf{p}_{2}) h \epsilon +O(\epsilon^{2}).
\end{align*}
We make the following ansatz for the form of $R$:
\[
  R = R_{0} \mathrm{exp}( \mathcal{N}(\epsilon f(x,p,\epsilon))).
  \]
  Using Proposition \ref{pro:perturbedexptill2}, it is clear that
  \[
    R \mathbf{x}_{j}R^{-1} = \mathbf{x}_{i} A_{ij} + \epsilon \mathrm{Ad}_{R_{0}} \mathcal{N}(\partial_{p_{j}} f(x,p)) +O(\epsilon^{2}).
    \]
    This formula shall be used to find $f(x,p,\epsilon) + O (\epsilon^{2})$.

    From \eqref{eq:rx1r}, it follows that
    \begin{align*}
      R \mathbf{x}_{1}R ^{-1} &= \mathcal{N}\bigg((x_{2} - T^{2}x_{2} + T x_{1}) \\
      &+  (-2 T^{2} x_{2} - 2 T^{2}x_{2}^{2}p_{2} + T x_{1} + 2 T x_{1}x_{2}p_{2})h\epsilon \bigg) +O(\epsilon^{2}).
    \end{align*}
With the aid of Lemma \ref{lem:someidentities}, we conclude that
    \begin{align*}
     R_{0}^{-1} R \mathbf{x}_{1}R ^{-1}R_{0} &= \mathcal{N}\bigg(x_{1} +  (x_{1} - T^{-1} x_{2} -T x_{2}+ 2 x_{1}x_{2}p_{2} \\
                        &+ 2(T^{-1}-T)x_{1}x_{2}p_{1} + 2(1-T^{-2})x_{2}^{2}p_{1} -2 T^{-1} x_{2}^{2}p_{2})h\epsilon \bigg) +O(\epsilon^{2}),
    \end{align*}
    which, with Proposition \ref{pro:perturbedexptill2} implies
    \begin{align*}
      \partial_{p_{1}} f(x,p,\epsilon) &=  (x_{1} - T^{-1} x_{2} -T x_{2}+ 2 x_{1}x_{2}p_{2} \\
                        &+ 2(T^{-1}-T)x_{1}x_{2}p_{1} + 2(1-T^{-2})x_{2}^{2}p_{1} -2 T^{-1} x_{2}^{2}p_{2})h +O(\epsilon).
    \end{align*}
    Similarly, using Lemma \ref{lem:expconj}
    \begin{align*}
     R_{0}^{-1} R x_{2}R ^{-1}R_{0} &= \mathcal{N}\bigg(x_{2} +( x_{2} + x_{1}x_{2} p_{1} -2 T^{-1}x_{2}^{2} p_{1}  )h\epsilon \bigg)+O(\epsilon^{2}),
    \end{align*}
    which, with Proposition \ref{pro:perturbedexptill2} implies
    \begin{equation}
      \partial_{p_{2}} f(x,p,\epsilon) = (x_{2} + 2 x_{1}x_{2} p_{1} - 2 T^{-1}x_{2}^{2} p_{1}) h +O(\epsilon).
      \end{equation}
      Thus, through integration, we conclude that
      \begin{align*}
        f(x,p,\epsilon) &= x_{1}p_{1} + x_{2} p_{2} - (T^{-1}+T) x_{2} p_{1} + 2 x_{1}x_{2}p_{1}p_{2} + (T^{-1} - T) x_{1} x_{2} p_{1}^{2} \\
                        &\quad - (1-T^{-2})x_{2}^{2} p_{1}^{2} -2 T^{-1}x_{2}^{2} p_{1} p_{2} + O(\epsilon),
      \end{align*}
      as claimed.
\end{proof}
Note that \eqref{thm:pertrmatrix}, the expression is written as
$\mathcal{N}(\dots) \times \mathcal{N}(\dots)$. To transform this into an expression of the form $\mathcal{N}(\dots)$, Proposition \ref{pro:usefulidentity} should be invoked.

We can now obtain a knot invariant in line with the construction in \cite{bar2021perturbed,bosch2025tensorsgaussiansalexanderpolynomial}. We will expand on this view in the forthcoming section. To carry out explicit computations in \emph{Mathematica}, all expressions must be written in normal-ordered form. Using the $R$-matrix in \eqref{eq:pertrmatrix}, we can perform computations up to first order in $\ epsilon$. In the \emph{Mathematica} code provided in Appendix~\ref{sec:mathematica}, a perturbed $R $-matrix up to second order in  $\epsilon^2 $ is derived. Higher-order terms can be obtained by applying the same methods described above.

\subsection{Perturbed Knot Invariants}
A standard approach for constructing knot invariants utilizes ribbon Hopf algebras~\cite{ReshetikhinTuraev1990,TuraevQI}. Informally, a \emph{ribbon Hopf algebra} is a Hopf algebra $A$ equipped with
distinguished invertible elements
\[
R=\sum_i \alpha_i\otimes \beta_i \in A\otimes A,\qquad \kappa\in A^\times,
\]
which turns the tensor category $\mathrm{fMod}_A$ of finite-dimensional $A$-modules into a ribbon  category. Here $\kappa$ denotes the balancing element.

It is equally possible to construct knot invariants universally. This is described in detail in \cite[Section 4.1]{ohtsuki2002quantum}. We recall the essentials, using a slightly different convention.

Let $(A,R,\kappa)$ be a ribbon Hopf algebra. Consider a long, framed, oriented knot $\mathcal{K}$ expressed within a diagram in which every strand of every crossing is locally oriented upward. Denote $R=\sum_i \alpha_i\otimes \beta_i \in A\otimes A$ and $R^{-1} = \sum \overline{\alpha}_i\otimes \overline{\beta}_i \in A\otimes A$. To each generator of the knot diagram, we place beads in accordance with the following rules:
\begin{equation*}
\centre{
\centering
\includegraphics[width=0.8\textwidth]{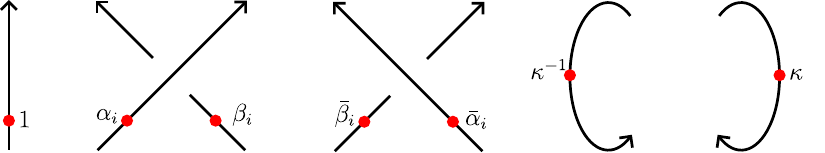}}
\end{equation*}

Traversing the knot once in the given orientation and multiplying the bead labels from right to left (and summing over the indices) produces an element
\[
Z_A(\mathcal{K})\in A,
\]
the \textit{universal invariant} associated to $\mathcal{K}$.

\begin{example}
  Let $\mathcal{K}$ be the trefoil as illustrated below. We have already decorated the crossings and closures as described above.
\begin{figure}[H]
\centering
\begin{tikzpicture}[scale=1]
  \node[anchor=south west, inner sep=0] (img) at (0,0)
    {\includegraphics[scale=0.7]{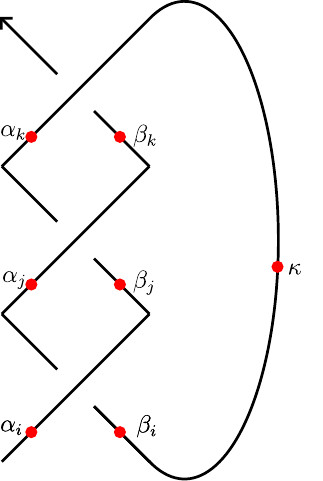}};

\end{tikzpicture}
\end{figure}
The universal invariant, associated with the trefoil, is given by
\[
  Z(\mathcal{K})= \sum_{i,j,k}  \beta_k\alpha_j \beta_i\kappa \alpha_k\beta_{j }\alpha_i.
  \]
\end{example}
To obtain a knot invariant, one does not need a full ribbon Hopf algebra structure on an algebra $A$. Instead, Becerra showed in \cite{becerra2025refinedfunctorialuniversaltangle} that it suffices to consider an algebra structure on $A$ along with specific relations between the elements $R$ and $\kappa$ that ensure the invariance of $Z_A(\mathcal{K})$ under changes of the knot diagram that do not alter the isotopy type of the knot $\mathcal{K}$ it represents. Equivalently, elements $R$ and $\kappa$ satisfy relations that force $Z_A(\mathcal{K})$ to be invariant under the framed Reidemeister moves. This more general structure, given by the triple $(A, R, \kappa)$, is called an \textit{XC-algebra}. Formally, it is defined as follows.
\begin{definition}
Let $ \bk $ be a commutative ring with a unit, and let $(A, \mu, 1) $ be a $ \bk $-algebra. An \emph{$XC $-structure} on $ A $ consists of two invertible elements,
\[
    R \in A \otimes A, \qquad \kappa \in A,
\]
called the \emph{universal $R $-matrix} and the \emph{balancing element}, respectively, which must satisfy the following conditions:
\begin{enumerate}
    \item\label{item:1} $ R = (\kappa \otimes \kappa) \cdot R \cdot (\kappa^{-1} \otimes \kappa^{-1}) $,
  \item $\mu^{[3]}(R_{13} \cdot \kappa_2) = \mu^{[3]}(R_{31} \cdot \kappa_2^{-1}) $,
        \item $1 \otimes \kappa^{-1} = (\mu \otimes \mu^{[3]})(R_{15} \cdot R_{23}^{-1} \cdot \kappa_{4}^{-1}),$
    \item $ \kappa \otimes 1 = (\mu^{[3]} \otimes \mu)(R_{34}^{-1} \cdot R_{15} \cdot \kappa_2) $,
    \item $R_{12} R_{13} R_{23} = R_{23} R_{13} R_{12} $,
\end{enumerate}
where $ \mu^{[3]} $ denotes the three-fold multiplication map. For indices $1 \leq i, j \leq n $ with $ i \neq j $, we define
\[
    R_{ij} := \begin{cases}
        (1^{\otimes (i-1)} \otimes \mathrm{Id} \otimes 1^{\otimes (j-i-1)} \otimes \mathrm{Id} \otimes 1^{\otimes (n-j)})(R^{\pm 1}), & i > j, \\
        (1^{\otimes (j-1)} \otimes \mathrm{Id} \otimes 1^{\otimes (i-j-1)} \otimes \mathrm{Id} \otimes 1^{\otimes (n-i)})(\mathrm{flip}_{A,A}R^{\pm 1}), & j > i.
    \end{cases}
\]
Similarly, we write $\kappa_i^{\pm 1} = (1^{\otimes (i-1)} \otimes \mathrm{Id} \otimes 1^{\otimes (n-i)})(\kappa^{\pm 1}) $.

A triple $ (A, R, \kappa) $ consisting of a $\bk $-algebra equipped with an $ XC $-structure is called an \emph{$XC $-algebra}.
\end{definition}

\begin{proposition}
  \label{pro:perturbedxc}
  Let $T = e^{ht}$. The triple
\[
  \mathscr{A} = (A_{1}(\mathbb{C}_\epsilon),R_0(1+ \epsilon \mathcal{N}(f(x,p,\epsilon))),T^{-1} \mathcal{N}(e^{x p e^{-2 h \epsilon-h} - h \epsilon} ))
\]
with
  \[
  R_0 = \varphi \left(\begin{pmatrix} T & 0 \\ 1-T^{2} & T \end{pmatrix} \right)
\]
and
\begin{align*}
f(x,p,\epsilon) &= x_{1}p_{1} + x_{2} p_{2} - (T^{-1}+T) x_{2} p_{1} + 2 x_{1}x_{2}p_{1}p_{2} + (T^{-1} - T) x_{1} x_{2} p_{1}^{2} \\
                        &\quad - (1-T^{-2})x_{2}^{2} p_{1}^{2} -2 T^{-1}x_{2}^{2} p_{1} p_{2} + O(\epsilon)
\end{align*}
is an XC-algebra.
\end{proposition}
\begin{proof}
First, we note that the triple $(U_h(\mathfrak{sl}_2),\mathcal{R}, q^H)$ with $q = e^{\epsilon h}$ is an XC-algebra, as shown in \cite[Prop. 4.4]{becerra2025refinedfunctorialuniversaltangle}. In Proposition \ref{pro:perthom}, we derived a homomorphism $\eta \colon \widetilde{\mathfrak{D}} \to A_1(\mathbb{C}_\epsilon(t))$. Note that this map descends to a homomorphism $\eta ' \colon U_h(\mathfrak{sl}_2) \to A_1(\mathbb{C}_\epsilon(t))$ with $\eta'(\mathcal{R}) =\alpha R$ (with $\alpha \in \mathbb{C}_\epsilon \llbracket h \rrbracket$ as in Theorem \ref{thm:rmatrixburau}) and $\eta'(q^{H}) =T^{-1} \mathcal{N}(e^{x p e^{-2 h \epsilon-h} - h \epsilon})$. We may rescale $R$ by $\alpha^{-1}$, since the XC-structure condition is preserved under this rescaling. With Theorem \ref{thm:pertrmatrix}, we obtain the desired result.
\end{proof}
\begin{example}
 Consider the $0$-framed trefoil, presented as a long knot, and denote it by $\mathcal{K}$. Let $\mathscr{A}$ be the XC-algebra from Proposition \ref{pro:perturbedxc}.

We now define tensor multiplication. For a more detailed discussion, see \cite{bosch2025tensorsgaussiansalexanderpolynomial,bar2021perturbed}.
Throughout this section, let $A$ be a unital
associative $\bk$-algebra with unit map $u\colon \bk \to A$, and let $I$ be a
finite set. Let $k \in I$, and let $u\colon \bk \to A$ be the unit map. Using the identification
\[
A^{\otimes(I\setminus\{k\})}\cong A^{\otimes(I\setminus\{k\})}\otimes \bk,
\]
define
\[
\eta_k\colon A^{\otimes(I\setminus\{k\})}\longrightarrow A^{\otimes I}
\]
to be the composite
\[
A^{\otimes(I\setminus\{k\})}
\cong A^{\otimes(I\setminus\{k\})}\otimes \bk
\xrightarrow{\ \mathrm{id}\otimes u\ }
A^{\otimes(I\setminus\{k\})}\otimes A
\longrightarrow A^{\otimes I}.
\]
Set
\[
V_k:=\operatorname{Im}(\eta_k)\subset A^{\otimes I}.
\]
Thus, $V_k$ consists of tensors with a trivial $k$-leg.

Let $i,j,k\in I$ be pairwise distinct. The \textit{multiplication along $(i,j)$ into $k$} is the $\bk$-linear map
\[
\mathbf{m}_{i,j \rightarrow k} \colon V_k \longrightarrow A^{\otimes I}
\]
defined by
\begin{equation}\label{eq:contraction-rule}
\mathbf{m}_{i,j\rightarrow k}\big((a)_i(b)_jx\big) := (ab)_kx
\qquad a,b\in A, \text{ and  }  x \text{ supported on }I\setminus\{i,j,k\},
\end{equation}
and extended linearly. In other words: for every elementary tensor whose $k$--leg equals the unit, the map replaces the elements in the $i$-- and $j$--legs with the unit and places their product in the $k$--leg of the output.

In what follows, we will often write composition as
\[
f \circ g = g \sslash f.
\]

Since we consider the $0$-framed trefoil, we need ribbon elements to remove the framing; these are of the form
\begin{align*}
   v_i &=  R^{-1}_{13} \kappa^{-1}_2 \sslash \mathbf{m}_{2,1 \to \bar{1}}\sslash \mathbf{m}_{3,\bar{1} \to i}, \\
   v_i^{-1} &=  R_{13} \kappa_2 \sslash \mathbf{m}_{2,\bar{1} \to 1}\sslash \mathbf{m}_{3,\bar{1} \to i}.
\end{align*}
Define
\[
Z:= R_{15}R_{62}R_{37} \kappa_4 v_8 v_9 v_{10} \in A_1(\mathbb{C}_\epsilon(t))^{\otimes 10}.
\]
Then
\[
 \mathbf{Z}_{\mathscr{A}}(\mathcal{K})
 = Z \sslash \mathbf{m}_{2,1 \to \bar{1}}
 \sslash \mathbf{m}_{3,\bar{1} \to 1}
 \sslash \mathbf{m}_{4,1 \to \bar{1}}
 \sslash \dots
 \sslash \mathbf{m}_{10,1 \to \bar{1}}.
\]
This computation was carried out explicitly in Mathematica; see Appendix \ref{sec:mathematica}.

\end{example}


\section{Towards $\mathfrak{sl}_{3}$}
\label{sec:towsl3}
\noindent
In theory, the method of finding a perturbed $R$-matrix described in the previous sections can be extended to other Lie algebras. In practice, this comes with some difficulties. First, one has to choose an appropriate initial seed related to a cluster algebra. One can find this through trial and error by performing mutations. Another difficulty is finding a good basis for the Schrödinger representation related to the chosen initial seed. Here ``good'' means: finding a representation for which an expansion of the universal R-matrix can be performed in $\epsilon$ that does not involve inverses of generators of the Weyl--Heisenberg algebra. In this section, we describe a possible direction for $\mathfrak{sl}_{3}$. Let us start by recalling the definition of $U_{q}(\mathfrak{sl}_{3})$.
Define
\[
A=(a_{ij})_{i,j\in\{1,2\}}=\begin{pmatrix}2&-1\\-1&2\end{pmatrix}.
\]
Let $\mathfrak{D}_{\mathfrak{sl}_{3}}$ be the $\mathbb C(q)$-algebra generated by
$\{E_{1},E_{2},F_{1},F_{2},K_{1}, \widetilde{K}_{1}, K_{2}, \widetilde{K}_{2}\}$, where the $K$’s are invertible, subject to the relations
\begin{gather}
  K_{i} K_{j}=K_{j} K_{i}, \quad \widetilde{K}_{i} \widetilde{K}_{j}=\widetilde{K}_{j} \widetilde{K}_{i}, \quad \widetilde{K}_{i} K_{j}=K_{j} \widetilde{K}_{i}, \nonumber\\
  \left\{\begin{array}{ll}K_{i} E_{j}=q^{a_{i j}} E_{j} K_{i}, & \widetilde{K}_{i} E_{j}=q^{-a_{i j}} E_{j} \widetilde{K}_{i} \\ K_{i} F_{j}=q^{-a_{i j}} F_{j} K_{i}, & \widetilde{K}_{i} F_{j}=q^{a_{i j}} F_{j} \widetilde{K}_{i},\end{array}\right. \nonumber\\
  E_{i} F_{j}-F_{j} E_{i}=\delta_{i j}\left(q-q^{-1}\right)\left(K_{i}-\widetilde{K}_{i}\right), \nonumber\\
  E_{i}^{2} E_{j}-\left(q+q^{-1}\right) E_{i} E_{j} E_{i}+E_{j} E_{i}^{2}=0, \quad \text{if } a_{i j}=-1, \nonumber\\
  F_{i}^{2} F_{j}-\left(q+q^{-1}\right) F_{i} F_{j} F_{i}+F_{j} F_{i}^{2}=0, \quad \text{if } a_{i j}=-1. \label{eq:qg}
\end{gather}

As with $U_{q}(\mathfrak{sl}_{2})$, this algebra has an $h$-formal version. Let $\overline{\mathfrak{D}}_{\mathfrak{sl}_{3}}$ be the algebra over the ring $\mathbb{C} \llbracket h \rrbracket$ defined by generators
\[
  \{ E_{i}, F_{i}, H_{i}, H'_i \mid i = 1, \dots, n \}
  \]
  subject to the relations above, where
  \[
    q = e^{h} \text{ and } K_{i} = q^{H_{i}} = e^{h H_{i}}.
    \]
    The algebra $U_{h}(\mathfrak{sl}_{3})$ of the quantum group is defined as the quotient of the associative algebra $\overline{\mathfrak{D}}_{\mathfrak{sl}_{3}}$ by the relations
    \[
      H_{i} + H_{i}' = 0 \quad \text{ for } \quad i = 1, \dots, n.
      \]
      Define
      \[
        e_{i} := (q-q^{-1})^{-1} E_{i}, \quad f_{i} := (q-q^{-1})^{-1}, F_{i} \quad \text{and} \quad h_{i} := H_{i}.
        \]
      We choose a normal ordering of positive roots $\alpha_{1} \prec \alpha_{1} + \alpha_{2} \prec \alpha_{2}$ with corresponding root vectors
    \[
e_{12}:=e_{\alpha_1+\alpha_2}=e_1e_2-q^{-1}e_2e_1,\qquad
f_{12}:=f_{\alpha_1+\alpha_2}=f_2f_1-q f_1f_2.
\]
Define the $q$-exponential by
\[
\exp_q(x):=\sum_{n=0}^\infty \frac{q^{\frac{n(n-1)}{2}}}{[n]_q!} x^n,
\qquad [n]_q!:=\prod_{k=1}^n \frac{q^k-q^{-k}}{q-q^{-1}} .
\]
      The corresponding universal $R$-matrix of $\mathfrak{D}_{\mathfrak{sl}_{3}}$ is of the following form
      \begin{equation}
      \label{eq:sl3rmatrix}
\mathcal R_{\mathfrak{sl}_3}
=
\left(\prod_{\beta\in\Phi^+}^{\rightarrow}
\exp_{q} \big((q-q^{-1})e_\beta\otimes f_\beta\big)\right)
q^{\sum_{i,j=1}^2 (A^{-1})_{ij}h_i\otimes h_j}
      \end{equation}
where $\Phi^+=\{\alpha_1,\alpha_1+\alpha_2,\alpha_2\}$ is the set of positive roots of $\mathfrak{sl}_3$ and the arrow on the product indicates that it is taken with respect to the chosen order
\[
\alpha_1 \prec \alpha_1+\alpha_2 \prec \alpha_2,
\qquad
\prod_{\beta\in\Phi^+}^{\rightarrow} X_\beta
=
X_{\alpha_1}X_{\alpha_1+\alpha_2}X_{\alpha_2}.
\]

We shall now introduce the corresponding quantum cluster algebra. Let $\mathcal{D}_{\mathfrak{sl}_{3}}$ be the quantum torus algebra related to the following quiver:
\[
\vcenter{\hbox{\begin{tikzpicture}[
  >=Triangle,
  circ/.style={circle,draw,fill=white},
  box/.style={rectangle,draw,fill=white},
  arr/.style={-{Triangle}, line width=1pt} 
]

\node[circ] (E1) at (0,  2.10) {};
\node[below] at (E1.south) {$9$};
\node[circ] (E2) at (0,  0.85) {};
\node[below] at (E2.south) {$10$};
\node[circ] (L)  at (-1.15, 0.00) {};
\node[below] at (L.south) {$5$};
\node[circ] (R)  at ( 1.15, 0.00) {};
\node[below] at (R.south) {$7$};
\node[circ] (F1) at (0, -0.70) {};
\node[below] at (F1.south) {$6$};
\node[circ] (B)  at (0, -2.05) {};
\node[below] at (B.south) {$2$};

\node[box] (UL) at (-2.35, 0.85) {};
\node[left] at (UL.west) {$4$};
\node[box] (LL) at (-2.35,-0.35) {};
\node[below] at (LL.south) {$1$};
\node[box] (UR) at ( 2.35, 0.85) {};
\node[right] at (UR.east) {$8$};
\node[box] (LR) at ( 2.35,-0.35) {};
\node[below] at (LR.south) {$3$};


\draw[arr] (E1) -- (UL);
\draw[arr, dashed] (LL) -- (UL);
\draw[arr] (UL) -- (L);

\draw[arr] (L) -- (LL);

\draw[arr] (LL) -- (B);
\draw[arr] (B) -- (L);

\draw[arr] (L) -- (E1);
\draw[arr] (E1) -- (R);
\draw[arr] (R) -- (UR);
\draw[arr] (UR) -- (E1);
\draw[arr, dashed] (UR) -- (LR);
\draw[arr] (B) -- (LR);
\draw[arr] (LR) -- (R);

\draw[arr] (L)  -- (F1);
\draw[arr] (F1) -- (R);
\draw[arr] (R)  -- (B);

\draw[arr] (R)  -- (E2);
\draw[arr] (E2) -- (L);

\end{tikzpicture}}}
\]

Following \cite{Schrader2019}, there exists a map $\iota \colon \mathfrak{D}_{\mathfrak{sl}_{3}} \to \mathcal{D}_{\mathfrak{sl}_{3}}$ given by
\begin{align*}
  &E_{1} \mapsto i X_{8} (1+qX_{9}),  && K_{1} \mapsto q^{2} X_{8}X_{9}X_{4}, \\
  &E_{2} \mapsto i X_{3}(1+q X_{7}(1+ q X_{10}(1+ q X_{5} ))), && K_{2} \mapsto q^{4}X_{3}X_{7}X_{10}X_{5}X_{1}, \\
  &F_{1} \mapsto i X_{4}(1 + q X_{5}(1+ q X_{6}(1+ qX_{7}))), && \widetilde{K}_{1} \mapsto q^{4} X_{4}X_{5}X_{6}X_{7}X_{8}, \\
  &F_{2} \mapsto i X_{1} (1+ q X_{2}), && \widetilde{K}_{2} \mapsto q^{2} X_1 X_{2} X_{3}.
\end{align*}
Although it is possible to construct a representation associated to this quantum torus algebra using the Weyl–Heisenberg representation, we have not found a “good” choice of basis for this representation. A better strategy is to perform a sequence of mutations for which such a “good” representation \textit{can} be obtained. Indeed, let $\mathcal{D}'_{\mathfrak{sl}_{3}}$ be the quantum torus algebra related to the following quiver:
\[
\vcenter{\hbox{\begin{tikzpicture}[
  >=Triangle,
  circ/.style={circle,draw,fill=white},
  box/.style={rectangle,draw,fill=white},
  arr/.style={-{Triangle}, line width=1pt, shorten >=2.5pt, shorten <=2.5pt}
]

\node[circ] (E1) at (0,  2.10) {};
\node[above] at (E1.north) {$9$};
\node[circ] (E2) at (0,  0.85) {};
\node[below, xshift=-1mm, yshift=-0.5mm] at (E2.west) {$10$};
\node[circ] (L)  at (-1.15, 0.00) {};
\node[below] at (L.south) {$5$};
\node[circ] (R)  at ( 1.15, 0.00) {};
\node[above] at (R.east) {$7$};
\node[circ] (F1) at (0, -0.70) {};
\node[below] at (F1.east) {$6$};
\node[circ] (B)  at (0, -2.05) {};
\node[below] at (B.south) {$2$};

\node[box] (UL) at (-2.35, 0.85) {};
\node[left] at (UL.west) {$4$};
\node[box] (LL) at (-2.35,-0.35) {};
\node[below] at (LL.south) {$1$};
\node[box] (UR) at ( 2.35, 0.85) {};
\node[right] at (UR.east) {$8$};
\node[box] (LR) at ( 2.35,-0.35) {};
\node[below] at (LR.south) {$3$};


\draw[arr]         (LL) -- (B);          
\draw[arr, dashed] (LL) -- (UL);         

\draw[arr] (B) -- (F1);                  

\draw[arr, dashed] (LR) -- (UR);         
\draw[arr]         (LR) to[bend left=-30] (E2); 

\draw[arr] (UL) -- (L);                  

\draw[arr] (L) -- (LL);                  
\draw[arr] (L) to[bend left=10] (UR);    
\draw[arr] (L) -- (E1);                  

\draw[arr] (F1) -- (L);                  
\draw[arr] (F1) -- (R);                  

\draw[arr] (R) -- (LR);                  
\draw[arr] (R) -- (E1);                  

\draw[arr] (UR) to[out=-110,in=20,looseness=1.15] (F1); 

\draw[arr] (E1) to[out=-110,in=110,looseness=1.25] (B); 
\draw[arr] (E1) -- (UL);                 
\draw[arr] (E1) -- (E2);                 

\draw[arr] (E2) -- (F1);                 

\end{tikzpicture}}}
\]
Now replace every occurrence of $h$ with $h\epsilon$ in algebras $\mathfrak{D}_{\mathfrak{sl}_3}$, $\mathcal{D}_{\mathfrak{sl}_3}$ and $\mathcal{D}'_{\mathfrak{sl}_3}$ (defined over $\mathbb{C}_{\epsilon}\llbracket h \rrbracket$) as in the previous sections of this paper.
The generators of the related quantum torus algebra are denoted by $\{Y_{1},Y_{2},\dots,Y_{10}\}$. The quantum torus algebra $\mathcal{D}'_{\mathfrak{sl}_{3}}$ is related to $\mathcal{D}_{\mathfrak{sl}_{3}}$ via mutations $\mu_{6} \circ \mu_{7}$. The map $\iota$ together with the above mutations defined in \eqref{eq:clasmut} yields a map \[\widetilde{i} \colon \mathfrak{D}_{\mathfrak{sl}_{3}} \to \mathcal{D}'_{\mathfrak{sl}_{3}}\] defined by
\begin{align*}
  &E_{1} \mapsto i (Y_{8} +Y_{8}Y_{6} + Y_{8} Y_{6} Y_{7} + Y_{8} Y_{6} Y_{7}Y_{9}) +O(\epsilon),  && K_{1} \mapsto  Y_{4} Y_{6} Y_{7} Y_{8} Y_{9} + O(\epsilon), \\
  &E_{2} \mapsto i (Y_{3} + Y_{3}Y_{10} + Y_{3} Y_{10} Y_{6} + Y_{3}Y_{10}Y_{6} Y_{5}) + O(\epsilon), && K_{2} \mapsto Y_{1}Y_{3}Y_{5} Y_{6} Y_{10} + O(\epsilon), \\
  &F_{1} \mapsto i (Y_{4} + Y_{4}Y_{5}) + O(\epsilon), && \widetilde{K}_{1} \mapsto  Y_{4}Y_{5}Y_{8}+ O(\epsilon), \\
  &F_{2} \mapsto i (Y_{1} + Y_{1}Y_{2} + Y_{1} Y_{2} Y_{6} + Y_{1}Y_{2}Y_{6} Y_{7}) + O(\epsilon), && \widetilde{K}_{2} \mapsto  Y_1 Y_{2} Y_{3} Y_{6}Y_{7} + O(\epsilon).
\end{align*}
To obtain the expression above, we have only used classical mutations instead of quantum mutations. This is why we defined the map $\widetilde{i}$ up to zeroth order in $\epsilon$. We conjecture that the quantum mutations lead to similar expressions, with extra factors of $q$.

Consider the algebra homomorphism $\widetilde{\psi} \colon  \mathcal{D}'_{\mathfrak{sl}_{3}}\to A_{3}(\mathbb{C}_{\epsilon}(t_1,t_2))[\mathbf{x}_1^{-1},\mathbf{x}_2^{-1},\mathbf{x}_3^{-1}]$ defined by
  \begin{align*}
    & Y_1 \mapsto (-i)\mathbf{x}_{2}^{-1}\mathbf{x}_{1}q^{\mathbf{x}_{2}\mathbf{p}_{2}}, && Y_2 \mapsto -q^{-2\mathbf{x}_{2} \mathbf{p}_{2}}, \\
    & Y_3 \mapsto (-i) T_{2} \mathbf{x}_{3} q^{- \mathbf{x}_{3}\mathbf{p}_{3}}, && Y_4 \mapsto (-i) \mathbf{x}_{1}^{-1} q^{\mathbf{x}_{1}\mathbf{p}_{1}}, \\
    & Y_5 \mapsto -q^{-2 \mathbf{x}_{1}\mathbf{p}_{1}}, && Y_6 \mapsto -\mathbf{x}_{3}^{-1} \mathbf{x}_{1}^{-1} \mathbf{x}_{2}q^{\mathbf{x}_{1}\mathbf{p}_{1}}q^{\mathbf{x}_{3}\mathbf{p}_{3}}, \\
    & Y_7 \mapsto -q^{-2 \mathbf{x}_{3}\mathbf{p}_{3}}, && Y_8 \mapsto (-i) T_{1}\mathbf{x}_{1}  q^{-\mathbf{x}_{1} \mathbf{p}_{1} - \mathbf{x}_{2}\mathbf{p}_{2}+ \mathbf{x}_{3}\mathbf{p}_{3}}, \\
    & Y_9 \mapsto -T^{-2}_{1}\mathbf{x}_{1}\mathbf{x}_{2}^{-1} \mathbf{x}_{3} q^{2 \mathbf{x}_{2}\mathbf{p}_{2}} q^{- \mathbf{x}_{3}\mathbf{p}_{3}}q^{\mathbf{x}_{1}\mathbf{p}_{1}}, && Y_{10} \mapsto -T_{2}^{-2} q^{2 \mathbf{x}_{3}\mathbf{p}_{3}},
  \end{align*}
  with $T_i = e^{h t_i}$. Based on the map $\widetilde{i} \colon \mathfrak{D}_{\mathfrak{sl}_{3}} \to \mathcal{D}'_{\mathfrak{sl}_{3}}$, we conjecture the following.
\begin{conjecture}
\label{conj:sl3}
There exist integers $(k_i)_{i=1}^{10}$ such that the rescaled map
\[
\widetilde{\psi}^{(k)} \colon \mathcal{D}'_{\mathfrak{sl}_3}\longrightarrow A_3(\mathbb{C}_\epsilon (t_1,t_2) )
\]
defined on generators by
\[
\widetilde{\psi}^{(k)}(Y_i)\ :=\ q^{k_i}\widetilde{\psi}(Y_i),\qquad i=1,\dots,10,
\]
induces an algebra homomorphism
\[
g\ :=\ \widetilde{\psi}^{(k)}\circ \widetilde{i}\ \colon\ \mathfrak{D}_{\mathfrak{sl}_3} \longrightarrow A_3(\mathbb{C}_\epsilon(t_1,t_2)),
\]
whose action on the generators is given by
  \begin{align*}
    K_{1} &\mapsto  T_{1}^{-1} q^{2 \mathbf{x}_{1}\mathbf{p}_{1} + \mathbf{x}_{2}\mathbf{p}_{2} - \mathbf{x}_{3}\mathbf{p}_{3}}, \\
    K_{2} &\mapsto T^{-1}_{2} q^{-\mathbf{x}_{1} \mathbf{p}_{1} + \mathbf{x}_{2}\mathbf{p}_{2} +2 \mathbf{x}_{3} \mathbf{p}_{3}}, \\
    F_{1} &\mapsto \mathbf{x}_{1}^{-1} (q^{\mathbf{x}_{1}\mathbf{p}_{1}} - q^{-\mathbf{x}_{1} \mathbf{p}_{1}}), \\
     F_{2} &\mapsto \mathbf{x}_{2}^{-1}\mathbf{x}_{1} (q^{\mathbf{x}_{2} \mathbf{p}_{2}} - q^{- \mathbf{x}_{2} \mathbf{p}_{2}}) + \mathbf{x}_{3}^{-1} q^{ \mathbf{x}_{1} \mathbf{p}_{1} -\mathbf{x}_{2} \mathbf{p}_{2}}(q^{\mathbf{x}_{3} \mathbf{p}_{3} }- q^{-\mathbf{x}_{3} \mathbf{p}_{3} }), \\
     E_{1} &\mapsto \mathbf{x}_{1}  (T_{1}q^{-\mathbf{x}_{1} \mathbf{p}_{1} - \mathbf{x}_{2}\mathbf{p}_{2}+ \mathbf{x}_{3}\mathbf{p}_{3}} - T^{-1}_{1}q^{\mathbf{x}_{1} \mathbf{p}_{1} + \mathbf{x}_{2}\mathbf{p}_{2}- \mathbf{x}_{3}\mathbf{p}_{3}}),  \\
    &\quad - T_{1} \mathbf{x}_{3}^{-1} \mathbf{x}_{2} q^{-\mathbf{x}_{2} \mathbf{p}_{2}} q^{\mathbf{x}_{3}\mathbf{p}_{3}} (q^{\mathbf{x}_{3}\mathbf{p}_{3}} - q^{-\mathbf{x}_{3}\mathbf{p}_{3}}), \\
     E_{2} &\mapsto  \mathbf{x}_{3}(T_2q^{-\mathbf{x}_{3}\mathbf{p}_{3}} -T_{2}^{-1} q^{\mathbf{x}_{3}\mathbf{p}_{3}}) + \mathbf{x}_{1}^{-1}\mathbf{x}_{2} T_{2}^{-1}q^{2 \mathbf{x}_{3}\mathbf{p}_{3}} (q^{\mathbf{x}_{1}\mathbf{p}_{1}} - q^{-\mathbf{x}_{1}\mathbf{p}_{1}}),
  \end{align*}
with $g(K_{i}) = g(\widetilde{K}_{i})^{-1}$ for $i \in \{1,2\}$ and $T_i = e^{ht_i}$.
  \end{conjecture}
  \begin{remark}
    A related map can be derived by considering the maps $\pi_{\lambda} \colon U_{q}(\mathfrak{sl}_{3}) \to \mathcal{H}_{n}$ defined in  \cite{awata1994heisenberg}.
\end{remark}
Note that the map $\widetilde{i} \colon \mathfrak{D}_{\mathfrak{sl}_{3}} \to \mathcal{D}'_{\mathfrak{sl}_{3}}$ was defined modulo $\epsilon$. A more complete description of the map requires additional factors of $q = e^{h \epsilon}$. These unknown factors of $q$ have been absorbed in the factors $q^{k_{i}}$ in the definition of $g^{(k)}$.

Consider the $U_{q}(\mathfrak{sl}_{3})$ universal R-matrix as given in \eqref{eq:sl3rmatrix}. Using the map $g$ as in Conjecture~\ref{conj:sl3}, we see that the expansion $(g\otimes g)(\mathcal R_{\mathfrak{sl}_3})$ in $\epsilon$ takes on an especially useful form. We employ the same strategy as in Section \ref{sec:quantumclusterper}. Upon setting $T_{1} = e^{h t_{1}}$ and $T_{2} = e^{h t_{2}}$, we find
\begin{gather*}
  \begin{aligned}
    &g((q-q^{-1}) f_{1}) = \mathbf{p}_{1} + O(\epsilon), &\qquad
    &g((q-q^{-1}) f_{2}) = \mathbf{x}_{1}\mathbf{p}_{2} + \mathbf{p}_{3} + O(\epsilon), \\
    &g(e_{1}) = \mathbf{x}_{1}(T_{1} - T_{1}^{-1}) + O(\epsilon), &\qquad
    &g(e_{2}) = \mathbf{x}_{3}(T_{2} - T_{2}^{-1}) + O(\epsilon),
  \end{aligned} \\
  (g \otimes g)(q^{h_{1} \otimes h_{1}}) = e^{h \epsilon^{-1} t_{1}^{2}} e^{-h t_{1}((2\mathbf{x}_{1}\mathbf{p}_{1} + \mathbf{x}_{2}\mathbf{p}_{2} - \mathbf{x}_{3}\mathbf{p}_{3}) \otimes 1 + 1 \otimes (2\mathbf{x}_{1}\mathbf{p}_{1} + \mathbf{x}_{2}\mathbf{p}_{2} - \mathbf{x}_{3}\mathbf{p}_{3}))} + O(\epsilon)
\end{gather*}
where $(g \otimes g)(q^{(A^{-1})_{ij} h_{i} \otimes h_{j}})$ for $i,j \in \{1,2 \}$ can be computed in a way similar to that in the last line.
Lastly, we consider the factor
\[
  (q-q^{-1}) e_{12} \otimes f_{12} = (q-q^{-1})^{-1} e_{12} \otimes (q-q^{-1})^{2}f_{12}
  \]
  that is part of the universal $R$-matrix of $U_{q}(\mathfrak{sl}_{3})$. We deduce
  \begin{align*}
    &(q-q^{-1})^{-1} = \epsilon^{-1} + \dots \\
    &g(e_{12}) = O(\epsilon)\\
    &(q-q^{-1})^{2}g(f_{12}) = (q-q^{-1})^{2} g(f_{2} f_{1} - q f_{1}f_{2}) = -\mathbf{p}_{2} +O(\epsilon).
    \end{align*}
    Thus, we conclude that $g((q-q^{-1}) e_{12} \otimes f_{12})$ contains no $\epsilon^{-1}$ terms.

    Observe that in all of the expressions above, no inverses of Weyl--Heisenberg generators appear inside the $\epsilon$ expansion. Hence, $(g\otimes g)(\mathcal{R}_{\mathfrak{sl}_3})$ can be expanded in orders of $\epsilon$, where one will only encounter (exponentials of) polynomials in $x_{i}$ and $p_{i}$. The expression $(g\otimes g)(\mathcal{R})$ can be brought into a normal ordered form similar to that of Theorem \ref{thm:pertrmatrix} by applying the BCH formula several times. With a rescaling $\epsilon \mapsto \epsilon^2$ and $x_i \mapsto \epsilon x_i$ and $p_i \mapsto \epsilon^{-1} p_i$, we conjecture that this leads to the following.
    \begin{conjecture}
      \label{conj:Rsl3}
Define
\[
R _{0}^{(i,j,k)} = \varphi^{6}_{j,k} \left(
\begin{pmatrix}
                                         S_i & 0 \\
                                         1-S_i^{2} & S
\end{pmatrix} \right) \in A_3(\mathbb{C}) \otimes A_3(\mathbb{C}) \qquad (i \in \{1,2,3\})
\]
where $S_1 = T_1$, $S_2 = T_2$ and $S_3 = T_1 T_2$ and $\varphi^{a}_{b,c}$ defined in Notation \ref{not:tensorthings}. Then 
    \[
      ( g \otimes g)(\mathcal{R}_{\mathfrak{sl}_3}) \sim R_{0}^{(1,1,2)} \cdot R_{0}^{(2,3,4)} \cdot R_{0}^{(3,5,6)}(1 + \epsilon \mathcal{N}(R_1(T_1,T_2,x,p)))
    \]
    for some function $R_1$. With the contraction formalism described in \cite{bar2021perturbed,bosch2025tensorsgaussiansalexanderpolynomial} (or the Mathematica implementation in Appendix~\ref{sec:mathematica}), this then gives rise to a knot invariant, which we denote by $\rho_{\mathfrak{sl}_3}(T_1,T_2)$. A closed formula can be obtained by using \cite[Proposition A.1]{bosch2025tensorsgaussiansalexanderpolynomial}.
\end{conjecture}

It is worth pointing out that the quantum torus algebra $\mathcal{D}'_{\mathfrak{sl}_{3}}$ is the only quantum torus algebra related to $\mathfrak{sl}_3$ that we have found for which an expansion of the $R$-matrix (after applying a homomorphism) can be derived that does not involve inverses of generators of the Weyl--Heisenberg algebra. This is crucial, since if these elements appeared in the expansion of $\epsilon$, any computation with the perturbed $R$-matrix would be orders of magnitude more complex. 
    
\section{Discussion}
\label{sec:discussionthird}
\noindent
The approach to perturbative expansions of knot invariants via cluster algebras has three main outcomes:
\begin{enumerate}[i)]
  \item It provides a method to derive the Alexander polynomial using only classical cluster algebras (Theorem \ref{thm:clusterburau}).
  \item Using the representation theory of quantum cluster algebras, one can construct a perturbative representation, which in turn produces an $\epsilon$-expanded $R$-matrix (Theorem \ref{thm:pertrmatrix}).
  \item It offers a promising route toward a universal strategy for deriving perturbative invariants associated with ADE-type Lie algebras (Section \ref{sec:towsl3}).
\end{enumerate}

Perturbed Alexander invariants still lack a topological interpretation. One advantage of the cluster-theoretic approach described here is that it includes a geometric interpretation. We therefore recall the geometric realization of cluster mutations as flips of ideal triangulations and gluings of ideal hyperbolic tetrahedra.

\subsection{Gluing Tetrahedra}

The Alexander polynomial is intrinsically a topological invariant, defined from the homology of the infinite cyclic cover of the knot complement. Quantum invariants, by contrast, are commonly expressed in a highly algebraic context, with their construction relying on a choice of Lie algebra. While this yields powerful invariants, we prefer to reduce reliance on Lie algebras. Since the Alexander polynomial is intrinsically topological, this suggests that perturbed Alexander invariants also have a topological origin. To connect the cluster-algebraic construction with topology, we review its geometric interpretation based on \cite{Fomin2008,hikami2015braids}.

Let $\mathcal{K}$ be a knot embedded in a 3-manifold $M$. In our setting, we consider $M = S^3$. We study the complement $M \setminus \mathcal{K}$. This manifold admits a decomposition into two pieces along a surface $\Sigma$. In the case $M = S^3$, the surface $\Sigma$ can be taken to be an $n$-punctured disk, where $n$ depends on the decomposition. In Figure \ref{fig:beauty}, $\Sigma$ is a disk with two punctures.

\begin{figure}[H]
\centering
\begin{tikzpicture}
\node[anchor=south west, inner sep=0] (img) at (0,0)
{\includegraphics[scale=0.9]{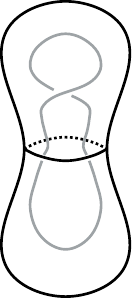}};

\begin{scope}[x={(img.south east)}, y={(img.north west)}]

\node at ($(1,0.5)$) {$\Sigma$};
\node at ($(-0.1,1)$) {$M$};

\end{scope}
\end{tikzpicture}
\caption{A manifold $M$ is being cut into two pieces}
\label{fig:beauty}
\end{figure}

By choosing a finite set of marked points on the boundary of $\Sigma$, we obtain an ideal triangulation of $\Sigma$. An example is shown in Figure~\ref{fig:triangulationintro}. Equivalently, we can obtain this triangulation by gluing two copies of a triangulated once-punctured disk along a boundary. This construction also applies to the $n$-punctured disk.

\begin{figure}[H]
\centering
\begin{tikzpicture}[scale=0.75, every path/.style={line width=1pt}]
\node[circle, draw, fill=black] (0) at (0, 2) {};
\node[circle, draw, fill=black] (1) at (0, -2) {};
\node (4) at (-2, 0) {};
\node (5) at (2, 0) {};
\node[circle, draw, fill=white] (2) at (-1, 0) {};
\node[circle, draw, fill=white] (3) at (1, 0) {};
\draw (0) to (1);
\node (16) at (2,1.5) {$\Sigma$};
\draw (1) to (2);
\draw (2) to (0);
\draw (0) to (3);
\draw (3) to (1);
\draw [in=180, out=-90] (4.center) to (1.center);
\draw [in=-180, out=90] (4.center) to (0.center);
\draw [in=90, out=0] (0.center) to (5.center);
\draw [in=0, out=-90] (5.center) to (1.center);
\end{tikzpicture}
\caption{A triangulation of a 2-punctured disk}
\label{fig:triangulationintro}
\end{figure}

A half-Dehn twist on this surface can be understood via its action on ideal triangulations, through a finite sequence of ``flips'', as explained in Section \ref{sec:braidingoperator}. The two-dimensional flip admits a three-dimensional realization. As explained in \cite{Fomin2008,hikami2015braids}, the flip can be seen as the attachment of an ideal hyperbolic tetrahedron along the quadrilateral. Hyperbolic geometry now gives rise to algebra, as we describe below. An ideal tetrahedron $\Delta\subset \mathds H^3$ with vertices $(v_0,v_1,v_2,v_3)\subset \partial\mathds H^3\simeq \widehat{\mathds C}$ is defined (up to orientation-preserving isometry) by its shape parameter $z\in\mathds{C} \setminus \{0,1 \}$, given by the cross-ratio
\[
z=[v_0,v_1;v_2,v_3]=\frac{(v_0-v_2)(v_1-v_3)}{(v_0-v_3)(v_1-v_2)}.
\]
The same tetrahedron also carries two other shape parameters $z’=1-z^{-1}$ and $z’'=(1-z)^{-1}$, corresponding to the three choices of opposite edge-pairs. See Figure~\ref{fig:tetrahedron}.
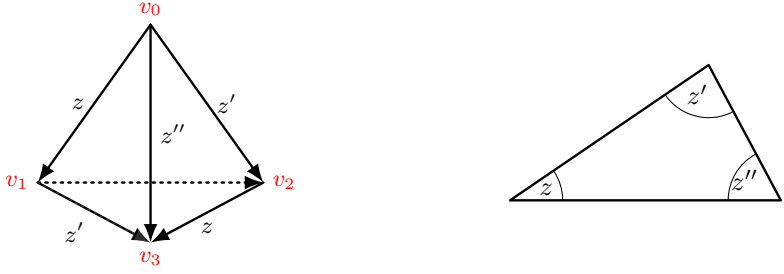
\begin{figure}[H]
\centering
\begin{minipage}{0.48\textwidth}
  \centering
\begin{tikzpicture}[scale=0.4, line cap=round, line join=round]
  \coordinate (v0) at (0,4.2);
  \coordinate (v1) at (-3,0);
  \coordinate (v2) at (3,0);
  \coordinate (v3) at (0,-1.6);

  \tikzset{
    edge/.style={line width=0.9pt},
    dottededge/.style={edge, dotted},
    dashededge/.style={edge, dashed},
    arr/.style={-Latex},
    arrsmall/.style={-Latex, line width=0.9pt},
    vlabel/.style={red, font=\small},
    elabel/.style={font=\small},
  }

  \draw[edge, arr] (v0) -- (v1)
    node[elabel, midway, left] {$z$};

  \draw[edge, arr] (v0) -- (v2)
    node[elabel, midway, right] {$z'$};

  \draw[edge, arr] (v1) -- (v3)
    node[elabel, midway, below left] {$z'$};

  \draw[dottededge,arr] (v1) -- (v2);

  \draw[edge, arr] (v2) -- (v3)
    node[elabel, midway, below] {$z$};

  \draw[edge, arr] (v0) -- (v3)
    node[elabel, midway, right] {$z''$};


  \draw[dotted, arr, gray, opacity=0.35, line width=0.7pt] (v1) -- (v3);

  \node[vlabel, above]       at (v0) {$v_0$};
  \node[vlabel, left]        at (v1) {$v_1$};
  \node[vlabel, right]       at (v2) {$v_2$};
  \node[vlabel, below]       at (v3) {$v_3$};

  \end{tikzpicture}
\end{minipage}\hfill
\begin{minipage}{0.48\textwidth}
  \centering
\begin{tikzpicture}[scale=0.5, line cap=round, line join=round]
  \coordinate (A) at (0,0);     
  \coordinate (B) at (6,0);     
  \coordinate (C) at (4.4,3.0); 

  \draw[line width=0.9pt] (A) -- (B) -- (C) -- cycle;

  \usetikzlibrary{angles,quotes}

  \pic[
    draw,
    angle radius=7mm,
    angle eccentricity=1.2,
  ] {angle = B--A--C};
  \node at ($(A)+(0.80,0.25)$) {$z$};

  \pic[
    draw,
    angle radius=7mm,
    angle eccentricity=1.2,
  ] {angle = A--C--B};
  \node at ($(C)+(-0.25,-0.65)$) {$z'$};

  \pic[
    draw,
    angle radius=7mm,
    angle eccentricity=1.2,
  ] {angle = C--B--A};
  \node at ($(B)+(-0.8,0.45)$) {$z''$};

\end{tikzpicture}
\end{minipage}
\caption{An oriented ideal tetrahedron (left) and the triangle obtained by intersecting it with a horosphere (right). Adapted from \cite{hikami2015braids}.}
\label{fig:tetrahedron}
\end{figure}
The flip is interpreted as gluing a hyperbolic ideal tetrahedron to a triangulation of a punctured disk, as in Figure \ref{fig:gluing}. Each dihedral angle on the triangulated surface after the gluing is

\[
\begin{cases}
\tilde z_1 = z_1 z', \\
\tilde z_2 = z_2 z'', \\
\tilde z_3 = z, \\
\tilde z_4 = z_4 z'', \\
\tilde z_5 = z_5 z'.
\end{cases}
\]
with consistency condition $z_3 z = 1$. When one sets
\[
z_k = -y_k^{-1},
\]
one obtains
\[
\begin{cases}
\tilde{y}_1 = y_1 (1+y_3), \\
\tilde{y}_2 = y_2(1+y_3^{-1})^{-1}, \\
\tilde{y}_3 = y_3^{-1}, \\
\tilde{y}_4 = y_4(1+y_3^{-1})^{-1}, \\
\tilde{y}_5 = y_5 (1+y_3),
\end{cases}
\]
which are related to the mutation of cluster $\mathcal{X}$-variables as in~\eqref{eq:clasmut}. Consequently, we may interpret a mutation as the operation of gluing an ideal tetrahedron
(with shape parameter $z = -y_3^{-1}$) to the punctured surface.
\begin{figure}[H]
\centering
\begin{tikzpicture}
\node[anchor=south west, inner sep=0] (img) at (0,0)
{\includegraphics[width=0.6\linewidth]{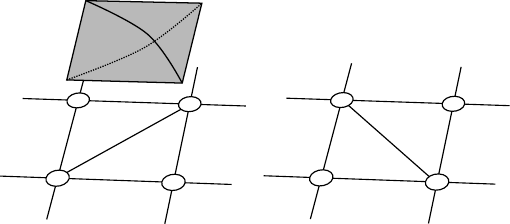}};

\begin{scope}[x={(img.south east)}, y={(img.north west)}]

\node at ($(0.1,0.4)$) {$z_1$};
\node at ($(0.23,0.15)$) {$z_4$};
\node at ($(0.39,0.38)$) {$z_5$};
\node at ($(0.25,0.5)$) {$z_2$};
\node at ($(0.20,0.38)$) {$z_3$};
\node at ($(0.12,0.83)$) {$z'$};
\node at ($(0.26,0.83)$) {$z$};
\node at ($(0.28,1.05)$) {$z''$};

\node at ($(0.1+0.52,0.4)$) {$\tilde{z}_1$};
\node at ($(0.23+0.52,0.15)$) {$\tilde{z}_4$};
\node at ($(0.39+0.52,0.38)$) {$\tilde{z}_5$};
\node at ($(0.25+0.52,0.5)$) {$\tilde{z}_2$};
\node at ($(0.20+0.52,0.38)$) {$z$};

\end{scope}
\end{tikzpicture}
\caption{Interpretation of the flip. Gluing an ideal tetrahedron onto a triangulated surface (left) produces a triangulated surface with the diagonal flipped (right). Adapted from \cite{hikami2014braiding}.}
\label{fig:gluing}
\end{figure}

From a three-dimensional perspective, each flip can be interpreted as the addition of an ideal tetrahedron. A sequence of flips therefore corresponds to a gluing of ideal tetrahedra. In this way, the braid word determines an ideal triangulation of the knot complement $S^3 \setminus \mathcal K$, where the tetrahedra are glued in accordance with its corresponding mutation sequence.

\subsection{The Neumann--Zagier Matrices}

In Section~\ref{sec:alexcluster}, we showed that for a given braid $\beta$, imposing the condition $\boldsymbol{x}[1]=\boldsymbol{x}[m+1]$ on the associated cluster pattern yields, for particular evaluations of the cluster variables, a knot invariant of the closure of $\beta$. Hence, if one seeks to use the theory of cluster algebras to derive knot invariants, then this leads one to consider fixed points of cluster patterns. A related structure is the Neumann--Zagier matrices.  In \cite{Mizuno2020Jacobian}, the authors show that the Neumann--Zagier matrices describe the fixed-point equations of such cluster transformations. In more recent work, these same Neumann--Zagier matrices were associated with ordered ideal triangulations of a knot complement, allowing a derivation of the Alexander polynomial \cite{GaroufalidisYoon2025TwistedAlexander}. The authors of the latter paper relied on a relationship between Neumann--Zagier matrices and Fox calculus. For future research, this correspondence could potentially be extended to higher order in $\epsilon$ to obtain a topological interpretation of the perturbed Alexander invariants.

We emphasize that the current construction of the perturbed Alexander invariants still strongly relies on the $R$-matrix formalism. As a future direction, it might be interesting to reduce this dependence by expressing the perturbed invariants purely in terms of triangulations. This viewpoint is precisely what makes a cluster-algebraic interpretation promising. Moreover, it may prove fruitful in the study of fibered knots, as knots of this type naturally arise with a surface on which mapping-class-group actions can be performed.

\subsection{Generalizations and the Weyl-Group Action}

In \cite[Section 13.3]{goncharov2019quantum}, Goncharov and Shen showed that the cluster varieties under consideration come equipped with a Weyl-group action, which they interpret within the context of (per)mutations. In our setting, this suggests that the palindromicity of the Alexander polynomial and the perturbed Alexander invariants may reflect a symmetry under the Weyl group $\mathbb{Z}/2\mathbb{Z}$. This observation has a geometric interpretation in terms of flips in the triangulation of the knot complement. This idea may extend to the (perturbed) invariants associated with $\mathfrak{sl}_3$, where symmetries could arise from the Weyl-group action of $S_3$, which likewise admits a geometric description.

This latter observation connects well with recent work \cite{barnatan2026faststrongtopologicallymeaningful}, in which a knot invariant depending on two variables $T_1$ and $T_2$ was introduced and conjectured to be equivalent to the perturbed Alexander invariant arising from $U_q(\mathfrak{sl}_3)$. This new knot invariant, $\theta_{\mathcal{K}}(T_1,T_2)$, can be represented as a 2D array that clearly shows its symmetries. Notably, $\theta_{\mathcal{K}}(T_1,T_2)$ admits a hexagonal symmetry; that is, for any knot $\mathcal{K}$, $\theta_{\mathcal{K}}(T_1,T_2)$ is invariant under the substitutions $(T_1 \to T_1, T_2 \to T_1^{-1}T_2^{-1})$ and $(T_1 \to T_1 T_2, T_2 \mapsto T_2^{-1})$. We conjecture that $\theta_{\mathcal{K}}(T_1,T_2)$ is equivalent to the knot invariant $\rho_{\mathfrak{sl}_3}(T_1,T_2)$ introduced in Conjecture~\ref{conj:Rsl3}.

  Since $\rho_{\mathfrak{sl}_3}(T_1,T_2)$ has an origin in the representation theory of quantum groups and cluster algebras, the hexagonal symmetry of $\theta_{\mathcal{K}}(T_1,T_2)$ can be explained as a symmetry under a Weyl group action. Let us go into more detail. Note that $\rho_{\mathfrak{sl}_3}(T_1,T_2)$ is defined through the map $g$ which has an origin in representation theory. We denote by $\mathcal{P}_{(t_1,t_2)}$ the representation that induces the map $g$. By~\cite[Section 13.3]{goncharov2019quantum}, this representation ace has the property that
  \begin{equation}
  \label{eq:equivalence}
    \mathcal{P}_{(t_1,t_2)} \simeq \mathcal{P}_{w(t_1,t_2)}
  \end{equation}
  are unitarily equivalent for any Weyl group element $w \in W$ acting on $t$. This observation is consistent with the fact that $\mathcal{P}_{(t_1,t_2)}$ is equivalent to the positive principal series representation as alluded to in Remark~\ref{rem:positiverep}, for which the same property holds as shown in \cite[Theorem 9.1]{Ip2020PositiveRepresentations}.
Since the universal invariant of a long knot associated with \(U_q(\mathfrak{g})\) lies in the center of the algebra, any equivalence of the form \eqref{eq:equivalence} manifests itself when the invariant is evaluated in a representation $\mathcal{P}_{(t_1,t_2)}$

  This observation is not isolated to the case $\mathfrak{g} = \mathfrak{sl}_3$. This leads to the following conjecture.

\begin{conjecture}
  Let $\mathfrak{g}$ be a Lie algebra of type ADE with rank $r$. The cluster-algebraic intepretation of quantum groups gives rise to a representation $\mathcal{P}_{t}$ of $U_q(\mathfrak{g})$ with parameter $t = (t_i)_{i = 1}^r$, which, after mapping $h\mapsto \epsilon^2 h$ and $t_i \mapsto \epsilon^{-2}t_i$, induces a map $g_{\mathfrak{g}} \colon U_q(\mathfrak{g}) \to A_n(\mathbb{C}_{\epsilon}(t))$ for some $n \in \mathbb{N}$, such that (after mapping $x_i \mapsto \epsilon x_i$ and $p_i \mapsto \epsilon^{-1} p_i$) the universal $R$-matrix $\mathcal{R}_{\mathfrak{g}}$ can be expressed as a perturbed Gaussian, that is
\[
(g_{\mathfrak{g}}  \otimes g_{\mathfrak{g}} )(\mathcal{R}_{\mathfrak{g}}) \simeq R_0(1 + \epsilon \mathcal{N}(R_1(T_1,\dots,T_r,x,p)) + \dots) \qquad (T_i = e^{h t_i})
\]
where $R_0 = \varphi(\dots)$ is a Gaussian of the type introduced in Notation~\ref{def:monhom}. For any order of $\epsilon$, this induces a knot invariant which has symmetries in its input variables $(T_i)_{i = 1}^r$ that arise from the equivalence
  \begin{equation}
  \label{eq:equivalence}
    \mathcal{P}_{t} \simeq \mathcal{P}_{w(t)}
  \end{equation}
for any Weyl group element $w \in W$ acting on $t$.
\end{conjecture}
In future work, we will explore this conjecture in more detail.

\appendix
\section{Mathematica implementation}
\label{sec:mathematica}
\noindent
To implement noncommutativity in \textit{Mathematica}, which is essential for the definition of quantum mutations, we make use of the \textit{NCAlgebra} package \cite{HeltonDeOliveira_NCAlgebra}. There, variables can be explicitly declared as noncommuting, and products between such variables are written using the \texttt{**} operator, so that, for instance, \texttt{a ** b} remains distinct from \texttt{b ** a}.  Moreover, the inverse of a non-commutative variable is denoted using \texttt{inv[x]}.

To compute the perturbed R-matrix, we need the following. Recall that $\mathcal{R}$ denotes the universal $R$-matrix of $U_{h}(\mathfrak{sl}_{2})$. Moreover, we defined $R = (\psi \circ \iota \otimes \psi \circ \iota)(\mathcal{R})$ and $R_{0} = R |_{\epsilon = 0}$. Let us assume that
\[
R = R_{0}e^{\epsilon \mathcal{N}(f_{1}(x,p) + f_{2}(x,p) \epsilon)} \mod \epsilon^{3}.
\]
The goal will be to find explicitly the functions $f_{1}(x,p)$ and $f_{2}(x,p)$. With the use of quantum mutations, it is possible to compute
\begin{align*}
F_{1} :=   P \circ \mathrm{Ad}_{R} x_{1} \quad \text{ and } \quad F_{2} := P \circ \mathrm{Ad}_{R} x_{2}.
\end{align*}
Using Proposition \ref{pro:perturbedexptill2}, we can explicitly compute the functions $f_{1}(x, p)$ and $f_{2}(x, p)$. In the following, we detail the procedure used to compute $F_{1}$ and $F_2$.

The expression \texttt{F}$_{1}$ computes
\[
  \mathbf{x}_{1} + \epsilon \mathcal{N}(\partial_{p_{1}} f_{1}(x,p)) + \epsilon^{2} \mathcal{N}(f_{2}(x,p)) + \frac{\epsilon^{2}}{2}[\mathcal{N}(f_{1}(x,p),\partial_{p_{1}} f_{1}(x,p)]
  \]
  while the expression \texttt{F}$_{2}$ computes
\[
  \mathbf{x}_{2} + \epsilon \mathcal{N}(\partial_{p_{2}} f_{1}(x,p)) + \epsilon^{2} \mathcal{N}(f_{2}(x,p)) + \frac{\epsilon^{2}}{2}[\mathcal{N}(f_{1}(x,p),\partial_{p_{2}} f_{1}(x,p)],
  \]
from which the functions $f_{1}(x,p)$ and $f_{2}(x,p)$ can be extracted, as in the Mathematica code below. The two expressions $f_{1}$ and $f_{2}$ are then brought together in the expressions \texttt{Pert$_{i,j}$} and \texttt{InvPert$_{i,j}$} which make up the perturbative parts of $R$ and $R^{-1}$. To confirm the XC-algebra relations, we use the contraction formalism and its Mathematica implementation, as provided in \cite{bar2021perturbed}.

The following implementation is publicly available at \cite{bosch2026perturbative}.
\includepdf[pages={1-9}, scale=0.85]{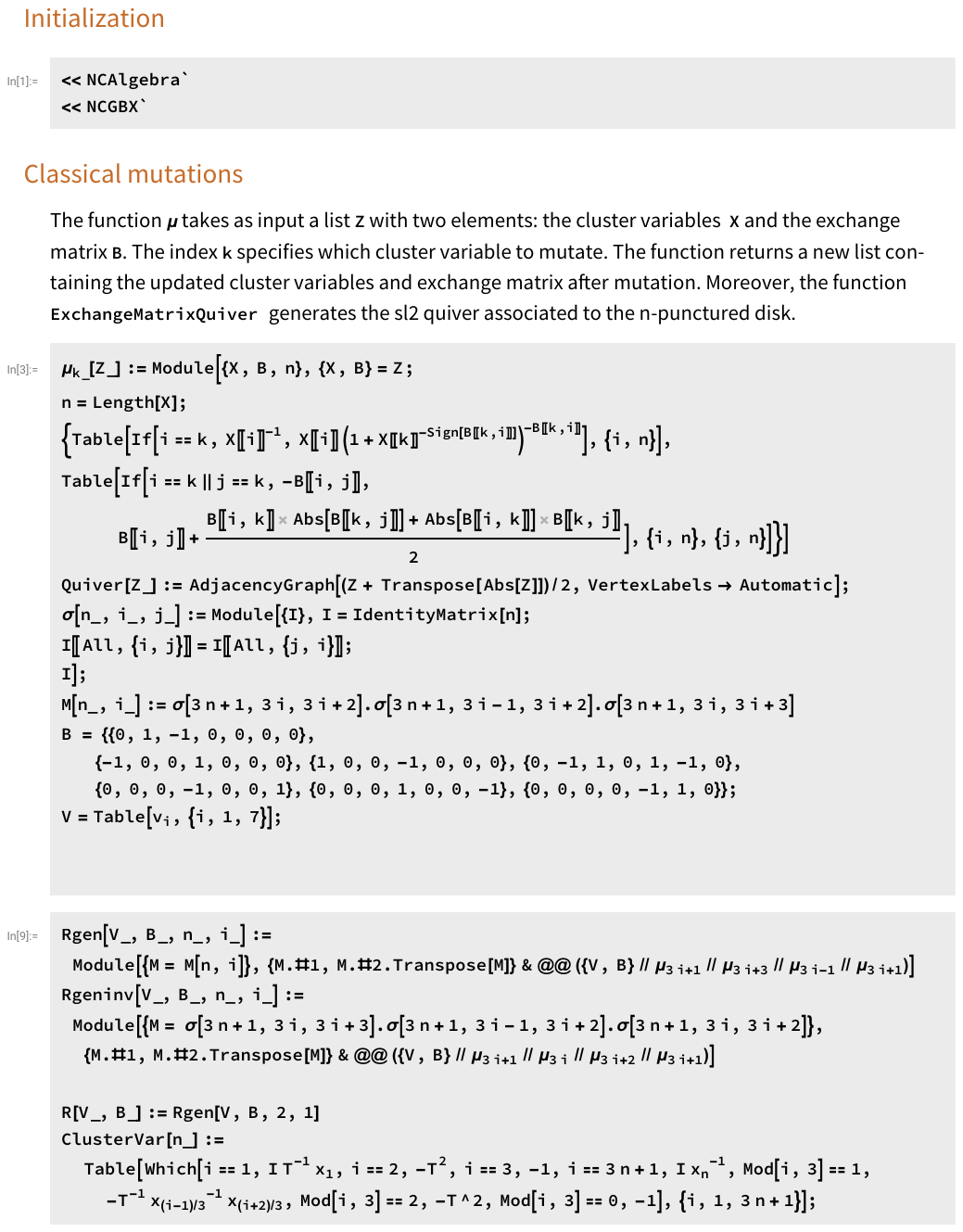}

\bibliographystyle{halpha-abbrv}
\bibliography{bibliografia}
\typeout{get arXiv to do 4 passes: Label(s) may have changed. Rerun}
\end{document}